\documentclass[11pt,reqno]{amsart}
\usepackage{amsmath,amssymb,graphicx,mathrsfs,amsopn,stmaryrd,amsthm,amscd,color,bm,extarrows}
\usepackage{newtxtext}
\usepackage{float}
\usepackage{ytableau,textcase}
\usepackage[colorlinks=true,citecolor=blue,linkcolor=blue
]{hyperref}
\usepackage[alphabetic,nobysame]{amsrefs}
\renewcommand{\MR}[1]{} \renewcommand{\PrintDOI}[1]{}
\newcommand{\arxiv}[1]{\href{http://arxiv.org/abs/#1}{arXiv:\nolinkurl{#1}}}
\usepackage[english]{babel}
\usepackage[mathscr]{euscript}
\usepackage[left=2.5cm,
            right=2.5cm,
            top=3cm,
            bottom=2.75cm
            ]{geometry}
\usepackage{collectbox}


\usepackage{dsfont}
\usepackage{tikz}    
\usetikzlibrary{arrows,matrix,decorations.markings,shapes.geometric}   
\usetikzlibrary{decorations.pathreplacing}

\usepackage{xy}
\allowdisplaybreaks
\xyoption{all}
\numberwithin{equation}{section}
\newtheorem{thm}{Theorem}[section]
\newtheorem{prop}[thm]{Proposition}
\newtheorem{lem}[thm]{Lemma}
\newtheorem{cor}[thm]{Corollary}

\theoremstyle{definition} 
\newtheorem{eg}[thm]{Example}
\newtheorem{dfn}[thm]{Definition}
\theoremstyle{remark}
\newtheorem{rem}[thm]{Remark}

\newcommand{\beq}{\begin{equation}}
\newcommand{\eeq}{\end{equation}}
\newcommand{\be}{\begin{equation*}}
\newcommand{\ee}{\end{equation*}}

\newcommand{\bC}{\mathbb{C}}

\newcommand{\bZ}{\mathbb{Z}}

\newcommand{\bN}{\mathbb{N}}

\newcommand{\mc}{\mathcal}
\newcommand{\cD}{\mathcal{D}}

\newcommand{\sfh}{\mathsf{h}}
\newcommand{\sfb}{\mathsf{b}}

\newcommand{\hf}{\tfrac12}

\newcommand{\bb}{\mathbf{b}}
\newcommand{\bh}{\mathbf{h}}

\newcommand{\g}{\mathfrak{g}}
\newcommand{\h}{\mathfrak{h}}

\newcommand{\gl}{\mathfrak{gl}}

\newcommand{\fksl}{\mathfrak{sl}}
\newcommand{\fkS}{\mathfrak{S}}

\newcommand{\id}{{\mathrm{id}}}   
   
\newcommand{\gr}{{\mathrm{gr}}}

\newcommand{\tl}{\tilde}
\newcommand{\wtl}{\widetilde}
\newcommand{\gge}{\geqslant}
\newcommand{\lle}{\leqslant}
\newcommand{\la}{\lambda}

\newcommand{\bla}{\bm\lambda}

\newcommand{\I}{{\mathbb I}}

\newcommand{\ve}{\varepsilon}

\newcommand{\Y}{{\mathscr{Y}}}
\newcommand{\Yi}{{^\imath\mathscr{Y}}}

\newcommand{\YiJ}{{^\imath\mathscr{Y}_{\jmath}}}

\newcommand{\ad}{\mathrm{ad}}
\newcommand{\imag}{\sqrt{-1}}
\newcommand{\comm}[2]{\bigl[#1,#2\bigr]}
\newcommand{\acomm}[2]{\bigl\{#1,#2\bigr\}}

\newenvironment{nouppercase}{%
  \renewcommand{\uppercasenonmath}[1]{}}{}

\begin{document}
\pagestyle{myheadings}
\setcounter{page}{1}

\title[\small Quasi-split iYangians: minimalistic presentations and coideal structures]{\Large Quasi-split iYangians: minimalistic presentations and coideal structures}

\author{Binhe Huang}

\address{
Department of Mathematical Sciences, Indiana University Indianapolis, 402
N. Blackford St., LD 270, Indianapolis, IN 46202, USA
}
\email{binhe.huang.math@gmail.com} 

\author{Kang Lu}

\address{
Shenzhen International Center for Mathematics and Department of Mathematics, Southern University of Science and Technology, Shenzhen, China
}
\email{kang.lu.math@gmail.com,luk@sustech.edu.cn}


\begin{abstract}
We study iYangians, namely twisted Yangians in the Drinfeld current presentation, associated with quasi-split symmetric pairs of type $\mathsf{ADE}$ with nontrivial diagram involution. We establish minimalistic presentations in terms of degree-zero and degree-one generators. Type $\mathsf A_{2n}$ is treated separately: the isolated rank-two case requires two additional relations, whereas in higher rank these relations are forced by the neighboring noncentral orbit. As a by-product, we strengthen \cite[Theorem~3.1]{Lu26min} by proving that the extra relation in the minimalistic presentation of split iYangian is redundant whenever $\g$ has rank at least two. We also construct explicit injective homomorphisms from quasi-split iYangians into Yangians, identify their images as right coideal subalgebras, and obtain isomorphisms between the Drinfeld and $J$ presentations. Finally, we derive triangular estimates for the Drinfeld currents and their coproducts and apply them to the $^\imath\bm\ell$-weights arising by restriction from finite-dimensional Yangian modules.
\end{abstract}
\begin{nouppercase}	
\maketitle
\end{nouppercase}
\setcounter{tocdepth}{1}
\tableofcontents

\section{Introduction}

Yangians, introduced by Drinfeld in \cite{Dri85,Dri87}, form an important class of Hopf algebras in representation theory and quantum integrable systems. They may be viewed as deformations of universal enveloping algebras of current algebras and admit several complementary realizations. The R-matrix presentation and Drinfeld's $J$ presentation make the Hopf structure particularly transparent, whereas the Drinfeld current presentation is especially well suited to representation theory and related geometric constructions.

Twisted Yangians are the counterparts associated with symmetric pairs. They first appeared in the R-matrix framework as coideal subalgebras governed by reflection equations and boundary integrable systems; see, for example, \cite{Che84,Skl88,Ols92,MNO96,MR02,GR16}. A uniform construction for general symmetric pairs in the $J$ presentation was given in \cite{Mac02} and further developed in \cite{BR17}. In both realizations, the coideal structure is visible from the defining formulas. By contrast, a Drinfeld-type presentation makes the current structure accessible and is better adapted to many representation-theoretic questions, but its embedding into the ambient Yangian and the resulting coproduct formulas are much less apparent.

Drinfeld-type presentations of twisted Yangians have emerged more recently. The type $\mathsf{AI}$ case was obtained through Gauss decomposition in \cite{LWZ25GD}, and a uniform current presentation for all split types was subsequently formulated in \cite{LWZ25affine} via degeneration from affine iquantum groups \cite{LW21,Zha22}. The quasi-split case was treated in \cite{LZ24,Lu27super}. We refer to the resulting algebras as \emph{iYangians}. Their associated graded algebras are the universal enveloping algebras of the twisted current algebras $\g[z]^{\check\omega}$, where $\omega$ is an involution of $\g$ and $\check\omega(xz^r)=\omega(x)(-z)^r$. This identifies iYangians as filtered deformations naturally situated in the theory of quantum symmetric pairs \cite{Kolb14}. In the split setting, their relation with affine iquantum groups was established in \cite{LWZ25affine}, while minimalistic presentations and coideal structures in Drinfeld generators were studied in \cite{Lu26min}.

Recent work has developed this picture in several directions. Parabolic presentations and shifted generalizations in types $\mathsf{AI}$ and $\mathsf{AII}$ were constructed in \cite{LPTTW25}; their truncations admit baby comultiplications and are related to finite $W$-algebras of classical type. A related study of the semiclassical limits in type $\mathsf{AI}$ and their connections with Slodowy slices was carried out in \cite{TT24}. Shifted iYangians of arbitrary quasi-split $\mathsf{ADE}$ type, together with PBW bases and iGKLO representations, were introduced in \cite{LWW25SiY}, and their connections with fixed-point loci of affine Grassmannian slices (affine Grassmannian islices) were investigated in \cite{LWW25}. In type $\mathsf{AI}$, a version of the GKLO representation omitting the inhomogeneous term was studied in \cite{BPT25}; the resulting truncations were related to quantizations of symmetric quotients of affine Grassmannian slices. Coulomb-branch realizations associated with quivers with involutions were considered in \cite{SSX25,wang2025quivers,SSX26}, while related constructions involving Steinberg varieties and involutions on quiver varieties appeared in \cite{Li19,su2026twisted,Nak25}. Weight modules over truncated shifted iYangians were studied through orientifold KLRW algebras in \cite{shen2026weight}. At the $q$-deformed level, a uniform construction of shifted affine iquantum groups and their iGKLO-type representations for all quasi-split $\mathsf{ADE}$ types was given in \cite{LWW26affine}. The split simply-laced setting was also considered in \cite{LP26GKLO}. Factorization and coproduct formulas for the Drinfeld--Cartan series of quasi-split affine quantum symmetric pairs of type $\mathsf{AIII}$, together with a compatible boundary $q$-character map, were established in \cite{LP26qchar}. An iHopf-algebraic approach to the Drinfeld presentations and
coideal realizations of affine iquantum groups was developed in
\cite{LuPan26Drinfeld}; this framework also leads to iquantum loop
algebras associated with arbitrary Kac--Moody algebras. Finite presentations and current realizations have also continued to attract attention. The minimalistic presentation was extended to split iYangians of affine type $\mathsf A$ in \cite{Ueda25twistedaffine}. For split types $\mathsf{BCD}$, the relation between the R-matrix and Drinfeld realizations was further developed in \cite{regelskis2026twisted}.

The purpose of the present paper is twofold. First, we establish minimalistic presentations of quasi-split iYangians using only generators of degrees zero and one. Second, we construct explicit realizations inside the corresponding Yangians and use them to describe the coideal structure in Drinfeld generators. Both problems require separate arguments in type $\mathsf A_{2n}$.

Our first main result gives a minimalistic presentation of the iYangian $\Yi:=\Yi^\tau(\g)$, in the spirit of the corresponding results for Yangians and split iYangians in \cite{Lev93gen,GNW18,Lu26min}. Here $\tau$ is a nontrivial Dynkin diagram involution. Although the Drinfeld presentation has infinitely many generators $h_{i,r}$ and $b_{i,r}$, we show that all higher modes can be recovered from $h_{i,0}$, $h_{i,1}$, $b_{i,0}$, $b_{i,1}$ for $i\in \I$. The defining relations are the corresponding low-degree current relations, the finite-type Serre relations, and the explicitly stated relations in the isolated low-rank cases. The proof proceeds through a parallel minimalistic presentation of the classical limit $\mathrm U(\g[z]^{\check\omega})$: we establish the required relations in the twisted current algebra and then lift the result to $\Yi$ using the PBW theorem.

An important consequence is that the apparent extra relation $\big[h_{i,1},[b_{i,1},b_{\tau i,1}]\big]=0$ (cf. \cite[(1.6)]{Lev93gen}) involving the two $\tau$-paired nodes in types $\mathsf A_3$ and $\mathsf D$ is not an independent defining relation. At the associated-graded level it follows from the ordinary low-degree and finite-type Serre relations, and the PBW argument then gives the desired minimalistic presentation of the iYangian.

As a by-product, we strengthen the
minimalistic presentation of split iYangians in \cite[Theorem~3.1]{Lu26min}. For
every simple Lie algebra of rank at least two, the extra relation \cite[(3.4)]{Lu26min} in
that presentation follows from the remaining low-degree and finite
Serre relations. Thus it is needed only in type $\mathsf A_1$ and is
redundant, in particular, in types $\mathsf B_2\cong\mathsf C_2$ and
$\mathsf G_2$; see Corollary~\ref{cor:split-extra-redundant-quantum}. Consequently, \cite[Theorems~4.1 and~5.1]{Lu26min} extend to type $\mathsf G_2$.

Type $\mathsf A_{2n}$ has a genuinely different local structure. On the central $\tau$-orbit $\{n,n+1\}$ one has $c_{n,n+1}=-1$, so the simplified recursion for the Cartan modes on noncentral orbits is no longer available. The rank-two calculation in type $\mathsf A_2$ produces two classical relations that is imposed in the isolated rank-two case. When $n\gge2$, however, the neighboring orbit $\{n-1,n+2\}$ supplies two degree-shifting derivations $\cD_1=\ad(\bh_{n-1,1})$ and $\cD_2=\ad(\bh_{n-1,2})$. Their actions on the two central current families have different signs, and the combination $\cD_1^2+\cD_2$ isolates the two rank-two obstructions. Both therefore follow from the ordinary low-degree and finite-type Serre relations in higher rank, explaining why no additional relation is needed for $\mathsf A_{2n}$.

There is a second, independent exceptional feature in type $\mathsf A_{2n}$. The diagram involution acts nontrivially on certain $\tau$-fixed root spaces, and $\alpha+\tau\alpha$ can be a positive root. The embedding formulas therefore contain correction terms absent from the other quasi-split types. This root-theoretic phenomenon is separate from the minimalistic-presentation argument above.

Our second main result gives explicit formulas for the images of $h_{i,0}$, $h_{i,1}$, $b_{i,0}$, and $b_{i,1}$ in the Yangian $\Y$. Outside type $\mathsf A_{2n}$, the formulas are uniform and root-theoretic; the correction terms just described account for the remaining case. The resulting homomorphism $\varphi:\Yi\longrightarrow\Y$ is injective, and its image is a right coideal subalgebra: $\Delta(\Yi)\subset \Yi\otimes\Y$.
Moreover, under this identification the iYangian agrees with the corresponding twisted Yangian in the $J$ presentation. Thus the current and $J$ realizations provide two complementary descriptions of the same quasi-split iYangian.

Together with \cite[Theorem~4.1]{Lu26min} and its extension to type $\mathsf G_2$ in Corollary~\ref{cor:split-extra-redundant-quantum}, Theorem~\ref{thm:embed} proves the conjecture in \cite[Remark~3.23]{LWW25} for all split and quasi-split finite types.
More precisely, after restoring the deformation parameter $\hbar$ via the
Rees construction and passing to the opposite coproduct to match the
left coideal convention of \cite{LWW25}, the unshifted iYangian in the
Drinfeld presentation is a twisted Yangian in the sense of
\cite[Definition~3.15]{LWW25}. Thus the problem of realizing quasi-split
iYangians, including all split types, as coideal subalgebras of the
corresponding Yangians is completely resolved.

This realization makes the coideal structure explicit in Drinfeld generators. We obtain a closed formula for the coproduct of the first modified Cartan mode $\tl h_{i,1}$ and, consequently, an explicit formula for the coproduct of $b_{i,1}$. These identities are the starting point for triangular estimates in arbitrary degree.

To formulate the higher-degree results, set
\begin{align*}
h_i(u)&=1+\sum_{r\gge0}h_{i,r}u^{-r-1},
&
b_i(u)&=\sum_{r\gge0}b_{i,r}u^{-r-1},
\end{align*}
and let $\xi_i(u)$ denote the Cartan current of $\Y$. We prove
\begin{equation*}
h_i(u)\equiv
\left(1+\frac{\wp_i}{4u}\right)
\xi_i(u)\xi_{\tau i}(-u)
\qquad
\pmod{\Y_{Q_+}[\![u^{-1}]\!]},
\end{equation*}
together with the uniform coproduct estimates
\begin{align*}
\Delta(h_i(u))
&\equiv
h_i(u)\otimes\xi_i(u)\xi_{\tau i}(-u)
&&\pmod{\Yi\otimes\Y_{Q_+}[\![u^{-1}]\!]},
\\
\Delta(b_i(u))
&\equiv
b_i(u)\otimes\xi_{\tau i}(-u)+1\otimes b_i(u)
&&\pmod{\Yi\otimes\Y_{Q_+}[\![u^{-1}]\!]}.
\end{align*}
Here $\wp_n=1$ and $\wp_{n+1}=-1$ at the two central nodes of type $\mathsf A_{2n}$, and $\wp_i=0$ otherwise. The estimate for $b_i(u)$ itself is modified at these two nodes, but the coproduct estimates retain the same uniform triangular form; see Theorem~\ref{thm:hb}.

These are the iYangian analogues of the familiar triangular coproduct estimates for Yangian currents; cf. Lemma~\ref{lem:copro}. They also have a representation-theoretic consequence. When a finite-dimensional $\Y$-module is restricted to $\Yi$, the estimate for $h_i(u)$ controls the resulting $^\imath\bm\ell$-weights in terms of the ambient Yangian $\bm\ell$-weights. In particular, the $i$-th and $\tau i$-th components are coupled by the spectral involution $u\mapsto-u$.

The paper is organized as follows. Section~\ref{sec:pre} recalls the necessary background on Yangians and quasi-split iYangians and fixes the root-vector conventions used throughout. Section~\ref{sec:main-results} states the main results. Section~\ref{sec:pf-thm1} proves the minimalistic presentation through the classical algebra $\mathrm U(\g[z]^{\check\omega})$ and establishes the strengthening for split iYangians. Section~\ref{sec:pfthmB0} constructs the embedding into $\Y$, proves injectivity and the right coideal property, and identifies the image with the twisted Yangian in the $J$ presentation. Section~\ref{sec:proof-thmC} establishes the estimates for the Drinfeld currents and their coproducts and discusses their application to restricted Yangian modules. Appendix~\ref{sec:pf-prop2} completes the classical minimalistic-presentation argument in type $\mathsf A_{2n}$, including the isolated rank-two reconstruction and the higher-rank reduction. Appendix~\ref{sec:app+} derives the degree-zero and degree-one embedding formulas in type $\mathsf A$ from the R-matrix presentation. 

\medskip

\noindent {\bf Acknowledgements.}
The authors thank Weiqiang Wang for helpful comments and suggestions that improved the presentation of this paper. KL thanks Vidas Regelskis for sharing his preliminary calculations \cite{regelskis2026note} on relations among Gaussian generators in quasi-split type $\mathsf D$. KL also gratefully acknowledges the hospitality of the Department of Mathematics at the University of Virginia during the completion of this project. BH is grateful to Vyacheslav Futorny for his guidance and support during his master's studies. ChatGPT 5.6 Sol was used to assist with Lemma \ref{lem:fixed} and calculations at the associated graded level.

\medskip

\section{Yangians and twisted Yangians}\label{sec:pre}
Throughout the paper, all (Lie) algebras are defined over the field of complex numbers $\mathbb C$.

\subsection{Lie algebras}
Let $\g$ be a finite-dimensional simple Lie algebra associated with the Cartan matrix $C=(c_{ij})_{i,j\in\I}$ of type $\mathsf{ADE}$, where $\I=\{1,2,\cdots,n\}$ is the vertex set of the Dynkin diagram corresponding to $\g$. 

Fix an invariant inner product $(\cdot,\cdot)$ on $\g$ and normalize the Chevalley generators $x_i^\pm,\xi_i$ so that $(x_i^+,x_i^-)=1$ and $\xi_i=[x_i^+,x_i^-]$. Let $\Phi$ and $\Phi^+$ denote the set of roots and the set of positive roots, respectively, and denote the simple roots by $\alpha_i$ for $i\in \I$. Let $\{\varpi_i^\vee\}_{i\in \I}\subset\h $ be the basis of fundamental coweights dual to $\xi_i$ for $i\in \I$ with respect to $(\cdot,\cdot)$.

Denote the root space corresponding to a root $\beta\in\Phi$ by $\g_\beta$. For each $\alpha\in \Phi^+$, let $x^\pm_\alpha\in \g_{\pm\alpha}$ be nonzero root vectors normalized so that $(x_\alpha^+,x_\alpha^-)=1$ and $x_{\alpha_i}^\pm=x_i^\pm$. We set $x_{\alpha}=0$ if $\alpha\notin \Phi$ and
\[
x_\alpha:=\begin{cases}
x_\alpha^+, & \text{ if }\alpha\in\Phi^+,\\
x_{-\alpha}^-,& \text{ if }-\alpha\in\Phi^+.
\end{cases}
\]
We also set $x_{\alpha}^\pm=0$ if $\alpha\notin\Phi^+$.

For $\alpha,\beta\in\Phi$, define the structure constant $\eta_{\alpha,\beta}$ as follows: $\eta_{\alpha,\beta}=0$ if $\alpha+\beta\notin\Phi$, and
\[
[x_{\alpha},x_\beta]=\eta_{\alpha,\beta}x_{\alpha+\beta}
\]
if $\alpha+\beta\in \Phi$. We can further rescale the root vectors so that $\eta_{\alpha,\beta}=\eta_{-\beta,-\alpha}$ for $\alpha,\beta\in\Phi^+$; see, for example, \cite[\S25.2]{Hum72}. Therefore we have the following equalities, which will be used frequently below:
\beq\label{eta}
\eta_{\alpha,\beta}=-\eta_{\beta,\alpha}=\eta_{-\beta,-\alpha}=\eta_{-\alpha,\alpha+\beta},
\eeq
for $\alpha,\beta\in\Phi^+$ such that $\alpha+\beta\in\Phi^+$. Here the last equality follows from the invariance of the bilinear form.

\subsection{Quasi-split symmetric pairs}
Let $\tau$ be a nontrivial involution of the Dynkin diagram of $\g$, i.e. $c_{ij}=c_{\tau i,\tau j}$ such that $\tau^2=\mathrm{id}$ and $\tau\ne\id$. Then $\g$ is of type $\mathsf{A}_n$, $\mathsf{D}_n$, or $\mathsf{E}_6$. The quasi-split Satake diagrams $(\I,\tau)$ with $\tau \neq \id$ can be found in Table \ref{tab:Satakediag}, where the index denotes the number of nodes in a given diagram. Except where a split case (i.e. $\tau=\mathrm{id}$) is explicitly mentioned, we assume
throughout that $\tau\ne\id$. In every statement involving type
$\mathsf A_1$, we take $\tau=\id$.

\begin{table}[h]  
\begin{center}
\centering
\setlength{\unitlength}{0.6mm}			\begin{picture}(70,35)(0,5)
    \put(-70,20){${\rm AIII}_{2n-1} \,(r\gge 1)$}
    \put(-70,-15){${\rm AIII}_{2n} \,(r\gge 1)$}
    \put(-70,-50){${\rm DI}_{r} \,(r\gge 4)$}
    \put(-70,-85){${\rm EII}_{6} $}
   
				\put(0,10){$\circ$}
				\put(0,30){$\circ$}
				\put(50,10){$\circ$}
				\put(50,30){$\circ$}
				\put(72,10){$\circ$}
				\put(72,30){$\circ$}
				\put(92,20){$\circ$}

                \put(3,11.5){\line(1,0){16}}
				\put(3,31.5){\line(1,0){16}}
				\put(23,10){$\cdots$}
				\put(23,30){$\cdots$}
				\put(33.5,11.5){\line(1,0){16}}
				\put(33.5,31.5){\line(1,0){16}}
				\put(53,11.5){\line(1,0){18.5}}
				\put(53,31.5){\line(1,0){18.5}}
				
				\put(75,12){\line(2,1){17}}
				\put(75,31){\line(2,-1){17}}

				\color{red}
                \put(-7,20){$\tau$}
				\qbezier(0,13.5)(-4,21.5)(0,29.5)
				\put(-0.25,14){\vector(1,-2){0.5}}
				\put(-0.25,29){\vector(1,2){0.5}}
				
				\qbezier(50,13.5)(46,21.5)(50,29.5)
				\put(49.75,14){\vector(1,-2){0.5}}
				\put(49.75,29){\vector(1,2){0.5}}

				\qbezier(72,13.5)(68,21.5)(72,29.5)
				\put(71.75,14){\vector(1,-2){0.5}}
				\put(71.75,29){\vector(1,2){0.5}}
			\end{picture}
            \\
\setlength{\unitlength}{0.6mm}			\begin{picture}(70,35)(0,5)

				\put(0,10){$\circ$}
				\put(0,30){$\circ$}
				\put(50,10){$\circ$}
				\put(50,30){$\circ$}
				\put(72,10){$\circ$}
				\put(72,30){$\circ$}
				
				\put(3,11.5){\line(1,0){16}}
				\put(3,31.5){\line(1,0){16}}
				\put(23,10){$\cdots$}
				\put(23,30){$\cdots$}
				\put(33.5,11.5){\line(1,0){16}}
				\put(33.5,31.5){\line(1,0){16}}
				\put(53,11.5){\line(1,0){18.5}}
				\put(53,31.5){\line(1,0){18.5}}
				
				\put(73.5,13.6){\line(0,1){16}}
				
				\color{red}
                \put(-7,20){$\tau$}
				\qbezier(0,13.5)(-4,21.5)(0,29.5)
				\put(-0.25,14){\vector(1,-2){0.5}}
				\put(-0.25,29){\vector(1,2){0.5}}
				
				\qbezier(50,13.5)(46,21.5)(50,29.5)
				\put(49.75,14){\vector(1,-2){0.5}}
				\put(49.75,29){\vector(1,2){0.5}}

				\qbezier(76,13.5)(80,21.5)(76,29.5)
				\put(75.85,14){\vector(-1,-2){0.5}}
				\put(75.85,29){\vector(-1,2){0.5}}
			\end{picture}
\\
\setlength{\unitlength}{0.6mm}	\begin{picture}(40,35)(20,-15)
				\put(0,-1){$\circ$}
				\put(3,0){\line(1,0){16.5}}
				\put(20,-1){$\circ$}
				\put(64,-1){$\circ$}
				\put(84,-10){$\circ$}
				\put(84,9.5){$\circ$}

				\put(38,0){\line(-1,0){15.5}}
				\put(64,0){\line(-1,0){15}}
				
				\put(40,-1){$\cdots$}
				
				\put(83.5,9.5){\line(-2,-1){16.5}}
				\put(83.5,-8.5){\line(-2,1){16.5}}
				
				\put(12,-20.5){\begin{picture}(100,40)	\color{red}
                \put(79,20){$\tau$}
						\qbezier(75,13.5)(79,21.5)(75,29.5)
						\put(75.25,14){\vector(-1,-2){0.5}}
						\put(75.25,29){\vector(-1,2){0.5}}
				\end{picture}}
			\end{picture}
        \\
\setlength{\unitlength}{0.6mm}		\begin{picture}(70,35)(20,7)
                \put(10,35){\rotatebox[origin=c]{180}{\begin{picture}(100,10)
							
							\put(10,10){$\circ$}
							
							\put(32,10){$\circ$}
							
							\put(10,30){$\circ$}
							
							\put(32,30){$\circ$}

							\put(31.5,11){\line(-1,0){19}}
							\put(31.5,31){\line(-1,0){19}}
							
							\put(52,22){\line(-2,1){17.5}}
							\put(52,20){\line(-2,-1){17.5}}
							
							\put(54.7,21.2){\line(1,0){19}}

							\put(52,20){$\circ$}
							
							\put(74,20){$\circ$}
					\end{picture}}

					\put(-37,-32.5){\begin{picture}(100,40)	\color{red}
							\qbezier(5,13.5)(9,21.5)(5,29.5)
							\put(5.25,14){\vector(-1,-2){0.5}}
							\put(5.25,29){\vector(-1,2){0.5}}
					\end{picture}}
					\put(-15,-32.5){\begin{picture}(100,40)	\color{red}
                    \put(9,20){$\tau$}
							\qbezier(5,13.5)(9,21.5)(5,29.5)
							\put(5.25,14){\vector(-1,-2){0.5}}
							\put(5.25,29){\vector(-1,2){0.5}}
					\end{picture}}
				}
				
			\end{picture}
\end{center}
\vspace{0.5cm}
\caption{Quasi-split Satake diagrams with $\tau \neq {\rm Id}$}
\label{tab:Satakediag}
\end{table}

Let $\I_1$ be a set of representatives for $\tau$-orbits in $\I$ of length 2 and define $\I_{-1}=\tau \I_1$. Denote by $\I_0$ the set of fixed points of $\tau$ in $\I$, i.e., $\I_0 =\{i\in \I \mid \tau i=i\}$. Then $\I =\I_1 \sqcup \I_0 \sqcup \I_{-1}$. It is also convenient to set $\I_{\gge 0}=\I_1 \sqcup \I_0$ and choose $\I_1$ so that the full Dynkin subdiagram induced by $\I_1$ is connected.

By abuse of notation, we denote the involutions on $\Phi$ and $\g$ induced by $\tau$ again by $\tau$. We have the standard bilinear form $(\cdot,\cdot)$ defined on the root lattice. Then $(\tau\alpha,\tau \beta) = (\alpha, \beta)$ for $\alpha, \beta\in \Phi$.

\begin{lem}\label{lem:fixed}
Suppose $\g$ is not of type $\mathsf A_{2n}$. Then for any $\alpha\in \Phi$ such that $\alpha=\tau \alpha$ we have $\tau|_{\g_\alpha}=\mathrm{id}$.
\end{lem}
\begin{proof}
Let
\[
\mathfrak g=\mathfrak h\oplus \bigoplus_{\alpha\in\Phi}\mathfrak g_\alpha
\]
be the root space decomposition. Since $\tau$ is a diagram automorphism, it
preserves $\Phi^+$. For any root $\alpha$ fixed by $\tau$, the root space
$\mathfrak g_\alpha$ is one-dimensional and $\tau$-stable. Hence $\tau|_{\mathfrak g_\alpha}=\varepsilon_\alpha \operatorname{id}$ for some $\varepsilon_\alpha\in\{\pm 1\}$. 

Let $\mathcal F=\{\alpha\in\Phi:\tau\alpha=\alpha\}$ and let $p=\#\{\alpha\in \mathcal F:\tau|_{\mathfrak g_\alpha}=\operatorname{id}\}$.
The fixed Lie algebra $\mathfrak g^\tau$ receives contributions from three
sources: $\mathfrak h^\tau$, the two-element $\tau$-orbits in $\Phi$, and the $\tau$-fixed root spaces
on which $\tau$ acts by $1$. Therefore $\dim \mathfrak g^\tau
=
\dim \mathfrak h^\tau
+
\tfrac{1}{2}(|\Phi|-|\mathcal F|)
+
p$.
Indeed, each two-element orbit $\{\alpha,\tau\alpha\}$ contributes one
dimension to $\mathfrak g^\tau$, while a fixed root space contributes one
dimension if and only if $\tau$ acts on it by $1$.

The relevant foldings are
\[
\mathsf A_{2n-1}\mapsto \mathsf  C_n,\qquad
\mathsf D_r\mapsto \mathsf B_{r-1},\qquad
\mathsf E_6\mapsto \mathsf F_4.
\]
The corresponding counts are
\[
\begin{array}{|c|c|c|c|c|c|c|}
\hline
\mathfrak g
&
\mathfrak g^\tau
&
\dim \mathfrak g^\tau
&
\dim\mathfrak h^\tau
&
|\Phi|
&
|\mathcal F|
&
p
\\ \hline
\mathsf A_{2n-1}
&
\mathsf C_n
&
n(2n+1)
&
n
&
2n(2n-1)
&
2n
&
2n
\\[1mm]
\mathsf A_{2n}
&
\mathsf B_n
&
n(2n+1)
&
n
&
2n(2n+1)
&
2n
&
0
\\[1mm]
\mathsf D_r
&
\mathsf B_{r-1}
&
(r-1)(2r-1)
&
r-1
&
2r(r-1)
&
2(r-1)(r-2)
&
2(r-1)(r-2)
\\[1mm]
\mathsf E_6
&
\mathsf F_4
&
52
&
4
&
72
&
24
&
24
\\ \hline
\end{array}
\]
Since $|\mathcal F|=p$ for all cases excluding type $\mathsf A_{2n}$, the statement follows.
\end{proof}

By Lemma \ref{lem:fixed}, if $\g$ is not of type $\mathsf A_{2n}$, then we can further choose the root vectors $x_\alpha^\pm$ such that $\tau(x_\alpha^\pm)=x_{\tau \alpha}^\pm$ for all $\alpha\in \Phi^+$. Then we have
\beq\label{eta-tau}
\eta_{\tau\alpha,\tau\beta}=\eta_{\alpha,\beta}
\eeq
for $\alpha,\beta\in\Phi^+$ such that $\alpha+\beta\in\Phi^+$.

When $\g=\gl_N$ or $\g=\fksl_N$, it is convenient to set $i'=N+1-i$ for $1\lle i\lle N$.

\begin{eg}\label{eg:A2n}
For type $\mathsf A$, set $\g=\gl_N$ with the standard generators $\mathbf e_{ij}$ and consider the Dynkin diagram involution 
\[
\tau:\gl_N\to \gl_N,\qquad \mathbf e_{ij}\mapsto (-1)^{i-j+1}\mathbf e_{j'i'}.
\]
Then $\mathbf e_{ij}\in \g_\alpha$ with $\tau\alpha=\alpha$ for $\alpha\in\Phi$ if and only if $i\ne j=i'$. Then
\[
\tau(\mathbf e_{ii'})=(-1)^{i'-i+1}\mathbf e_{ii'}=(-1)^N\mathbf e_{ii'}.
\]
Thus, for any $\alpha\in\Phi$ such that $\tau \alpha=\alpha$, we have
\[
\tau|_{\g_\alpha} =
\begin{cases}
   \mathrm{id}, &\text{ if }N=2n,\\
   - \mathrm{id}, &\text{ if }N=2n+1.  
\end{cases}
\]
\end{eg}

Let $\omega$ be the composition of $\tau$ with the Chevalley involution of $\g$:
\beq\label{omega}
\omega:\g\to\g,\qquad \xi_i\to -\xi_{\tau i},\qquad x_i^\pm \to -x_{\tau i}^\mp,\qquad (i\in \I).
\eeq
Denote by $\mathfrak k:=\g^\omega$ the $\omega$-fixed point Lie subalgebra of $\g$ and by $\mathfrak p$ the eigenspace of $\omega$ corresponding to the eigenvalue $-1$. Then $(\g,\mathfrak k)$ is a symmetric pair of \textit{quasi-split type}. 

Note that 
\be
[\mathfrak k,\mathfrak k]\subset \mathfrak k,\qquad [\mathfrak k,\mathfrak p]\subset \mathfrak p,\qquad [\mathfrak p,\mathfrak p]\subset \mathfrak k.
\ee
We can always pick $x_\alpha^\pm$ so that $\omega(x_\alpha^\pm)=-x_{\tau \alpha}^\mp$ for $\alpha\in\Phi^+$. For instance, in the case $\g=\mathfrak{sl}_{N}$ where $N=2n+1$, we fix a specific choice of root vectors for $\alpha=\ve_i-\ve_j$ with $1\lle i<j\lle N$:
\beq\label{eq:rt-A2n0}
x_\alpha^+=(\sqrt{-1})^{j-i-1}\mathbf e_{ij},\qquad x_\alpha^-=(\sqrt{-1})^{i-j+1}\mathbf e_{ji}.
\eeq
cf. Example \ref{eg:A2n}.

A basis for $\mathfrak k$ is given by $x_{\alpha}^+-x_{\tau \alpha}^-$ and $\xi_i-\xi_{\tau i}$ for $\alpha\in\Phi^+$ and  $i\in \I_1$ whereas a basis for $\mathfrak p$ is given by $x_{\alpha}^++x_{\tau \alpha}^-$, $\xi_i$, and $\xi_j+\xi_{\tau j}$ for $\alpha\in\Phi^+$, $i\in \I_0$, and $j\in \I_1$. Set
\be
b_{\alpha}:=x_{\alpha}^+-x_{\tau\alpha}^-,\qquad  y_{\alpha}:=x_{\alpha}^++x_{\tau\alpha}^-,\qquad (\alpha\in\Phi^+).
\ee

Denote by $C_{\mathfrak k}$, $C_{\mathfrak p}$, and $C_{\g}$ the Casimir elements of $\mathfrak k$, $\mathfrak p$, and $\g$, respectively. We have
\be
C_{\mathfrak k}=-\tfrac12\sum_{\alpha\in\Phi^+}(x_{\alpha}^+-x_{\tau \alpha}^-)(x_{\tau\alpha}^+ - x_{\alpha}^-)+ \tfrac12\sum_{i\in \I_1}(\xi_i-\xi_{\tau i})(\varpi_i^\vee-\varpi_{\tau i}^\vee).
\ee
Denote by $\Omega_{\mathfrak k}$, $\Omega_{\mathfrak p}$, and $\Omega_{\g}$ the two-tensor Casimir elements of $\mathfrak k$, $\mathfrak p$, and $\g$, respectively. We have
\be
\Omega_{\g}= \Omega_{\mathfrak k}+\Omega_{\mathfrak p}, \qquad [x\otimes 1+1\otimes x,\Omega_\g]=0,
\ee
for $x\in\g$.

The following standard lemma will be used below; see, for example, \cite[Lemma~2.1]{Lu26min}.
\begin{lem}\label{lem:omega}
For $x\in\mathfrak p$, we have
\[
[x\otimes 1,\Omega_{\mathfrak p}]+[1\otimes x,\Omega_{\mathfrak k}]=[x\otimes 1,\Omega_{\mathfrak k}]+[1\otimes x,\Omega_{\mathfrak p}]=0.
\]
\end{lem}

\subsection{Yangians}\label{sec:Yang}
We recall the definitions of Yangians, following \cite{GNW18}. The Yangian $\Y$ associated with an arbitrary finite-dimensional simple Lie algebra $\g$ was first introduced by Drinfeld in the $J$ presentation.
\begin{dfn}[\cite{Dri85}]
\label{def:Y-J}
The Yangian $\Y:=\Y(\g)$ is the unital associative algebra generated by the elements $x$ and $J(x)$ for $x\in\g$, subject to the following defining relations:
\begin{equation*}
\begin{split}
&xy-yx = [x,y] \text{ for all } x,y\in\g,  \text{ $J$ is linear in }  x\in\g, \quad
   J([x,y]) = [x, J(y)],
\\
  &\big[J(x), J([y,z])\big] + \big[J(z), J([x,y])\big] + \big[J(y), J([z,x])\big]\\
  &\hskip3.5cm = \sum_{a,b,c\in \Lambda} \big([x,\xi_a],\big[[y,\xi_b],[z,\xi_c]\big]\big) \big\{ \xi_a,\xi_b,\xi_c \big\},
\\
  &\big[[J(x), J(y)], [z, J(w)]\big] + \big[[J(z), J(w)], [x, J(y)]\big]\\
 &=  \sum_{a,b,c\in \Lambda} \Big(\big([x,\xi_a],\big[[y,\xi_b],[[z,w],\xi_c]\big]\big) \\&
 \hskip3.5cm + \big([z,\xi_a],\big[[w,\xi_b],[[x,y],\xi_c]\big]\big) \Big)\big\{ \xi_a,\xi_b,J(\xi_c) \big\},
\end{split}
\end{equation*}
where $\{ \xi_a \}_{a \in \Lambda}$ is an orthonormal basis of $\g$, $\Lambda$ is a fixed indexing set of size $\dim \g$, and $\{ \xi_a,\xi_b,\xi_c \} = \frac{1}{24} \sum_{\pi\in \fkS_3} \xi_{\pi(a)} \xi_{\pi(b)} \xi_{\pi(c)}$, with $\fkS_3$ denoting the group of permutations of $\{ a,b,c\}$.
\end{dfn}

The Yangian $\Y$ is a Hopf algebra with the coproduct determined by
\beq\label{copro-J}
\Delta(x)=x\otimes 1+1\otimes x,\quad \Delta(J(x))=J(x)\otimes 1+1\otimes J(x)+\tfrac12[x\otimes 1,\Omega_\g].
\eeq

Later, Drinfeld introduced a new presentation of $\Y$, now known as the Drinfeld, or new/current, presentation. We recall it below.

Given two elements $x,y$ in an associative algebra $\mathcal A$, set $\{x,y\}:=xy+yx$.

\begin{dfn}[\cite{Dri87}]
\label{def:Yangian}The Yangian $\Y$ is the unital associative algebra with generators $x_{i
  ,r}^\pm,\xi_{i,r}$ for $i\in \I, r\in\bN$ subject to the following
defining relations: 
\begin{align}
  [\xi_{i,r}, \xi_{j,s}]& = 0, \notag\\
  [\xi_{i,0}, x^\pm_{j,s} ]& = \pm (\alpha_i,\alpha_j) x^\pm_{j,s},
  \label{eq:relHX}
\\
  [x^+_{i,r}, x_{j,s}^-]& = \delta_{ij} \xi_{i, r+s}, \label{eq:relXX}
\\
  [\xi_{i, r+1}, x^\pm_{j,s}] - [\xi_{i,r}, x^\pm_{j, s+1}]& =
  \pm\tfrac12(\alpha_i,\alpha_j) \big\{
    \xi_{i,r},x^\pm_{j,s}
  \big\}, \label{eq:relexHX}
\\
  [x^\pm_{i, r+1}, x^\pm_{j, s}] - [x^\pm_{i, r}, x^\pm_{j, s+1}] &= 
  \pm\tfrac12(\alpha_i,\alpha_j) \big\{
    x^\pm_{i,r},x^\pm_{j,s}\big\},
   \label{eq:relexXX}
\\
  \sum_{\sigma\in \fkS_m}
   \Big[x^\pm_{i, r_{\sigma(1)}}, \big[ x^\pm_{i, r_{\sigma(2)}}, \cdots,\,
       &[x^\pm_{i, r_{\sigma(m)}}, x^\pm_{j, s}] \cdots \big]\Big]
   = 0
   \quad \text{if $i\neq j$,}
 \notag
\end{align}
where $m = 1 - c_{ij}$.
\end{dfn}

Set
\beq\label{tlxidef}
\tl \xi_{i,1}:=\xi_{i,1}-\tfrac12\xi_{i,0}^2,
\eeq 
then by \eqref{eq:relHX} and \eqref{eq:relexHX} we have
\beq\label{tlxicom}
\big[\tl \xi_{i,1},x_{j,r}^\pm\big]=\pm(\alpha_i,\alpha_j)x_{j,r+1}^\pm.
\eeq 

The universal enveloping algebra $\mathrm{U}(\g)$ is identified with a subalgebra of $\Y$ via the map $\xi_i\mapsto \xi_{i,0}$, $x_i^{\pm}\mapsto x_{i,0}^\pm$ for $i\in\I$. 

The isomorphism between the \(J\)-presentation in Definition \ref{def:Y-J} and the Drinfeld presentation in Definition \ref{def:Yangian} is given by (see \cite[Theorem~1]{Dri87} and \cite[Theorem~2.6]{GRW19})
\begin{equation}\label{eq:J}
  \begin{split}
  x^\pm_i & \mapsto x^\pm_{i,0}, \hskip1.7cm \xi_i \mapsto \xi_{i,0},
\\
  J(\xi_i) & \mapsto \xi_{i,1} + v_i, \hskip0.9cm
  v_i := \tfrac14 \sum_{\alpha\in\Phi^+}
  (\alpha,\alpha_i) \{x^+_\alpha, x^-_\alpha\}
  - \tfrac12 \xi_i^2,
\\
  J(x^\pm_i) &\mapsto x^\pm_{i, 1} + w^\pm_i, \hskip0.6cm
  w^\pm_i := \pm \tfrac14
  \sum_{\alpha\in\Phi^+} \left\{[x_i^\pm,x^\pm_\alpha],x^\mp_\alpha\right\} 
  - \tfrac14\{x^\pm_i,\xi_i\}.
  \end{split}
\end{equation}
In the Drinfeld presentation, the coproduct defined in \eqref{copro-J} satisfies
\begin{align*}
&\Delta(\xi_i)=\xi_i\otimes 1+1\otimes \xi_i,\qquad ~~~~~\Delta(x_i^\pm)=x_i^\pm\otimes 1+1\otimes x_i^\pm,\\
&\Delta(\xi_{i,1})=\xi_{i,1}\otimes 1+1\otimes \xi_{i,1}+\xi_{i}\otimes \xi_{i}-\sum_{\alpha\in\Phi^+}(\alpha,\alpha_i)x_{\alpha}^-\otimes x_\alpha^+,\\
&\Delta(x_{i,1}^+)=x_{i,1}^+\otimes 1+1\otimes x_{i,1}^++\xi_{i}\otimes x_{i}^+-\sum_{\alpha\in\Phi^+}x_{\alpha}^-\otimes [x_{i}^+,x_\alpha^+],\\
&\Delta(x_{i,1}^-)=x_{i,1}^-\otimes 1+1\otimes x_{i,1}^-+x_{i}^-\otimes \xi_{i}+\sum_{\alpha\in\Phi^+}[x_{i}^-,x_\alpha^-]\otimes x_\alpha^+.
\end{align*}

The coproduct for Drinfeld generators of higher degrees has the following estimates. 

Let $\Y^{\gge 0}$ (resp. $\Y^{\lle 0}$) be the Borel subalgebra of $\Y$ generated by the elements $\xi_{i,r}$, $x_{i,r}^+$ (resp $x_{i,r}^-$) for $i\in\I$ and $r\in\bN$. Let $Q$ be the root lattice. Set
\[
Q_+:=\Big\{\sum_{i\in\I}k_i\alpha_i~\Big|~k_i\in\bN,\sum_{i\in \I}k_i>0\Big\},\qquad Q_-:=-Q_+.
\]
It is well known that $\Y$ is $Q$-graded by setting
\[
\mathrm{deg}\,\xi_{i,r}=0,\qquad \mathrm{deg}\,x_{i,r}^\pm=\pm\alpha_i. 
\]
For $\alpha\in Q$ and $\mathsf S\subset Q$, denote
\[
\Y_{\alpha} :=\{w\in \Y\mid \deg w=\alpha\},\qquad \Y_{\mathsf S} :=\mathrm{span}\{w\in \Y\mid \deg w\in \mathsf S\}.
\]
In particular, $\Y_{Q_+}$ is the subalgebra of $\Y$ spanned by homogeneous elements of degree in $Q_+$. Similarly, one can define these notations for the subalgebras $\Y^{\gge 0}$ and $\Y^{\lle 0}$.

Set
\be
\xi_i(u)=1+\sum_{r\gge 0}\xi_{i,r}u^{-r-1},\qquad  x_i^\pm(u)=\sum_{r\gge 0}x_{i,r}^\pm u^{-r-1}.
\ee

The following well-known statement can be found in \cite[Proposition 2.9]{GW23} and \cite[Lemma 2.5]{HZ24}. It improves \cite[Lemma 1]{Kn95} and \cite[Proposition 2.8]{CP91}.
\begin{lem}\label{lem:copro}
We have
\begin{align*}
&\Delta(\xi_i(u))\equiv \xi_i(u)\otimes \xi_i(u) &\pmod{\Y^{\lle 0}_{Q_-}\otimes \Y^{\gge0}_{Q_+}[\![u^{-1}]\!]},\\
&\Delta(x_i^+(u)) \equiv x_i^+(u)\otimes 1+\xi_i(u)\otimes x_i^+(u) &\pmod{\Y^{\lle 0}_{Q_-}\otimes \Y^{\gge0}_{\alpha_i+Q_+}[\![u^{-1}]\!]},\\
&\Delta(x_i^-(u)) \equiv x_i^-(u)\otimes \xi_i(u)+1\otimes x_i^-(u) &\pmod{\Y^{\lle 0}_{-\alpha_i+Q_-}\otimes \Y^{\gge0}_{Q_+}[\![u^{-1}]\!]}.
\end{align*}
\end{lem}

Define a filtration on $\Y$ by setting $\mathrm{deg}\,\xi_{i,r}=\mathrm{deg}\,x_{i,r}^\pm=r$, and denote by $\gr\,\Y$ the associated graded algebra of $\Y$. Let $\bar\xi_{i,r},\bar x_{i,r}^\pm$ be the images of $ \xi_{i,r}, x_{i,r}^\pm$ in the $r$-th component of $\mathrm{gr}\,\Y$, respectively. Then there is a well-known graded Hopf algebra isomorphism
\beq\label{ass}
\begin{split}
\rho:~&\mathrm{U}(\g[z])\stackrel{\cong\,}{\longrightarrow} \mathrm{gr}\,\Y, \\
&\xi_iz^r\mapsto \bar\xi_{i,r},\qquad x_i^{\pm }z^r \mapsto \bar x_{i,r}^\pm,\qquad  (i\in \I,r\in\bN).
\end{split}
\eeq

\subsection{Twisted Yangians in $J$ presentation}\label{sec:ty}
We now recall twisted Yangians in the $J$ presentation from \cite{Mac02,BR17}. 

Twisted Yangians for arbitrary symmetric pairs were defined as coideal subalgebras of $\Y$ via the $J$ presentation in \cite{Mac02}. They were further studied in \cite{BR17} via homogeneous quantization, where an explicit set of generators with defining relations similar to those in Definition \ref{def:Y-J} was obtained.
\begin{dfn}[\cite{Mac02,BR17}]\label{def:tY-J}
The twisted Yangian $\YiJ$ in the $J$ presentation is the subalgebra of $\Y$ generated by the following elements:
\[
x,\quad B(y):=J(y)-\tfrac14[y,C_{\mathfrak k}],
\]
where $x\in \mathfrak k$ and $y\in\mathfrak p$. 
\end{dfn}
Note that $\YiJ$ is a \textit{right} coideal subalgebra of $\Y$, i.e. $\Delta(\YiJ)\subset \YiJ\otimes \Y$. Indeed, Lemma \ref{lem:omega} and \eqref{copro-J} imply that
\begin{align*}
&\Delta(x)=x\otimes 1+1\otimes x,\\
&\Delta(B(y))=B(y)\otimes 1+1\otimes B(y)-[1\otimes y,\Omega_{\mathfrak k}].
\end{align*}
Moreover, this definition works for an arbitrary involution $\omega$ and hence can be used to define twisted Yangians associated to arbitrary symmetric pairs.

In \cite[Theorem~5.5]{BR17}, the definition of the twisted Yangian in the $J$ presentation is slightly different. The relation between the two definitions of twisted Yangians in the $J$ presentation is described in \cite[\S2.3]{Lu26min}.

One can define a family of twisted Yangians in the $J$ presentation as follows. 

Let $c$ be an arbitrary complex number and $\mathscr K$ an element in $\g$ such that $\mathscr K\in \mathfrak k\cap Z(\mathfrak k)$, where 
$$Z(\mathfrak k):=\{x\in \mathfrak k\,|\,[x,\mathfrak k]=0\}.$$

\begin{dfn}\label{def:tY-J2}
The twisted Yangian $\YiJ(\mathscr K,c)$ in the $J$ presentation is the subalgebra of $\Y$ generated by the following elements:
\[
x,\quad B_{\mathscr K,c}(y):=B(y)-c[y,\mathscr K],
\]
where $x\in \mathfrak k$ and $y\in\mathfrak p$. 
\end{dfn}

Since $\mathscr K$ is a primitive element, one shows that $\YiJ(\mathscr K,c)$ is a coideal subalgebra of $\Y$. The condition $\mathscr K\in \mathfrak k$ implies that $[y,\mathscr K]\in \mathfrak p$ for $y\in \mathfrak p$. Hence in general $\YiJ(\mathscr K,c)$ is different from $\YiJ$. The condition $\mathscr K\in Z(\mathfrak k)$ implies that 
$$
[x,B_{\mathscr K,c}(y)]=B_{\mathscr K,c}([x,y]),\qquad \text{for $x\in \mathfrak k$ and $y\in \mathfrak p$.}
$$
If we set $c=0$, then we recover the twisted Yangians in Definition \ref{def:tY-J}.

We shall use the following specific $\mathscr K$ for the case of type $\mathsf A_{2n}$,
\beq\label{eq:K-def}
\mathscr K:=\sum_{i=1}^n \frac{i}{2n+1}(\xi_i-\xi_{\tau i})-\sqrt{-1}\sum_{i=1}^n b_{\ve_i-\ve_{i'}}.
\eeq
By convention, we define $\mathscr K=0$ for other types.

\begin{lem}\label{lem:K-central}
If $\g=\mathfrak{sl}_{2n+1}$, then the element $\mathscr K$ in \eqref{eq:K-def} belongs to $\mathfrak k\cap Z(\mathfrak k)$.
\end{lem}
\begin{proof}
Put $N=2n+1$ and set
\[
 \mathscr S=\sum_{i=1}^{N}(-1)^{n-i}\mathbf e_{i',i}.
\]
Then $\mathscr S^2=\mathbf 1_N$ and, for the matrix realization fixed in
\eqref{eq:rt-A2n0},
\[
 \omega(\mathbf e_{ab})=(-1)^{a-b}\mathbf e_{a',b'}
 =\mathscr S\mathbf e_{ab}\mathscr S^{-1}.
\]
Consequently, $\mathfrak k=Z_{\mathfrak g}(\mathscr S)$.  Our choice of root
vectors gives
\[
 -\sqrt{-1}\,b_{\ve_i-\ve_{i'}}
 =(-1)^{n-i}(\mathbf e_{i,i'}+\mathbf e_{i',i}),
 \qquad 1\lle i\lle n.
\]
We also have
\[
 \sum_{i=1}^n\frac{i}{N}(\xi_i-\xi_{\tau i})
 =\frac1N\mathbf 1_N-\mathbf e_{n+1,n+1}.
\]
It follows that $\mathscr K=\mathscr S+\frac1N\mathbf 1_N$. Since $\operatorname{tr}(\mathscr S)=-1$, the matrix $\mathscr K$ is traceless and
hence belongs to $\mathfrak g=\mathfrak{sl}_N$.  It is clearly fixed by $\omega$, so
$\mathscr K\in\mathfrak k$; moreover it commutes with
$Z_{\mathfrak g}(\mathscr S)=\mathfrak k$.  Thus
$\mathscr K\in\mathfrak k\cap Z(\mathfrak k)$.
\end{proof}

\subsection{iYangians}\label{sec:iy}
Recently, a Drinfeld-type presentation for twisted Yangians of quasi-split type was given in \cite{LZ24}; see also \cite{LWW25SiY} for shifted generalizations. We shall call these twisted Yangians the \emph{quasi-split iYangians}.

\begin{dfn}[{\cite[Theorem~3.9]{LZ24}}]\label{def:YN}
The \emph{quasi-split iYangian} associated to $(\I,\tau)$, denoted simply by $\Yi$, is the $\mathbb{C}$-algebra generated by $h_{i,r},b_{i,r},i\in \mathbb{I},r\in \bN$, subject to
\begin{align}\label{qs1}
& [h_{i,r},h_{j,s}]=0,\qquad\qquad  h_{\tau i,0}=-h_{i,0},
\\\label{qs2}
& [h_{i,0},b_{j,r}]=  (c_{ij}-c_{\tau i,j}) b_{j,r},
\\\label{qs3}
& [h_{i,1},b_{j,r}]=  (c_{ij}+c_{\tau i,j}) b_{j,r+1}+\frac{(c_{ij}-c_{\tau i,j})}{2}\{h_{i,0},b_{j,r}\},
\\\notag
& [h_{i,r+2},b_{j,s}] - [h_{i,r},b_{j,s+2}]=\frac{c_{ij}-c_{\tau i,j}}{2}\{h_{i,r+1},b_{j,s}\}
\\\label{qs4}
&\hskip2cm +\frac{c_{ij}+c_{\tau i,j}}{2}\{h_{i,r},b_{j,s+1}\}+\frac{c_{ij}c_{\tau i,j}}{4}[h_{i,r}, b_{j,s}],
\\\label{qs5}
& [b_{i,r+1 },b_{j,s}]  - [b_{i,r },b_{j,s+1 }] 
 =\frac{c_{ij}}{2}\{b_{i,r },b_{j,s }\}-2 \delta_{\tau i,j}(-1)^{r}h_{j,r+s+1 },
\end{align}
and the finite type Serre relations \eqref{finserre1-ty}--\eqref{finserre4-ty}: 
\begin{align}
&[b_{i,0},b_{j,0}]=\delta_{\tau i, j}h_{j,0}, &c_{ij}=0,\label{finserre1-ty}
\\
&\big[b_{i,0},[b_{i,0},b_{\tau i,0}]\big]  
= 4b_{i,0}, &c_{i,\tau i}=-1,\label{finserre2-ty}\\
&\big[b_{i,0},[b_{i,0},b_{j,0}]\big] =0, & c_{ij}=-1 \text{ and }i\ne \tau i\ne j, \label{finserre3-ty}\\
&\big[b_{i,0},[b_{i,0},b_{j,0}]\big] =-b_{j,0}, & c_{ij}=-1 \text{ and }i=\tau i.\label{finserre4-ty}
\end{align}
\end{dfn}

One of our main results, Theorem~\ref{thm:embed}, shows that the
iYangian $\Yi$ is isomorphic to a twisted Yangian in the
$J$ presentation and hence is a right coideal subalgebra of $\Y$.
This identifies the Drinfeld-current and $J$ realizations of the
same quasi-split twisted Yangian.

\medskip

We record some properties of iYangians that will be used later. Set
\be
\tl h_{i,1}:=h_{i,1}-\hf h_{i,0}^2.
\ee
By \eqref{qs2}--\eqref{qs3}, we have
\beq\label{eq:hi1bjr-new}
[\tl h_{i,1},b_{j,r}]=(c_{ij}+c_{\tau i,j})b_{j,r+1}.
\eeq
\begin{lem}\label{lem:generate}
The iYangian $\Yi$ is generated by $h_{i,0},h_{i,1},b_{i,0}$ for $i\in \I$.
\end{lem}
\begin{proof}
It follows from \eqref{eq:hi1bjr-new} that all $b_{j,r}$ for $j\in \I$ and $r\in\bN$ are generated by $h_{i,0},h_{i,1},b_{i,0}$ for $i\in \I$. Then \eqref{qs5} shows that all $h_{i,r}$ are also generated by $h_{i,0},h_{i,1},b_{i,0}$ for $i\in \I$.  
\end{proof}

\begin{lem}[{\cite[Proposition~3.12]{LZ24}}]
Suppose the relations \eqref{qs1}--\eqref{qs5} hold. Assuming further that $c_{i,\tau i}=-1$ and the relation \eqref{finserre2-ty} holds, then we have
\beq\label{todo-}
\mathrm{Sym}_{k_1,k_2}\big[b_{i,k_1},[b_{i,k_2},b_{\tau i,r}] \big] =\tfrac{4}{3}\,\mathrm{Sym}_{k_1,k_2}(-1)^{k_1}\sum_{p=0}^{k_1+r}3^{-p}[b_{i,k_2+p},h_{\tau i,k_1+r-p}],
\eeq
where $\mathrm{Sym}_{k_1,k_2}$ is the symmetrization map over $k_1,k_2$.
\end{lem}

\begin{lem}\label{lem:hbhb}
We have the following relations in $\Yi$.
\begin{enumerate}
    \item If $i\in \I_0$, then  $\big[h_{i,1},\big[b_{i,1},[h_{i,1},b_{i,1}]\big]\big]=4\big[b_{i,1}^2,h_{i,1}\big]$ .
\item If $i\in \I_1$ and $c_{i,\tau i}=-1$, then 
\begin{align}
&\big[b_{i,1},[b_{i,1},b_{\tau i,0}]\big]
+
\big[b_{\tau i,1},[b_{\tau i,1},b_{i,0}]\big]-2\{h_{i,0},b_{i,1}\}-2\{h_{\tau i,0},b_{\tau i,1}\}
=0,\label{eq:A2-different1}\\
&\big[[\tl h_{i,1},b_{i,1}],b_{\tau i,1}\big]-\big[b_{i,0},\big[\tl h_{i,1},[\tl h_{i,1},b_{\tau i,1}]\big]\big]+\hf\big[\tl h_{i,1},\{b_{i,0},b_{\tau i,1}\}\big]=0.\label{eq:A2-different2}
\end{align}
\end{enumerate}
\end{lem}
\begin{proof}
The first relation is proved in \cite[Lemma~2.9]{Lu26min}. 
The relation \eqref{eq:A2-different1} follows from \eqref{todo-} and \eqref{qs3}. By \eqref{qs5}, we have $-2h_{\tau i,2}=[b_{i,1 },b_{\tau i,1}]  - [b_{i,0 },b_{\tau i,2 }] 
 +\tfrac{1}{2}\{b_{i,0 },b_{\tau i,1 }\}$. The relation \eqref{eq:A2-different2} is a reformulation of $0=-2[\tl h_{i,1},h_{\tau i,2}]$ by using this relation, where we also used $b_{\tau i,2 }=[\tl h_{i,1},b_{\tau i,1}]$ from \eqref{eq:hi1bjr-new}.
\end{proof}

Extend the involution $\omega$ on $\g$ defined in \eqref{omega} to an involution $\check\omega $ of the current algebra $\g[z]$ by
\[
\check\omega: \g[z]\to \g[z],\qquad xz^r\mapsto \omega(x)(-z)^r,
\]
for $x\in \g$ and $r\in\bN$. We call $\g[z]^{\check\omega}$ a \textit{twisted current algebra}. The twisted current algebra $\g[z]^{\check\omega}$ has the following basis:
\[
\xi_iz^{2r+1},\quad (\xi_{j}-(-1)^r \xi_{\tau j})z^r,\quad \big(x_\alpha^+-(-1)^rx_{\tau \alpha}^-\big)z^r,\quad (i\in\I_0,j\in \I_1,\alpha\in\Phi^+,r\in\bN).
\]

Define a filtration on $\Yi$ by setting $\mathrm{deg}\,h_{i,2r+1}=2r+1$ and $\mathrm{deg}\,b_{i,r}=\mathrm{deg}\,b_{j,r}=\mathrm{deg}\,h_{j,r}=r$ for $i\in\I_0$, $j\in\I_1\sqcup\I_{-1}$, and $r\in\bN$. Denote by $\gr\,\Yi$ the associated graded algebra of $\Yi$. Let $\bar h_{i,2r+1},\bar h_{j,r},\bar b_{i,r}$ be the images of $ h_{i,2r+1},h_{j,r}, b_{i,r}$ in the $(2r+1)$-st and $r$-th components of $\mathrm{gr}\,\Yi$, respectively. It is known from \cite[Corollary~3.8]{LZ24} that there exists an algebra isomorphism
\beq\label{assi}
\begin{split}
\varrho:~&\mathrm{U}(\g[z]^{\check\omega})\stackrel{\cong\,}{\longrightarrow} \mathrm{gr}\,\Yi,\\
&(\xi_i-(-1)^r\xi_{\tau i})z^{r}\mapsto \bar h_{i,r},\quad (x_i^+-(-1)^rx_{\tau i}^-)z^r \mapsto \bar b_{i,r},\quad (i\in \I,r\in\bN).
\end{split}
\eeq

\section{Main results} \label{sec:main-results}
\subsection{A minimalistic presentation for $\Yi$}
Our first main result is a minimalistic presentation of the iYangian $\Yi$ in the spirit of \cite{Lev93gen,GNW18,Lu26min} for all quasi-split types.
\begin{thm}\label{thm:min}
Let $\g$ be of type $\mathsf A_{r}$, $\mathsf D_r$, or $\mathsf E_6$. Then the iYangian $\Yi$ is isomorphic to the $\mathbb{C}$-algebra generated by $h_{i,r},b_{i,r}$, for $i\in \mathbb{I},0\lle r\lle 1$, subject to
\begin{align}\label{redqs1}
& [h_{i,r},h_{j,s}]=0,\qquad\qquad  h_{\tau i,0}=-h_{i,0},
\\\label{redqs2}
& [h_{i,0},b_{j,r}]=  (c_{ij}-c_{\tau i,j}) b_{j,r},
\\\label{redqs3}
& [h_{i,1},b_{j,0}]=  (c_{ij}+c_{\tau i,j}) b_{j,1}+\frac{(c_{ij}-c_{\tau i,j})}{2}\{h_{i,0},b_{j,0}\},
\\\label{redqs5}
& [b_{i,1 },b_{j,0}]  - [b_{i,0 },b_{j,1 }] 
 =\frac{c_{ij}}{2}\{b_{i,0 },b_{j,0 }\}-2 \delta_{\tau i,j}h_{j,1 },
\end{align}
the finite type Serre relations \eqref{finserre1-ty}--\eqref{finserre4-ty}, and possibly extra relations as follows.
\begin{enumerate}
\item If $\g$ is of type $\mathsf A_1$, then an additional relation is required:
\begin{align}
&\big[h_{1,1},\big[b_{1,1},[h_{1,1},b_{1,1}]\big]\big]=4[b_{1,1}^2,h_{1,1}].\label{eq:extraA1}
\end{align}
\item If $\g$ is of type $\mathsf A_2$, then two additional relations are required:
\begin{align}
&\big[b_{1,1},[b_{1,1},b_{2,0}]\big]
+
\big[b_{2,1},[b_{2,1},b_{1,0}]\big]-2\{h_{1,0},b_{1,1}\}-2\{h_{2,0},b_{2,1}\}
=0,\label{eq:extraA2-1}\\
&\big[[\tl h_{1,1},b_{1,1}],b_{2,1}\big]-\big[b_{1,0},\big[\tl h_{1,1},[\tl h_{1,1},b_{2,1}]\big]\big]+\hf\big[\tl h_{1,1},\{b_{1,0},b_{2,1}\}\big]=0.\label{eq:extraA2-2}
\end{align}
\item No additional relations are required for other types.
\end{enumerate} 
Here $\tl h_{i,1}:=h_{i,1}-\hf h_{i,0}^2$.
\end{thm}

Instead of working directly with the Yangian as in \cite{Lev93gen,GNW18}, we use the idea from the proof of \cite[Theorem~4.14]{LWZ25affine}, which greatly simplifies the problem by reducing it to the associated graded level.

\begin{prop}\label{prop:red}
Let $\g$ be of type $\mathsf A_{r}$, $\mathsf D_r$, or $\mathsf E_6$. Then the algebra $\mathrm{U}(\g[z]^{\check\omega})$ is isomorphic to the unital associative algebra generated by the elements $\bh_{i,0}$, $\bh_{i,1}$, $\bb_{i,0}$,  $\bb_{i,1}$ for $i\in\I$ subject only to the following relations:
\begin{align}
[\bh_{i,r},\bh_{j,s}]&=0&(0\lle r,s\lle 1),\label{tchh}\\
\bh_{i,r}&=(-1)^{r+1}\bh_{\tau i,r} & (0\lle r\lle 1),\label{tch-sym}\\
[\bh_{i,0},\bb_{j,r}]&=(c_{ij}-c_{\tau i,j})\bb_{j,r}& (0\lle r\lle 1),\label{tch0b}\\
[\bh_{i,1},\bb_{j,0}]&=(c_{ij}+c_{\tau i,j})\bb_{j,1},\label{tch1b}\\
[\bb_{i,1},\bb_{j,0}]&=[\bb_{i,0},\bb_{j,1}]-2\delta_{\tau i,j}\bh_{j,1},\label{tcbb}\\
[\bb_{i,r},\bb_{\tau i,s}]&=(-1)^{r}\bh_{\tau i,r+s} &(r+s\lle 1 \text{\emph{ and }} c_{i,\tau i}=0),\label{tcbb2}
\end{align}
the finite type Serre relations: 
\begin{align}
&[\bb_{i,0},\bb_{j,0}]=\delta_{\tau i, j}\bh_{j,0}, &c_{ij}=0,\label{tcfSerre0}
\\
&\big[\bb_{i,0},[\bb_{i,0},\bb_{\tau i,0}]\big]  
= 4\bb_{i,0}, &c_{i,\tau i}=-1,\label{tcfSerre1}\\
&\big[\bb_{i,0},[\bb_{i,0},\bb_{j,0}]\big] =0, & c_{ij}=-1 \text{ and }i\ne \tau i\ne j, \label{tcfSerre2}\\
&\big[\bb_{i,0},[\bb_{i,0},\bb_{j,0}]\big] =-\bb_{j,0}, & c_{ij}=-1 \text{ and }i=\tau i,\label{tcfSerre3}
\end{align}
and possibly extra relations as follows.
\begin{enumerate}
\item If $\g$ is of type $\mathsf A_1$, then an additional relation is required:
\begin{align}
\big[\bh_{1,1},\big[\bb_{1,1},[\bh_{1,1},\bb_{1,1}]\big]\big]=0.\label{eq:extraA1-gr}
\end{align}
\item If $\g$ is of type $\mathsf A_2$, then two additional relations are required:
\begin{align}
&
\big[\bb_{1,1},[\bb_{1,1},\bb_{2,0}]\big]
+
\big[\bb_{2,1},[\bb_{2,1},\bb_{1,0}]\big]
=0,
\label{eq:A2-extra-Serre}
\\
&
\big[[\bh_{1,1},\bb_{1,1}],\bb_{2,1}\big]
-
\big[\bb_{1,0},
[\bh_{1,1},[\bh_{1,1},\bb_{2,1}]]
\big]
=0.
\label{eq:A2-extra-Cartan}
\end{align}
\item No additional relations are required for other types.
\end{enumerate}
\end{prop}

Theorem \ref{thm:min} and Proposition \ref{prop:red} are proved in Section \ref{sec:pf-thm1}. The case for type $\mathsf A_{2n}$ is treated separately in Appendix \ref{sec:pf-prop2}. We also strengthen \cite[Theorem~3.1]{Lu26min} for split iYangians by removing the extra relation for types $\mathsf B_2\cong\mathsf C_2$ and $\mathsf G_2$; see Corollary \ref{cor:split-extra-redundant-quantum}.

\subsection{Embedding into Yangians}
Let $\g$ be a finite-dimensional simple Lie algebra of type $\mathsf{ADE}$.

If $\g$ is not of type $\mathsf A_{2n}$, then define a map $\varphi:\Yi\to \Y$ by the rule:
\begin{align}
	h_{i,0}&\mapsto \xi_{i}-\xi_{\tau i},\label{hi0-embedding}\\
    h_{i,1}&\mapsto \xi_{i,1}+\xi_{\tau i,1}-\xi_{i}\xi_{\tau i}+\hf\sum_{\alpha\in\Phi^+}(\alpha,\alpha_i)\{x_\alpha^+,x_{\tau \alpha}^+\},\label{hi1-embedding}\\
    b_{i,0}&\mapsto x_{i}^+-x_{\tau i}^-,\label{bi0-embedding}\\
    b_{i,1}&\mapsto x_{i,1}^++x_{\tau i,1}^-+\tfrac12\sum_{\alpha\in\Phi^+}\big\{[x_{i}^+,x_{\alpha}^+],x_{\tau\alpha}^+\big\}-\tfrac12\{x_i^+,\xi_{\tau i}\}.\label{bi1-embedding}
\end{align}
Note that if we set $\tau =\mathrm{id}$, then one recovers the formulas \cite[(4.1)--(4.3)]{Lu26min}.

If $\g=\mathfrak{sl}_{N}$ where $N=2n+1$. Recall from \eqref{eq:rt-A2n0} our specific choice of root vectors for $\alpha=\ve_i-\ve_j$ with $1\lle i<j\lle N$:
\be
x_\alpha^+=(\sqrt{-1})^{j-i-1}\mathbf e_{ij},\qquad x_\alpha^-=(\sqrt{-1})^{i-j+1}\mathbf e_{ji},
\ee
and introduce
\beq\label{eq:Theta}
    \Theta_i=\begin{cases}
    -[x_{\beta_i}^+,x_{\tau\beta_i}^+]+[x_{\beta_{i+1}}^+,x_{\tau\beta_{i+1}}^+],&\text{ if }~1\lle i<n,\\
    -[x_{\beta_{\tau i}}^+,x_{\tau\beta_{\tau i}}^+]+[x_{\beta_{i'}}^+,x_{\tau\beta_{i'}}^+],&\text{ if }~n+1<i<N,\\
    \tfrac12(\xi_n-\xi_{n+1})-[x_n^+,x_{n+1}^+],&\text{ if }~i=n,n+1,
    \end{cases}
\eeq
where $\beta_i:=\ve_i-\ve_{n+1}$ for $1\lle i\lle n$ and $\tau\beta_i=\ve_{n+1}-\ve_{i'}$. Note that $\tau i=(i+1)'=i'-1$.

Define
\beq\label{eq:wpi-def}
\wp_i=\begin{cases}
    1,&\text{ if }N=2n+1,i=n,\\
    -1,&\text{ if }N=2n+1,i=n+1\\
    0,&\text{ otherwise.}
\end{cases}
\eeq
For every other quasi-split type, we set $\wp_i=0$ for all $i\in \I$.
We first set
\[
\begin{split}
\Psi^+_{i,<}&=\{\ve_k-\ve_i\mid 1\lle k<i\}
\cup \{\ve_k-\ve_{i+1}\mid 1\lle k\lle i\},\\
\Psi^+_{i,>}&=\{\ve_i-\ve_k\mid i<k\lle N\}
\cup \{\ve_{i+1}-\ve_k\mid i+1<k\lle N\}.
\end{split}
\]
We then define
\be
\begin{aligned}
\Phi^+_{i,<}&=\begin{cases}
\Psi^+_{i,<}\setminus\{\alpha_i\},&1\lle i<n,\\
(\Psi^+_{n,<}\setminus\{\alpha_n\})\cup\{\alpha_{n+1}\},&i=n,\\
\Psi^+_{n+1,<}\setminus\{\alpha_n\},&i=n+1,\\
\Psi^+_{i,<},&n+1<i<N,
\end{cases}\\[1ex]
\Phi^+_{i,>}&=\begin{cases}
\Psi^+_{i,>},&1\lle i<n,\\
\Psi^+_{n,>}\setminus\{\alpha_{n+1}\},&i=n,\\
(\Psi^+_{n+1,>}\setminus\{\alpha_{n+1}\})\cup\{\alpha_n\},&i=n+1,\\
\Psi^+_{i,>}\setminus\{\alpha_i\},&n+1<i<N.
\end{cases}
\end{aligned}
\ee
In particular,
\[
\Phi^+_{i,<}\cap\Phi^+_{i,>}=\varnothing,
\qquad
\Phi^+_{i,<}\sqcup\Phi^+_{i,>}
=\{\alpha\in\Phi^+\mid (\alpha,\alpha_i)\ne0\}.
\]

The root vectors are chosen so that the corresponding embedding formulas look similar to the formulas for the split case as in \cite{Lu26min}; see the proof below.

Define an assignment on the low-degree generators $\varphi:\Yi\to \Y$ by the rule:
\begin{align}
	h_{i,0}&\mapsto \xi_{i}-\xi_{\tau i}+\tfrac14\wp_i,\label{hi0-embedding2}\\
    h_{i,1}&\mapsto \xi_{i,1}+\xi_{\tau i,1}-\xi_{i}\xi_{\tau i}+\hf\sum_{\alpha\in\Phi^+}(\alpha,\alpha_i)\{x_\alpha^+,x_{\tau \alpha}^+\}+\hf\Theta_i,\label{hi1-embedding2}\\
    b_{i,0}&\mapsto x_{i}^+-x_{\tau i}^-,\label{bi0-embedding2}\\
    b_{i,1}&\mapsto x_{i,1}^{+}+x_{\tau i,1}^{-}
    -\tfrac12\acomm{\xi_{\tau i}}{x_i^{+}}+\hf\wp_i x_i^+ +\sum_{\alpha }
       \comm{x_i^{+}}{x_{\alpha}^{+}}x_{\tau\alpha}^{+}
    +\sum_{\beta }
       x_{\tau\beta}^{+}\comm{x_i^{+}}{x_{\beta}^{+}},\label{bi1-embedding2}
\end{align}
where $\alpha$ runs over all positive roots $\ve_k-\ve_i$ for $1\lle k<i$ while $\beta$ runs over all positive roots $\ve_{i+1}-\ve_k$ for $i+1<k\lle N$. 

Note that the formula \eqref{bi1-embedding2} can be slightly simplified as 
\be
-\tfrac12\acomm{\xi_{\tau i}}{x_i^{+}}+\hf\wp_i x_i^+=\begin{cases}
       -\xi_{n+1}x_n^+,&\text{ if }i=n,\\
       -x_{n+1}^+\xi_n,&\text{ if }i=n+1,\\
    -\tfrac12\acomm{\xi_{\tau i}}{x_i^{+}},&\text{ otherwise;}
       \end{cases}
\ee
while the formula \eqref{hi1-embedding2} can be simplified as
\beq\label{eq:hi1-simle}
\varphi(h_{i,1})={}\xi_{i,1}+\xi_{\tau i,1}
      -\xi_{i}\xi_{\tau i}+\tfrac14\wp_i(\xi_i-\xi_{\tau i})+\sum_{\alpha\in\Phi^+_{i,<}}(\alpha,\alpha_i)x_\alpha^+x_{\tau\alpha}^++\sum_{\alpha\in\Phi^+_{i,>}}(\alpha,\alpha_i)x_{\tau\alpha}^+x_\alpha^+,
\eeq
see the derivation from the R-matrix presentation using \cite{LZ24} in Appendix \ref{ssec:obtain-h}. 

\begin{rem}
One of the main reasons that the formulas look different from the non-$\mathsf A_{2n}$ case is that $[x_{\tau\alpha}^+,x_\alpha^+]$ could be nonzero as  $\alpha+\tau\alpha$ may still belong to $\Phi^+$ for some $\alpha\in\Phi^+$. In all the remaining quasi-split types (including obviously the split case $\tau=\mathrm{id}$ considered in \cite{Lu26min}), $\alpha+\tau\alpha\notin \Phi^+$  for all $\alpha\in\Phi^+$.    
\end{rem}

Our next main result identifies the iYangian $\Yi$ in the Drinfeld
presentation with the twisted Yangian $\YiJ$
in the $J$ presentation.

\begin{thm}\label{thm:embed}
Let $\g$ be a simple Lie algebra of type $\mathsf A_{r}$, $\mathsf D_{r}$ or $\mathsf E_6$. We have the following.
\begin{enumerate}
    \item The map $\varphi:\Yi\to \Y$ defined by \eqref{hi0-embedding2}--\eqref{bi1-embedding2} for type $\mathsf A_{2n}$ and by \eqref{hi0-embedding}--\eqref{bi1-embedding} for other types induces an injective algebra homomorphism from $\Yi$ to $\Y$.
\item Identifying $\Yi$ with a subalgebra of $\Y$ via $\varphi$, we have
\beq\label{coprohi1}
\begin{split}
\Delta(\tl h_{i,1})=\tl h_{i,1}\otimes 1+1\otimes \tl h_{i,1}+\tfrac1{32}\wp_i^2+\sum_{\alpha\in\Phi^+}(\alpha,\alpha_i+\alpha_{\tau i})(x_{\tau\alpha}^+-x_{\alpha}^-)\otimes x_\alpha^+.
\end{split}
\eeq
In particular, $\Yi$ is a right coideal subalgebra of $\Y$, i.e. $\Delta(\Yi)\subset \Yi\otimes \Y$.
\item If $\g$ is not of type $\mathsf A_{2n}$, then the quasi-split iYangian $\Yi$ is isomorphic to the twisted Yangian $\YiJ$ in the $J$ presentation. If $\g$ is  of type $\mathsf A_{2n}$, then $\Yi$ is isomorphic to the twisted Yangian $\YiJ(\mathscr K,\tfrac18)$ in the $J$ presentation, where $\mathscr K$ is given by \eqref{eq:K-def}.
\end{enumerate}
\end{thm}

Theorem \ref{thm:embed} is proved in Section \ref{sec:pfthmB0}. Again, in type $\mathsf A_{2n}$, Part (1) is treated separately using the R-matrix presentation in Appendix \ref{sec:app+}.

It is proved in \cite{CGM14} that the twisted Yangians of type AIII in the R-matrix and J presentations are isomorphic. Hence in type \(\mathsf A\), Part~(3) alternatively follows from that along with the identification between the R-matrix and Drinfeld realizations
established in \cite{LZ24}.

In terms of $B(\xi_i+\xi_{\tau i})$ from Definition \ref{def:tY-J}, we can rewrite \eqref{hi1-embedding} and \eqref{hi1-embedding2} as
\be
h_{i,1}\mapsto B(\xi_i+\xi_{\tau i})+\hf(\xi_i-\xi_{\tau i})^2+\tfrac{1}4\sum_{\alpha\in\Phi^+}(\alpha,\alpha_i)\big\{b_{\alpha}, b_{\tau \alpha}\big\}+\hf\Theta_i.
\ee
Here by convention $\Theta_i=0$ if $\g$ is not of type $\mathsf A_{2n}$.

We set
\beq\label{eq:gamma-def}
\gamma_{ij}=c_{ij}+c_{\tau i,j}.
\eeq
Then we have $\gamma_{ij}=\gamma_{ji}$.

The coproduct for $b_{i,1}$ can be explicitly computed using the relation 
\[
\Delta(b_{i,1})=\gamma_{ii}^{-1}\big[\Delta(\tl h_{i,1}),b_{i,0}\otimes 1+1\otimes b_{i,0}\big]
\]
as
\begin{equation*}
\Delta(b_{i,1})
=
b_{i,1}\otimes 1+1\otimes b_{i,1}
+
\gamma_{ii}^{-1}
\sum_{\alpha\in\Phi^+}
(\alpha,\alpha_i+\alpha_{\tau i})
\Big(
[b_{\tau\alpha},b_{i,0}]\otimes x_\alpha^+
+
b_{\tau\alpha}\otimes[x_\alpha^+,b_{i,0}]
\Big).
\end{equation*}
In particular, if $\g$ is not of type $\mathsf A_{2n}$, then using calculations similar to those in \S\ref{ssec:homo} we obtain
\beq\label{helper98}
\begin{split}
\Delta(b_{i,1})={}&b_{i,1}\otimes 1+1\otimes b_{i,1}+h_{i,0}\otimes x_{i}^+-b_{i,0}\otimes \xi_{\tau i}\\&-\sum_{\alpha\in\Phi^+}\Big(\big([x_{\tau\alpha}^+,x_i^+]-[x_{\tau i}^-,x_{\alpha}^-]\big)\otimes x_\alpha^++(x_\alpha^+-x_{\tau\alpha}^-)\otimes [x_{\tau\alpha}^+,x_i^+]\Big).
\end{split}
\eeq

\subsection{Coproduct estimates}
From now on, we regard the iYangian $\Yi$ as a right coideal subalgebra of the Yangian $\Y$ via the homomorphism $\varphi$ defined by \eqref{hi0-embedding2}--\eqref{bi1-embedding2} for type $\mathsf A_{2n}$ and by \eqref{hi0-embedding}--\eqref{bi1-embedding} for other types. 

Our next main result gives estimates for the generators $h_{i,r}$, $b_{i,r}$ in terms of the elements $\xi_{j,s}$ and $x_{j,s}^\pm$.

Set 
\begin{align*}
h_i(u)=1+\sum_{r\gge 0}h_{i,r}u^{-r-1},\qquad b_i(u)=\sum_{r\gge 0}b_{i,r}u^{-r-1}.
\end{align*}
\begin{thm}\label{thm:hb}
Let $\g$ be a simple Lie algebra of type $\mathsf{ADE}$. We have
\begin{align}
&h_i(u)\equiv \big(1+\tfrac{\wp_i}{4u}\big)\xi_i(u)\xi_{\tau i}(-u) &\pmod{\Y_{Q_+}[\![u^{-1}]\!]},\label{eq:h-est}\\
&\Delta(h_i(u))\equiv h_i(u)\otimes\xi_i(u)\xi_{\tau i}(-u)  &\pmod{\Yi\otimes \Y_{Q_+}[\![u^{-1}]\!]},\label{eq:h-co-est}\\
&\Delta(b_i(u))\equiv b_i(u)\otimes\xi_{\tau i}(-u)+1\otimes b_i(u) &\pmod{\Yi\otimes \Y_{Q_+}[\![u^{-1}]\!]}.\label{eq:b-co-est}
\end{align}
If $\g$ is not of type $\mathsf A_{2n}$, we have
\beq
b_i(u)\equiv 
\tfrac12 \{x_i^+(u),\xi_{\tau i}(-u)\}+x_{\tau i}^-(-u), 
\qquad \pmod{\Y_{\alpha_i+Q_+}^{\gge 0}[\![u^{-1}]\!]}.\label{eq:b-est}
\eeq
If $\g$ is of type $\mathsf A_{2n}$, then
\beq
b_i(u)\equiv \begin{cases}
\xi_{n+1}(-u)x_n^+(u)+x_{n+1}^-(-u),&\text{ if }i=n,\\
x_{n+1}^+(u)\xi_n(-u)+x_n^-(-u),&\text{ if }i=n+1,\\
\tfrac12 \{x_i^+(u),\xi_{\tau i}(-u)\}+x_{\tau i}^-(-u), &\text{ otherwise. }
\end{cases}   \pmod{\Y_{\alpha_i+Q_+}^{\gge 0}[\![u^{-1}]\!]}.\label{eq:b-est2}
\eeq
\end{thm}

Theorem \ref{thm:hb} is proved in Section \ref{sec:proof-thmC} with type $\mathsf A_{2n}$ included. The $q$-analogue of \eqref{eq:h-est}--\eqref{eq:h-co-est} for affine iquantum groups of quasi-split type $\mathsf A$ were obtained by Li--Prze\'zdziecki in \cite{LP26qchar}. A uniform extension to all split and quasi-split untwisted affine types is given by Lu--Pan in \cite[Theorem~3.18]{LuPan26Drinfeld}, using the framework of iHopf algebras. Their approach, however, does not carry over directly to the Yangian setting.

\subsection{Restriction of modules}
Let $\mathscr P_\I:=(1+u^{-1}\bC[\![u^{-1}]\!])^{\I}$ denote the set of $\I$-tuples of power series in $u^{-1}$ with constant term 1. We call an element of $\mathscr P_{\I}$ an $^\imath\bm\ell$-weight. We write an $^\imath\bm\ell$-weight in the form $\bla=(\la_i(u))_{i\in \I}$, 
where
\[
\la_i(u)=1+\sum_{r\gge 0}\la_{i,r}u^{-r-1}.
\]
For a finite-dimensional $\Yi$-module $\mc M$ and $\bla\in\mathscr P_{\I}$, the \textit{$^\imath\bm\ell$-weight space} of $^\imath\bm\ell$-weight $\bla$ is the subspace of $\mc M$ defined by
\[
\mc M_{\bla}:=\big\{v\in \mc M\mid \forall~ i\in\I \text{ and }r\gge 0, \exists~ p>0 \text{ such that } (h_{i,r}-\la_{i,r})^pv=0\big\}.
\]
If $\mc M_{\bla}\ne 0$, then we say $\bla$ is an $^\imath\bm\ell$-weight of $\mc M$.

The estimate \eqref{eq:h-est} can be used to determine the spectrum of $h_i(u)$ acting on a finite-dimensional $\Y$-module regarded as a $\Yi$-module by restriction. 

For a monic polynomial $P(u)$ in $u$, denote
\[
P^-(u)=(-1)^{\deg P}P(-u).
\]
Note that $P^-(u)$ is the monic polynomial whose roots are the opposites of the roots of $P(u)$. 

By \eqref{eq:h-est}, $h_i(u)$ and $(1+\frac{\wp_i}{4u})\xi_i(u)\xi_{\tau i}(-u)$, acting on a finite-dimensional $\Y$-module, share the same eigenvalues . The following is an immediate corollary of \cite[Theorem~1]{Kn95}. 

\begin{cor}\label{cor:cha}
Let $\g$ be a simple Lie algebra of type $\mathsf{ADE}$. Let $\mathscr V$ be a finite-dimensional $\Y$-module regarded as a $\Yi$-module by restriction. Then an $^\imath\bm\ell$-weight $\bm\la$ of $\mathscr V$ has the form
\[
\la_i(u)=\Big(1+\frac{\wp_i}{4u}\Big)\frac{\Xi_i(u+\hf )\,\Xi_{\tau i}^-(u-\hf )}{\Xi_i(u-\hf )\,\Xi_{\tau i}^-(u+\hf )},\qquad \text{ for }i\in \I,
\]
where $\Xi_i(u)$ is a monic polynomial in $u$.
\end{cor}

\section{Proof of Theorem \ref{thm:min}}\label{sec:pf-thm1}
\subsection{An auxiliary lemma}
We first recall the following lemma from \cite[Proposition~2.7]{LZ24}, which is closely related to Definition \ref{def:YN} and \eqref{assi}. 
\begin{lem}\label{lem:full-gr}
The algebra $\mathrm U(\g[z]^{\check\omega})$ is generated by $\sfh_{i,r},\sfb_{i,r}$ for $i\in \I,r\in\bN$, subject to the relations \eqref{eq:qsclassical1}--\eqref{eq:qsclassical0}, for $i,j\in\I$, $r,s\in \bN$:
\begin{align}
\label{eq:qsclassical1}
&[\sfh_{i,r},\sfh_{j,s}]=0,\qquad \sfh_{i,r} =(-1)^{r+1} \sfh_{\tau i,r},
\\
\label{eq:qsclassical2}
&[\sfh_{i,r},\sfb_{j,s}]=\big(c_{ij}-(-1)^r c_{\tau i,j} \big) \sfb_{j,r+s},
\\
\label{eq:qsclassical3}
&[\sfb_{i,r+1},\sfb_{j,s}]-[\sfb_{i,r},\sfb_{j,s+1}]=-2\delta_{j,\tau i} (-1)^r \sfh_{\tau i,r+s+1}, 
\\
&[\sfb_{i,r},\sfb_{\tau i,s}]=(-1)^r \sfh_{\tau i,r+s},\qquad\qquad \qquad  c_{i,\tau i}=0,\label{eq:qsclassical0}
\end{align}
and the Serre relations \eqref{eq:qss1}-\eqref{eq:qss4}: 
\begin{align}
&[\sfb_{i,0},\sfb_{j,0}]=\delta_{\tau i, j}\sfh_{j,0}, &c_{ij}=0,\label{eq:qss1}
\\
&\big[\sfb_{i,0},[\sfb_{i,0},\sfb_{\tau i,0}]\big]  
= 4\sfb_{i,0}, &c_{i,\tau i}=-1,\label{eq:qss2}\\
&\big[\sfb_{i,0},[\sfb_{i,0},\sfb_{j,0}]\big] =0, & c_{ij}=-1 \text{ and }i\ne \tau i\ne j, \label{eq:qss3}\\
&\big[\sfb_{i,0},[\sfb_{i,0},\sfb_{j,0}]\big] =-\sfb_{j,0}, & c_{ij}=-1 \text{ and }i=\tau i.\label{eq:qss4}
\end{align}
Here the elements $\sfb_{i,r}$ and $\sfh_{i,r}$ are identified as elements of $\g[z]^{\check\omega}$ by
\beq\label{eq:hb-explicit}
\sfh_{i,r} \mapsto (\xi_i-(-1)^r\xi_{\tau i})z^{r},\quad \sfb_{i,r}\mapsto (x_i^+-(-1)^rx_{\tau i}^-)z^r.
\eeq
\end{lem}

Recall $\gamma_{ij}$ from \eqref{eq:gamma-def}. We have the following simple observation.
\begin{lem}\label{lem:gener}
We have $$\sfb_{i,r+1}=\gamma_{ii}^{-1}[\sfh_{i,1},\sfb_{i,r}],\quad 
2\sfh_{i,r+1}=[\sfb_{\tau i,0},\sfb_{i,r+1}]-[\sfb_{\tau i,1},\sfb_{i,r}].
$$
If $\g$ is not of type $\mathsf A_{2n}$, then $\sfh_{i,r+1}=[\sfb_{\tau i,0},\sfb_{i,r+1}]$. In particular, the algebra $\mathrm{U}(\g[z]^{\check\omega})$ is generated by $\sfb_{i,0}$, $\sfh_{i,0}$ and $\sfh_{i,1}$.
\end{lem}

\begin{lem}\label{lem:add-}
We have the following relations:
\begin{enumerate}
    \item $\big[\sfh_{1,1},\big[\sfb_{1,1},[\sfh_{1,1},\sfb_{1,1}]\big]\big]=0$, if $\g$ is of type $\mathsf A_1$;
    \item $\big[\sfb_{1,1},[\sfb_{1,1},\sfb_{2,0}]\big]
+
\big[\sfb_{2,1},[\sfb_{2,1},\sfb_{1,0}]\big]
=0$ and $\big[[\sfh_{1,1},\sfb_{1,1}],\sfb_{2,1}\big]
-
\big[\sfb_{1,0},
[\sfh_{1,1},[\sfh_{1,1},\sfb_{2,1}]]
\big]
=0$, if $\g$ is of type $\mathsf A_2$.
\end{enumerate}
\end{lem}
\begin{proof}
It follows from a straightforward calculation by \eqref{eq:qsclassical1}--\eqref{eq:qsclassical0}.
\end{proof}

\subsection{Proof of Proposition \ref{prop:red}, Part I}\label{sec:pf-prop1}
Let $\mathcal U$ be the algebra generated by $\bb_{i,r}$ and $\bh_{i,r}$ for $r=0,1$ subject to the relations \eqref{tchh}--\eqref{tcfSerre3} along with 
\begin{itemize}
    \item the extra relation \eqref{eq:extraA1-gr} when $\g$ is of type $\mathsf A_{1}$,
    \item and the extra relations \eqref{eq:A2-extra-Serre}--\eqref{eq:A2-extra-Cartan} when $\g$ is of type $\mathsf A_{2}$.
\end{itemize}
Then it follows from Lemma \ref{lem:full-gr} and Lemma \ref{lem:add-} that there exists an algebra homomorphism
\beq\label{eq:psi}
\begin{split}
\psi:~&\mathcal U\to \mathrm{U}(\g[z]^{\check\omega})\\
& \bh_{i,r}\mapsto \sfh_{i,r},\qquad \bb_{i,r}\mapsto \sfb_{i,r} \qquad (r=0,1),
\end{split}
\eeq
which is surjective by Lemma \ref{lem:gener}. 

In the rest of this subsection, we exclude the case $\g=\mathfrak{sl}_{2n+1}$, which needs to be considered in a very different manner in Appendix \ref{sec:pf-prop2}.

Define
\beq\label{bhdef}
\bb_{i,r+1}=\gamma_{ii}^{-1}[\bh_{i,1},\bb_{i,r}],\quad \bh_{i,r+1}=[\bb_{\tau i,0},\bb_{i,r+1}]\qquad (r>0).
\eeq
To prove $\psi$ is an isomorphism, it suffices to prove that \eqref{eq:qsclassical1}--\eqref{eq:qsclassical0} can be deduced from \eqref{tchh}--\eqref{tcfSerre3} and the extra relation \eqref{eq:extraA1-gr} when $\g$ is of type $\mathsf A_1$. This should justify the existence of the inverse homomorphism of $\psi$. We establish this result by adapting the methods of \cite{Lev93gen,GNW18,Lu26min}, partitioning the proof into five main steps for clarity.

The relations \eqref{tcbb}--\eqref{tcbb2} imply that 
\be
[\bb_{\tau i,0},\bb_{i,1}]=\bh_{i,1}.
\ee
Hence \eqref{bhdef} also holds true if $r=0$.

\subsubsection*{\bf Step 1} We first derive some easy relations that will be used later.

Using an easy induction on $r$ by applying $\mathrm{ad}(\bh_{j,1})$ to \eqref{tch1b} recursively together with \eqref{tchh} and \eqref{bhdef}, we obtain
\beq\label{hi1bjr}[\bh_{i,0},\bb_{j,r}]=(c_{ij}-c_{\tau i,j})\bb_{j,r},\qquad [\bh_{i,1},\bb_{j,r}]=\gamma_{ij}\bb_{j,r+1} 
\eeq
for $i,j\in \I$ and $r\in\bN$.

\begin{lem}\label{lem:bb}
If $i,j$ are in distinct $\tau$-orbits, then the relation \eqref{eq:qsclassical3} holds:
\be
[\bb_{i,r+1},\bb_{j,s}]-[\bb_{i,r},\bb_{j,s+1}]=0.
\ee
\end{lem}
\begin{proof}
Let $\mathcal X(r,s):=[\bb_{i,r+1},\bb_{j,s}]-[\bb_{i,r},\bb_{j,s+1}]$. We prove by induction on $r+s$ that  $\mathcal X(r,s)=0$. 

The base case $r=s=0$ follows from \eqref{tcbb}. Suppose now $\mathcal X(r,s)=0$. Applying $\mathrm{ad}(\bh_{i,1})$ to $\mathcal X(r,s)=0$, we find that 
\beq\label{helper6}
\gamma_{ii}\mathcal X(r+1,s)+\gamma_{ij}\mathcal X(r,s+1)=0.
\eeq
Applying $\mathrm{ad}(\bh_{j,1})$ to $\mathcal X(r,s)=0$, we find that 
\beq\label{helper7}
\gamma_{ij}\mathcal X(r+1,s)+\gamma_{jj}\mathcal X(r,s+1)=0.
\eeq
Via case-by-case study, we find that the coefficient matrix of the linear system \eqref{helper6}--\eqref{helper7} is invertible and hence it has only trivial solution. Therefore, $\mathcal X(r+1,s)=\mathcal X(r,s+1)=0$.
\end{proof}

By \eqref{bhdef}, we have $[\bb_{\tau i,0},\bb_{i,1}]=\bh_{i,1}$ and $[\bb_{\tau i,0},\bb_{i,2}]=\bh_{i,2}$. Applying $\mathrm{ad}(\bh_{i,1})$ to the first equality and using \eqref{hi1bjr}, we find that
\beq\label{eq:fork-h2-bracket}
[\bb_{\tau i,1},\bb_{i,1}]+[\bb_{\tau i,0},\bb_{i,2}]=0\Longrightarrow [\bb_{i,1},\bb_{\tau i,1}]=\bh_{i,2}.
\eeq
In particular, 
\beq\label{eq:h2i-taui}
\bh_{\tau i,2}=[\bb_{\tau i,1},\bb_{i,1}]=-\bh_{i,2}.
\eeq
For $i\in\I_1$, relation \eqref{tcbb} gives
$[\bb_{i,0},\bb_{i,1}]=0$. Applying $\operatorname{ad}(\bh_{i,1})$
to this equality further gives $[\bb_{i,0},\bb_{i,2}]=0$.

\subsubsection*{\bf Step 2}
In this step, we assume that $\I_1$ contains at least two distinct elements $i,j$. Moreover, we can assume that $c_{ij}=-1$ and $c_{i,\tau j}=c_{\tau i,j}=c_{i,\tau i}=c_{j,\tau j}=0$.

\begin{lem}\label{lem:h2bj}
Let $i\in\I_1$ satisfy $c_{i,\tau i}=0$. Then we have $$[\bh_{i,2},\bb_{i,0}]=
 (c_{ii}-c_{i,\tau i})\bb_{i,2},\qquad 
 [\bh_{i,2},\bb_{k,r}]=
 (c_{ik}-c_{\tau i,k})\bb_{k,r+2},$$ 
 where $k\notin \{i,\tau i\}$.
\end{lem}
\begin{proof}
By the equalities above and $c_{\tau i,i}=0$, we have
\begin{align*}
 [\bh_{i,2},\bb_{i,0}]
 &=
 \bigl[[\bb_{i,1},\bb_{\tau i,1}],\bb_{i,0}\bigr]
 \\
 &=
 \bigl[\bb_{i,1},[\bb_{\tau i,1},\bb_{i,0}]\bigr]
 +
 \bigl[[\bb_{i,1},\bb_{i,0}],\bb_{\tau i,1}\bigr]
 \\
 &=
 \bigl[\bb_{i,1},[\bb_{\tau i,1},\bb_{i,0}]\bigr]=
 [\bb_{i,1},-\bh_{i,1}]=
(c_{ii}-c_{\tau  i,i})\bb_{i,2}.
\end{align*}
For the second identity, we compute
\begin{align*}
 [\bh_{i,2},\bb_{k,r}]
 &=
 \bigl[[\bb_{i,1},\bb_{\tau i,1}],\bb_{k,r}\bigr]
 \\
 &=
 \bigl[[\bb_{i,1},\bb_{k,r}],\bb_{\tau i,1}\bigr]
 +
 \bigl[\bb_{i,1},[\bb_{\tau i,1},\bb_{k,r}]\bigr]
 \\
 &=
 \bigl[[\bb_{i,0},\bb_{k,r+1}],\bb_{\tau i,1}\bigr]
 +
 \bigl[\bb_{i,1},[\bb_{\tau i,0},\bb_{k,r+1}]\bigr]
 \\
 &=
 [\bh_{\tau i,1},\bb_{k,r+1}]
 +
 \bigl[\bb_{i,0},[\bb_{k,r+1},\bb_{\tau i,1}]\bigr]\\
 &\qquad\qquad-
 [\bh_{\tau i,1},\bb_{k,r+1}]
 +
 \bigl[\bb_{\tau i,0},[\bb_{i,1},\bb_{k,r+1}]\bigr]
 \\
 &=
 \bigl[\bb_{i,0},[\bb_{k,r+1},\bb_{\tau i,1}]\bigr]
 +
 \bigl[\bb_{\tau i,0},[\bb_{i,1},\bb_{k,r+1}]\bigr]
 \\
 &=
 \bigl[\bb_{i,0},[\bb_{k,r+2},\bb_{\tau i,0}]\bigr]
 +
 \bigl[\bb_{\tau i,0},[\bb_{i,0},\bb_{k,r+2}]\bigr]
 \\
 &=
 -\bigl[[\bb_{i,0},\bb_{\tau i,0}],\bb_{k,r+2}\bigr]=
 -[\bh_{\tau i,0},\bb_{k,r+2}]
=
 (c_{ik}-c_{\tau i,k})\bb_{k,r+2}.
\end{align*}
This proves the lemma.
\end{proof}

\begin{lem}
We have $[\bb_{i,1},[\bb_{i,1},\bb_{j,0}]]=0$.
\end{lem}

\begin{proof}
For \(r_1,r_2,s\in\mathbb N\), set
\[
 \mathsf X(r_1,r_2\,|\,s)
 :=
 [\bb_{i,r_1},[\bb_{i,r_2},\bb_{j,s}]]
 +
 [\bb_{i,r_2},[\bb_{i,r_1},\bb_{j,s}]].
\]
Notice that $\mathsf X(r_1,r_2\,|\,s)
 =
 \mathsf X(r_2,r_1\,|\,s)$.

Applying $\mathrm{ad}(\bh_{i,1})$ and
$\mathrm{ad}(\bh_{j,1})$ to $\mathsf X(0,0\,|\,0)=0$,
respectively, we obtain as in the proof of Lemma \ref{lem:bb} that
$\mathsf X(1,0\,|\,0)
 =
 \mathsf X(0,0\,|\,1)
 =
 0$. Similarly, by applying $\mathrm{ad}(\bh_{i,1})$ and
$\mathrm{ad}(\bh_{j,1})$ to
\(\mathsf X(1,0\,|\,0)=0\), we have
\beq\label{eq:X-intermediate}
 \mathsf X(2,0\,|\,0)+\mathsf X(1,1\,|\,0)=0,
 \qquad
 \mathsf X(1,0\,|\,1)=0. 
\eeq

We next apply $\mathrm{ad}(\bh_{i,2})$ and
$\mathrm{ad}(\bh_{j,2})$ to
\(\mathsf X(0,0\,|\,0)=0\). Using Lemma \ref{lem:h2bj}, we obtain
\begin{align*}
 &(c_{ii}-c_{i,\tau i})\mathsf X(2,0\,|\,0)
 +(c_{ii}-c_{i,\tau i})\mathsf X(0,2\,|\,0)
 +(c_{ij}-c_{\tau i,j})\mathsf X(0,0\,|\,2)
 =0,
 \\
 &(c_{ji}-c_{\tau j,i})\mathsf X(2,0\,|\,0)
 +(c_{ji}-c_{\tau j,i})\mathsf X(0,2\,|\,0)
 +(c_{jj}-c_{\tau j,j})\mathsf X(0,0\,|\,2)
 =0.
\end{align*}
Note that $\mathsf X(2,0\,|\,0)=\mathsf X(0,2\,|\,0)$. The same approach in the proof of Lemma \ref{lem:bb} implies that $\mathsf X(2,0\,|\,0)=0$, which by \eqref{eq:X-intermediate} further completes the proof.
\end{proof}

Now we move to the next relation. Note that if $i,j\in \I_1$, $i\ne j$, and $c_{i,\tau j}=0$, then $[\bb_{i,1},\bb_{\tau j,1}]=0$ follows easily from \eqref{tcfSerre0} and Lemma \ref{lem:bb} by repeatedly applying $\mathrm{ad}(\bh_{i,1})$.
\begin{lem}\label{lem:hi1hi2-com}
Under the assumption that $i,j\in \I_1$ and $c_{ij}=-1$, we have
\begin{equation*}
 [\mathbf h_{i,1},\mathbf h_{i,2}]
 =
 [\mathbf h_{i,1},
   [\mathbf b_{i,1},\mathbf b_{\tau i,1}]]
 =0.
\end{equation*}
\end{lem}

\begin{proof}
By the preceding lemma, we have
$0=
 [\mathbf b_{i,1},
  [\mathbf b_{i,1},\mathbf b_{j,0}]]$.
Applying the operation
\([\,\cdot\,,\mathbf b_{\tau j,1}]\)
to it, and using the relevant commutation
relations, gives
\begin{equation}
 0=
 [\mathbf b_{i,1},
  [\mathbf b_{i,1},\mathbf h_{j,1}]].
 \label{eq:bi-bi-hj}
\end{equation}
We next apply
\( [\,\cdot\,,\mathbf b_{\tau i,0} ]\)
to \eqref{eq:bi-bi-hj}. By the Jacobi identity, we obtain
\begin{align}
 0
 =
 -[\mathbf h_{\tau i,1},
   [\mathbf b_{i,1},\mathbf h_{j,1}]]
 -
 [\mathbf b_{i,1},
   [\mathbf h_{\tau i,1},\mathbf h_{j,1}]]
 -
 [\mathbf b_{i,1},
   [\mathbf b_{i,1},\mathbf b_{\tau i,1}]].
 \label{eq:first-ad-btaui0}
\end{align}
Since the Cartan generators commute, the middle term in
\eqref{eq:first-ad-btaui0} vanishes and we find
\begin{equation}
 0=
 [\mathbf h_{\tau i,1},
  [\mathbf b_{i,1},\mathbf h_{j,1}]]
 +
 [\mathbf b_{i,1},
  [\mathbf b_{i,1},\mathbf b_{\tau i,1}]].
 \label{eq:intermediate-Cartan}
\end{equation}

Applying
\(\bigl[\,\cdot\,,\mathbf b_{\tau i,0}\bigr]\)
once again to \eqref{eq:intermediate-Cartan}, and repeatedly using the
Jacobi identity, gives
\begin{align}
 0
 ={}&
 2[\mathbf b_{\tau i,1},
   [\mathbf b_{i,1},\mathbf h_{j,1}]]
 -
 [\mathbf h_{\tau i,1},
   [\mathbf h_{\tau i,1},\mathbf h_{j,1}]]
 \notag\\
 & 
 -
 [\mathbf h_{\tau i,1},
   [\mathbf b_{i,1},\mathbf b_{\tau i,1}]]
 -
 [\mathbf h_{\tau i,1},
   [\mathbf b_{i,1},\mathbf b_{\tau i,1}]]
 -
 [\mathbf b_{i,1},
   [\mathbf h_{\tau i,1},\mathbf b_{\tau i,1}]].
 \label{eq:second-ad-btaui0}
\end{align}
Again,
\(
 [\mathbf h_{\tau i,1},\mathbf h_{j,1}]=0
\),
so the second term in \eqref{eq:second-ad-btaui0} vanishes.

By the Cartan action relations, we have
\begin{equation}
 2[\mathbf b_{\tau i,1},
   [\mathbf b_{i,1},\mathbf h_{j,1}]]
 =
 -[\mathbf b_{\tau i,1},
   [\mathbf b_{i,1},\mathbf h_{\tau i,1}]].
 \label{eq:Cartan-action-comparison}
\end{equation}
Substituting \eqref{eq:Cartan-action-comparison} into
\eqref{eq:second-ad-btaui0}, we obtain
\begin{align}
 0
 =
 -[\mathbf b_{\tau i,1},
   [\mathbf b_{i,1},\mathbf h_{\tau i,1}]]
 -
 2[\mathbf h_{\tau i,1},
   [\mathbf b_{i,1},\mathbf b_{\tau i,1}]]
 -
 [\mathbf b_{i,1},
   [\mathbf h_{\tau i,1},\mathbf b_{\tau i,1}]].
 \label{eq:three-Jacobi-terms}
\end{align}
Using the Jacobi identity, \eqref{eq:three-Jacobi-terms} becomes
\begin{equation*}
 0=
 -3[\mathbf h_{\tau i,1},
    [\mathbf b_{i,1},\mathbf b_{\tau i,1}]].
\end{equation*}
It follows that $[\mathbf h_{\tau i,1},
  [\mathbf b_{i,1},\mathbf b_{\tau i,1}]]
 =0$.

Finally, using
\(
 \mathbf h_{\tau i,1}=\mathbf h_{i,1}
\)
and
\(
 [\mathbf b_{i,1},\mathbf b_{\tau i,1}]
 =\mathbf h_{i,2}
\),
we conclude that $[\mathbf h_{i,1},\mathbf h_{i,2}]=0$,
completing the proof of  the lemma.
\end{proof}

\begin{lem}\label{lem:h2bj-2}
If $i\in \I_1$ and $[\bh_{i,1},\bh_{i,2}]=0$, then we have $[\bh_{i,2},\bb_{i,r}]=
 (c_{ii}-c_{i,\tau i})\bb_{i,r+2}$ and $[\bh_{i,2},\bb_{\tau i,r}]=
 (c_{\tau i,i}-c_{ii})\bb_{\tau i,r+2}$ for $r\in\bN$.
\end{lem}
\begin{proof}
The first equality immediately follows from Lemma \ref{lem:h2bj} and $[\bh_{i,1},\bh_{i,2}]=0$ by repeatedly commuting with $\bh_{i,1}$. The second equality then follows from the fact that $\bh_{\tau i,2}=-\bh_{i,2}$ established in \eqref{eq:h2i-taui}.
\end{proof}

\subsubsection*{\bf Step 3}
Now we consider the case that $\I_1$ contains exactly one element. This happens only for the case of type $\mathsf A_2$, $\mathsf A_3$, and $\mathsf D_r$ for $r\gge 4$. Note we only consider the latter two cases. 
\begin{lem}\label{lem:extraD-redundant-gr}
Suppose that $i\in\I_1$, $j\in\I_0$, and
\begin{equation}\label{eq:common-fixed-neighbor}
 c_{i,\tau i}=0,\qquad
 c_{ij}=c_{\tau i,j}=c_{ji}=c_{j,\tau i}=-1.
\end{equation}
Then the relations \eqref{tchh}--\eqref{tcfSerre3}, without
\eqref{eq:extraD-gr}, imply
\begin{equation}\label{eq:extraD-gr}
\big[\bh_{i,1},[\bb_{i,1},\bb_{\tau i,1}]\big]=0.
\end{equation}
The hypotheses hold for the unique nontrivial $\tau$-orbit in type
$\mathsf A_3$ and in type $\mathsf D_r$.
\end{lem}
\begin{proof}
We claim that $\bh_{i,2}$ commutes with every $\bb_{j,r}$.  Using
\eqref{eq:fork-h2-bracket}, the Jacobi identity, and then
Lemma \ref{lem:bb}, we compute
\begin{align*}
 [\bh_{i,2},\bb_{j,r}]
 &=
 [[\bb_{i,1},\bb_{\tau i,1}],\bb_{j,r}]
 \\
 &=
 [[\bb_{i,0},\bb_{j,r+1}],\bb_{\tau i,1}]
 +
 [\bb_{i,1},[\bb_{\tau i,0},\bb_{j,r+1}]]
 \\
 &=
 [\bb_{i,0},[\bb_{j,r+1},\bb_{\tau i,1}]]
 +
 [\bb_{\tau i,0},[\bb_{i,1},\bb_{j,r+1}]]
 \\
 &=
 [\bb_{i,0},[\bb_{j,r+2},\bb_{\tau i,0}]]
 +
 [\bb_{\tau i,0},[\bb_{i,0},\bb_{j,r+2}]]
 \\
 &=
 -[[\bb_{i,0},\bb_{\tau i,0}],\bb_{j,r+2}]
 \\
 &=
 -[\bh_{\tau i,0},\bb_{j,r+2}]
 =
 (c_{ij}-c_{\tau i,j})\bb_{j,r+2}
 =
 0.
\end{align*}
In the third equality, the two terms containing
$[\bh_{\tau i,1},\bb_{j,r+1}]$ cancel.  Thus $[\bh_{i,2},\bb_{j,r}]=0$ for $r\in\bN$.

Since $j=\tau j$, we have $\bh_{j,1}=[\bb_{j,0},\bb_{j,1}]$. It follows from the claim that 
\begin{equation}\label{eq:h2-J-commute}
 [\bh_{i,2},\bh_{j,1}]
 =
 [\bh_{i,2},[\bb_{j,0},\bb_{j,1}]]
 =
 0.
\end{equation}
By \eqref{hi1bjr} and
\eqref{eq:common-fixed-neighbor}, we have
\begin{align*}
 [\bh_{i,1}+\bh_{j,1},\bb_{i,1}]
 &=
 (\gamma_{ii}+\gamma_{ji})\bb_{i,2}
 =
 (2-2)\bb_{i,2}=0,
 \\
 [\bh_{i,1}+\bh_{j,1},\bb_{\tau i,1}]
 &=
 (\gamma_{i,\tau i}+\gamma_{j,\tau i})
 \bb_{\tau i,2}
 =
 (2-2)\bb_{\tau i,2}=0.
\end{align*}
Together with \eqref{eq:fork-h2-bracket}, this yields
\[
[\bh_{i,1}+\bh_{j,1},\bh_{i,2}]
 =
 [\bh_{i,1}+\bh_{j,1},[\bb_{i,1},\bb_{\tau i,1}]]
 =
 0.
\]
Subtracting \eqref{eq:h2-J-commute}, we obtain $[\bh_{i,1},\bh_{i,2}]=0$.
Equation \eqref{eq:extraD-gr} now follows from
\eqref{eq:fork-h2-bracket}.
\end{proof}

\subsubsection*{\bf Step 4}
Then, with the previous lemmas, we are ready to deduce the higher-degree
relations for the vertices in $\I_1$. Fix $i\in\I_1$, so that
$i\neq\tau i$ and $c_{i,\tau i}=0$. Note that
$[\bh_{i,1},\bh_{i,2}]=0$ has already been established in Step 2 and Step 3. We prove by
induction on $\ell$ that
\begin{align}
[\bh_{i,r},\bh_{i,s}]&=[\bh_{i,r},\bh_{\tau i,s}]=0,
&\bh_{i,p}=(-1)^{p+1}\bh_{\tau i,p},
\label{todo1-I1}
\\
[\bh_{i,r},\bb_{j,s}]&=
\big(c_{ij}-(-1)^r c_{\tau i,j}\big)\bb_{j,r+s},
&j\in\{i,\tau i\},
\label{todo2-I1}
\\
[\bb_{i,r},\bb_{i,s}]&=[\bb_{\tau i,r},\bb_{\tau i,s}]=0,
\label{todo3-I1}
\\
[\bb_{i,r},\bb_{\tau i,s}]&=(-1)^r\bh_{\tau i,r+s}.
\label{todo4-I1}
\end{align}
where $r+s\lle \ell$ and $p\lle \ell$. This verifies \eqref{eq:qsclassical1}--\eqref{eq:qsclassical0} for the case $i\in \I_1$ and $j\in \{i,\tau i\}$.

We first prove \eqref{todo3-I1}.

\begin{lem}\label{lem:bibi-zero-I1}
For every $i\in\I_1$ and $r,s\in\bN$, one has
\begin{equation*}
[\bb_{i,r},\bb_{i,s}]=0=[\bb_{\tau i,r},\bb_{\tau i,s}].
\end{equation*}
\end{lem}
\begin{proof}
We prove the first assertion by induction on $\ell=r+s$. The degree-one
relation gives
\begin{align*}
0
&=[\bb_{i,1},\bb_{i,0}]-[\bb_{i,0},\bb_{i,1}]
=-2[\bb_{i,0},\bb_{i,1}],
\end{align*}
so that $[\bb_{i,0},\bb_{i,1}]=0$. Applying
$\mathrm{ad}(\bh_{i,1})$ and using \eqref{hi1bjr}, we obtain
\begin{align*}
0
&=[\bh_{i,1},[\bb_{i,0},\bb_{i,1}]]
=2[\bb_{i,0},\bb_{i,2}],
\end{align*}
whence $[\bb_{i,0},\bb_{i,2}]=0$. These are the required initial cases.

Suppose now that $\ell\gge 2$ and that
$[\bb_{i,r},\bb_{i,\ell-r}]=0$ for every $0\lle r\lle\ell$. We prove
that
\begin{equation}
[\bb_{i,r},\bb_{i,\ell+1-r}]=0
\qquad (0\lle r\lle \ell+1).
\label{bibi-next-I1}
\end{equation}
Applying $\mathrm{ad}(\bh_{i,1})$ to the induction hypothesis and using
\eqref{hi1bjr}, we obtain
\begin{equation}
[\bb_{i,r},\bb_{i,\ell+1-r}]
+[\bb_{i,r+1},\bb_{i,\ell-r}]=0
\qquad (0\lle r\lle\ell).
\label{bibi-chain-I1}
\end{equation}
In particular,
\begin{align}
0=[\bb_{i,0},\bb_{i,\ell+1}]
  +[\bb_{i,1},\bb_{i,\ell}]=[\bb_{i,1},\bb_{i,\ell}]
  +[\bb_{i,2},\bb_{i,\ell-1}].
\label{bibi-A-I1}
\end{align}

On the other hand, applying $\mathrm{ad}(\bh_{i,2})$ to
$[\bb_{i,0},\bb_{i,\ell-1}]=0$ and Lemma \ref{lem:h2bj-2}, we get
\begin{equation}
0=[\bb_{i,0},\bb_{i,\ell+1}]
  +[\bb_{i,2},\bb_{i,\ell-1}].
\label{bibi-C-I1}
\end{equation}
Equation \eqref{bibi-A-I1}  implies that $[\bb_{i,0},\bb_{i,\ell+1}]
=[\bb_{i,2},\bb_{i,\ell-1}]$.
Together with equation \eqref{bibi-C-I1}, this gives
$[\bb_{i,0},\bb_{i,\ell+1}]=0$. Equation
\eqref{bibi-chain-I1} then successively implies
\eqref{bibi-next-I1} for every $0\lle r\lle\ell+1$. This completes the
induction. Replacing $i$ by $\tau i$ proves the second assertion.
\end{proof}

We next prove the remaining relations \eqref{todo1-I1}--\eqref{todo4-I1}, using induction on $\ell=r+s$.  The cases of total
degree $0,1,2$ follow from the defining relations and the lower-degree
relations established in the previous steps. Suppose that the assertions are known in all
total degrees strictly smaller than $\ell$.

For the inductive step at total degree $\ell$, we proceed in two
stages. First, assuming \eqref{todo4-I1} in total degree $\ell$,
we show that it implies \eqref{todo1-I1} and \eqref{todo2-I1}
in the same total degree. We then prove \eqref{todo4-I1} in total
degree $\ell$ using only relations of strictly smaller total degree
and the already established relation \eqref{todo3-I1}. This closes
the induction.

Assume for the moment that \eqref{todo4-I1} holds in total degree
$\ell$. Then
\begin{align*}
\bh_{\tau i,\ell}
=[\bb_{i,0},\bb_{\tau i,\ell}]
=-[\bb_{\tau i,\ell},\bb_{i,0}]
=(-1)^{\ell+1}\bh_{i,\ell},
\end{align*}
proving the symmetry relation in
\eqref{todo1-I1}. We now prove the commutativity relation in \eqref{todo1-I1}. We may
assume $0\lle r\lle s$. The cases of low total degree have already been
established, so it remains to consider $s\gge 2$. First suppose that
$r=0$. By \eqref{todo4-I1},
$\bh_{i,s}=-[\bb_{\tau i,1},\bb_{i,s-1}]$, and therefore by \eqref{hi1bjr} we have
\begin{align*}
[\bh_{i,0},\bh_{i,s}]
&=-[\bh_{i,0},[\bb_{\tau i,1},\bb_{i,s-1}]]
\\
&=2[\bb_{\tau i,1},\bb_{i,s-1}]
  -2[\bb_{\tau i,1},\bb_{i,s-1}]
=0.
\end{align*}
Suppose next that $r\gge 1$. Using the induction hypothesis for
\eqref{todo2-I1} and then \eqref{todo4-I1}, we obtain
\begin{align*}
[\bh_{i,r},\bh_{i,s}]
&=-[\bh_{i,r},[\bb_{\tau i,1},\bb_{i,s-1}]]
\\
&=2(-1)^r[\bb_{\tau i,r+1},\bb_{i,s-1}]
  -2[\bb_{\tau i,1},\bb_{i,r+s-1}]
\\
&=2(-1)^r(-1)^{r+1}
  [\bb_{\tau i,0},\bb_{i,r+s}]
  +2[\bb_{\tau i,0},\bb_{i,r+s}]
=0.
\end{align*}
Thus $[\bh_{i,r},\bh_{i,s}]=0$ for all $r,s\in\bN$.
Together with $\bh_{i,p}=(-1)^{p+1}\bh_{\tau i,p}$, this also gives all Cartan
commutativity relations supported on the orbit $\{i,\tau i\}$.

For \eqref{todo2-I1}, if $r<\ell$, then the induction hypothesis gives $[\bh_{i,r},\bb_{j,0}]
=\big(c_{ij}-(-1)^r c_{\tau i,j}\big)\bb_{j,r}.$
Repeatedly applying $\mathrm{ad}(\bh_{i,1})$ to
it, and using the Cartan commutativity proved above, we
obtain \eqref{todo2-I1}. If $r=\ell$, then
\begin{align*}
[\bh_{i,\ell},\bb_{i,0}]&=-\big[[\bb_{\tau i,1},\bb_{i,\ell-1}],\bb_{i,0}\big]=-\big[[\bb_{\tau i,1},\bb_{i,0}],\bb_{i,\ell-1}\big]\\&=[\bh_{i,1},\bb_{i,\ell-1}]=(c_{ii}-(-1)^\ell c_{\tau i,i})\bb_{i,\ell},
\end{align*}
completing the proof of 
\eqref{todo2-I1} by Cartan symmetry.

It remains to prove \eqref{todo4-I1} in total degree $\ell$
without using \eqref{todo1-I1} or \eqref{todo2-I1} in that same degree.

Assume first that $\ell=2l+1$. By the induction hypothesis, for every
$0\lle r\lle 2l$ one has
\begin{equation}
\bh_{i,2l}=(-1)^r
[\bb_{\tau i,r},\bb_{i,2l-r}].
\label{cross-ind-odd-I1}
\end{equation}
Applying $\mathrm{ad}(\bh_{i,1})$ to
\eqref{cross-ind-odd-I1}, we obtain
\begin{align}
[\bh_{i,1},\bh_{i,2l}]
=2(-1)^r\big(
[\bb_{\tau i,r},\bb_{i,2l+1-r}]
+[\bb_{\tau i,r+1},\bb_{i,2l-r}]
\big).
\label{cross-rec-odd-I1}
\end{align}
Set
\begin{equation*}
\mathfrak X=[\bb_{\tau i,0},\bb_{i,2l+1}],
\qquad
2\mathfrak Y=[\bh_{i,1},\bh_{i,2l}].
\end{equation*}
It follows recursively from \eqref{cross-rec-odd-I1} that
\begin{equation}
[\bb_{\tau i,r},\bb_{i,2l+1-r}]
=(-1)^r(\mathfrak X-r\mathfrak Y)
\qquad (0\lle r\lle 2l+1).
\label{cross-linear-odd-I1}
\end{equation}
On the other hand, using
$\bh_{i,1}=[\bb_{\tau i,0},\bb_{i,1}]$, the induction hypothesis for
\eqref{todo2-I1}, and \eqref{cross-linear-odd-I1}, we obtain
\begin{align*}
2\mathfrak Y
&=[\bh_{i,1},\bh_{i,2l}]
 =[[\bb_{\tau i,0},\bb_{i,1}],\bh_{i,2l}]
\\
&=2[\bb_{\tau i,2l},\bb_{i,1}]
  -2[\bb_{\tau i,0},\bb_{i,2l+1}]=2(\mathfrak X-2l\mathfrak Y)-2\mathfrak X.
\end{align*}
Here we also used $[\bh_{i,2l},\bb_{i,1}]=2\bb_{i,2l+1}$ that follows from the induction
hypothesis for \eqref{todo4-I1}, together with
\eqref{todo3-I1}:
\begin{align*}
[\bh_{i,2l},\bb_{i,1}]
&=
\big[[\bb_{\tau i,0},\bb_{i,2l}],\bb_{i,1}\big]
\\
&=
-\big[\bb_{i,2l},
      [\bb_{\tau i,0},\bb_{i,1}]\big]
=
-\big[\bb_{i,2l},\bh_{i,1}\big]
=
2\bb_{i,2l+1}.
\end{align*}
Consequently, $(4l+2)\mathfrak Y=0$, and hence
$\mathfrak Y=0$. This proves the assertion in total degree $2l+1$.

We next consider $\ell=2l+2$. The same argument, now starting from the
already established total degree $2l+1$, gives
\begin{equation*}
[\bb_{\tau i,r},\bb_{i,2l+2-r}]
=(-1)^r(\mathfrak X-r\mathfrak Y),
\end{equation*}
where
\begin{equation*}
\mathfrak X=[\bb_{\tau i,0},\bb_{i,2l+2}],
\qquad
2\mathfrak Y=[\bh_{i,1},\bh_{i,2l+1}].
\end{equation*}

Note that by induction hypothesis, we have $[\bh_{i,l+1},\bb_{i,0}]=2\bb_{i,l+1}$ and $[\bh_{i,l+1},\bh_{i,1}]=0$. Therefore, by repeatedly applying $\mathrm{ad}(\bh_{i,1})$ to the former equality, we obtain $[\bh_{i,l+1},\bb_{i,l+1}]=2\bb_{i,2l+2}$. Using the induction hypothesis and the definition
$\bh_{i,l+1}=[\bb_{\tau i,0},\bb_{i,l+1}]$, we have
\begin{align*}
0
&=[\bh_{i,l+1},\bh_{i,l+1}]
 =[\bh_{i,l+1},[\bb_{\tau i,0},\bb_{i,l+1}]]
\\
&=2(-1)^l[\bb_{\tau i,l+1},\bb_{i,l+1}]
  +2[\bb_{\tau i,0},\bb_{i,2l+2}]
\\
&=2(-1)^l(-1)^{l+1}
  \big(\mathfrak X-(l+1)\mathfrak Y\big)
  +2\mathfrak X=2(l+1)\mathfrak Y.
\end{align*}
Thus $\mathfrak Y=0$, proving the assertion in total degree $2l+2$. We have therefore proved \eqref{todo4-I1} in total degree $\ell$.
The deductions above now imply \eqref{todo1-I1} and
\eqref{todo2-I1} in the same total degree. Hence all the required
relations hold in total degree $\ell$, and the induction is complete.
In particular,
\[
[\bb_{\tau i,r},\bb_{i,\ell-r}]
=
(-1)^r[\bb_{\tau i,0},\bb_{i,\ell}]
=
(-1)^r\bh_{i,\ell}.
\]

\subsubsection*{\bf Step 5}
Now we are ready to finalize the proof of Proposition \ref{prop:red}. We start with the following lemma.

\begin{lem}\label{lem:com-transit}
Suppose $i\in \I_1$, $j\in \I_0$, and $c_{ij}=-1$. If $[\bh_{i,1},\bh_{i,2}]=0$, then $[\bh_{j,1},\bh_{j,3}]=0$.
\end{lem}
\begin{proof}
If $[\bh_{i,1},\bh_{i,2}]=0$, then it follows from the discussion in Step 4 that $[\bh_{i,1},\bh_{i,3}]=0$. Moreover, we have $[\bh_{i,3},\bb_{j,0}]=2c_{ij}\bb_{j,3}$ by the same proof as in Lemma \ref{lem:h2bj}. We have
\begin{align*}
0&\,\,\,=\,\,[\bh_{i,1},\bh_{i,3}]\stackrel{\eqref{todo4-I1}}{=}\big[\bh_{i,1},[\bb_{\tau i,2},\bb_{i,1}]\big]=\big[[\bh_{i,1},\bb_{\tau i,2}],\bb_{i,1}\big]+\big[\bb_{\tau i,2},[\bh_{i,1},\bb_{i,1}]\big]\\
&\stackrel{\eqref{hi1bjr}}{=}\frac{c_{ii}}{2c_{ij}}\Big(\big[[\bh_{j,1},\bb_{\tau i,2}],\bb_{i,1}\big]+\big[\bb_{\tau i,2},[\bh_{j,1},\bb_{i,1}]\big]\Big)\\
&\,\,\,=\,\,\frac{c_{ii}}{2c_{ij}}\big[\bh_{j,1},[\bb_{\tau i,2},\bb_{i,1}]\big]=\frac{c_{ii}}{2c_{ij}}[\bh_{j,1},\bh_{i,3}].
\end{align*}
Thus $[\bh_{j,1},\bh_{i,3}]=0$. Applying $[\,\cdot\,,\bb_{j,0}]$ to it and then using \eqref{hi1bjr}, we obtain
\[
c_{jj}[\bb_{j,1},\bh_{i,3}]+c_{ij}[\bh_{j,1},\bb_{j,3}]=0.
\]
Applying $\mathrm{ad}(\bb_{j,0})$ to the above equation, a similar calculation using $[\bh_{j,1},\bh_{i,3}]=0$ shows that
\begin{align*}
-4c_{ij}c_{jj}[\bb_{j,1},\bb_{j,3}]+c_{ij}[\bh_{j,1},\bh_{j,3}]=0.
\end{align*}
Note that \cite[(3.35)]{Lu26min} implies that $[\bh_{j,1},\bh_{j,3}]=2(\alpha_j,\alpha_j)[\bb_{j,3},\bb_{j,1}]$. Combining this with the above equation implies $[\bh_{j,1},\bh_{j,3}]=0$.
\end{proof}

We first observe from the previous steps that we always have
\begin{align}
&[\bh_{i,1},\bh_{i,2}]=0, &\text{ for all }i\in \I_1,\label{eq:hi1hi2-commm}\\
&[\bh_{j,1},\bh_{j,3}]=0, &\text{ for all }j\in \I_0.\label{eq:hj1hj3-commm}
\end{align}
Indeed, 
if $\I_1$ has at least two elements, then \eqref{eq:hi1hi2-commm} follows from Lemma \ref{lem:hi1hi2-com} as $\I_1$ is connected. 
If $\I_0$ has at least two elements, then \eqref{eq:hj1hj3-commm} follows from \cite[Lemmas 3.9, 3.11]{Lu26min}. Lemma \ref{lem:com-transit} indicates if $c_{ij}=-1$, then \eqref{eq:hi1hi2-commm} implies \eqref{eq:hj1hj3-commm}. Thus, we have the desired commutativity in \eqref{eq:hi1hi2-commm}--\eqref{eq:hj1hj3-commm} for type $\mathsf A_{2r+1}$ for $r\gge 2$, $\mathsf E_6$. For type $\mathsf A_1$, the commutativity \eqref{eq:hj1hj3-commm} is obtained by enforcing the extra relation $\big[\bh_{1,1},\big[\bb_{1,1},[\bh_{1,1},\bb_{1,1}]\big]\big]=0$. For type $\mathsf A_3$ and $\mathsf D_r$, the relation \eqref{eq:hi1hi2-commm} follows from Lemma \ref{lem:extraD-redundant-gr} and its proof while the relation  \eqref{eq:hj1hj3-commm} is obtained by Lemma \ref{lem:com-transit}.

Due to \eqref{eq:hi1hi2-commm}--\eqref{eq:hj1hj3-commm}, it follows from Step 4 and \cite[\S3.3]{Lu26min} that \eqref{eq:qsclassical1}--\eqref{eq:qsclassical0} hold for $i,j$ in the same $\tau$-orbit. It suffices to prove \eqref{eq:qsclassical1}--\eqref{eq:qsclassical3} for $i,j$ in different $\tau$-orbits. In this case, \eqref{eq:qsclassical3} follows from Lemma \ref{lem:bb}. It remains to prove \eqref{eq:qsclassical1}--\eqref{eq:qsclassical2}. 

Now suppose $i,j$ are in different $\tau$-orbits. 
If $i\in \I_1$, then we have
\begin{align*}
 [\bh_{i,r},\bb_{j,s}]
 &=
 \bigl[[\bb_{i,r-1},\bb_{\tau i,1}],\bb_{j,s}\bigr]
 \\
 &=
 \bigl[[\bb_{i,r-1},\bb_{j,s}],\bb_{\tau i,1}\bigr]
 +
 \bigl[\bb_{i,r-1},[\bb_{\tau i,1},\bb_{j,s}]\bigr]
 \\
 &=
 \bigl[[\bb_{i,0},\bb_{j,r+s-1}],\bb_{\tau i,1}\bigr]
 +
 \bigl[\bb_{i,r-1},[\bb_{\tau i,0},\bb_{j,s+1}]\bigr]
 \\
 &=
 [\bh_{\tau i,1},\bb_{j,r+s-1}]
 +
 \bigl[\bb_{i,0},[\bb_{j,r+s-1},\bb_{\tau i,1}]\bigr]\\
 &\qquad\qquad-
 [\bh_{i,r-1},\bb_{j,s+1}]
 +
 \bigl[\bb_{\tau i,0},[\bb_{i,r-1},\bb_{j,s+1}]\bigr]
 \\
 &=(c_{ij}+c_{\tau i,j})\bb_{j,r+s}-[\bh_{i,r-1},\bb_{j,s+1}]\\
 &\qquad\qquad
 +
 \bigl[\bb_{i,0},[\bb_{j,r+s},\bb_{\tau i,0}]\bigr]
 +
 \bigl[\bb_{\tau i,0},[\bb_{i,0},\bb_{j,r+s}]\bigr]
 \\
 &=
 (c_{ij}+c_{\tau i,j})\bb_{j,r+s}-[\bh_{i,r-1},\bb_{j,s+1}]-\bigl[[\bb_{i,0},\bb_{\tau i,0}],\bb_{j,r+s}\bigr]
 \\
  &=
 (c_{ij}+c_{\tau i,j})\bb_{j,r+s}-[\bh_{i,r-1},\bb_{j,s+1}]-[\bh_{\tau i,0},\bb_{j,r+s}]
 \\
 &=2c_{ij}\bb_{j,r+s}-[\bh_{i,r-1},\bb_{j,s+1}].
\end{align*}
Thus, we easily obtain by induction on $r$ that $[\bh_{i,r},\bb_{j,s}]=(c_{ij}-(-1)^rc_{\tau i,j})\bb_{j,r+s}$. If $i\in \I_0$, then one proves similarly that $[\bh_{i,2r+1},\bb_{j,s}]=2c_{ij}\bb_{j,2r+s+1}$, see also \cite[Step 3 of \S3.3]{Lu26min}.

Then for any $i\in \I$ and $j\in \I$, we have $[\bh_{i,r},\bb_{j,s}]=(c_{ij}-(-1)^rc_{\tau i,j})\bb_{j,r+s}$. Here $\tau i=i$ is allowed. Applying $\mathrm{ad}(\bb_{\tau j,0})$ to it, we obtain that
\begin{align*}
(-1)^r(c_{ij}-(-1)^rc_{\tau i,j})[\bb_{\tau j,r},\bb_{j,s}]+[\bh_{i,r},\bh_{j,s}]=(c_{ij}-(-1)^rc_{\tau i,j})\bh_{j,r+s}.
\end{align*}
Using \eqref{eq:qsclassical0} for the case $j\in \I_1$ and \cite[(3.42)]{Lu26min} for the case $j\in \I_0$, we find that $[\bh_{i,r},\bh_{j,s}]=0$, proving the first equality in \eqref{eq:qsclassical1}. 

Combining with the treatment for type $\mathsf A_{2n}$ in Appendix \ref{sec:pf-prop2}, the proof of Proposition \ref{prop:red} is complete.

\subsection{Proof of Theorem \ref{thm:min}}\label{sec:pfthmA}
Now we prove Theorem \ref{thm:min} with the help of Proposition \ref{prop:red}.

Let $\Yi^{\min}$ be the algebra generated by $h_{i,r},b_{i,r}$ for $i\in \I$ and $r=0,1$ subject to the relations \eqref{redqs1}--\eqref{redqs5}, the finite Serre type relations \eqref{finserre1-ty}--\eqref{finserre4-ty}, and the extra relations \eqref{eq:extraA1}--\eqref{eq:extraA2-2} depending on the underlying type. Then we have an algebra homomorphism
\begin{align*}
\pi:&~ \Yi^{\min}\to \Yi\\
&~h_{i,r}\mapsto h_{i,r},\quad b_{i,r}\mapsto b_{i,r} \qquad (i\in \I,r=0,1).
\end{align*}
Note that when the extra relations are included, then they are also satisfied in $\Yi$ by Lemma \ref{lem:hbhb}.
Moreover, $\pi$ is surjective as $\Yi$ is generated by $h_{i,0},h_{i,1},b_{i,0}$ for $i\in \I$ by Lemma \ref{lem:generate}. 

Introduce a filtration on $\Yi^{\min}$  by setting $\deg h_{i,r}=\deg b_{i,r}=r$ for $i\in \I$ and $r=0,1$. Recall the filtration on $\Yi$ from Section \ref{sec:iy}. Then $\pi$ is clearly a filtered homomorphism. Taking the associated graded, we obtain a surjective homomorphism $\gr\,\pi:\gr\,\Yi^{\min}\to\gr\,\Yi$.

On the other hand, recall the algebra isomorphism $\varrho: \mathrm{U}(\g[z]^{\check\omega})\stackrel{\cong\,}{\longrightarrow} \mathrm{gr}\,\Yi$ from \eqref{assi}. By the presentation of $\mathrm{U}(\g[z]^{\check\omega})$ from Proposition \ref{prop:red}, we also have a natural homomorphism $\mathrm{U}(\g[z]^{\check\omega})\to \gr\,\Yi^{\min}$, which matches on generators (in different math fonts). By composing  $\varrho^{-1}$ with this map, we obtain a homomorphism $\gr \,\Yi\to \gr\,\Yi^{\min}$; by \eqref{assi}, \eqref{eq:hb-explicit} and \eqref{eq:psi}, this homomorphism is clearly the inverse to $\gr\,\pi$ as seen on generators. Hence, we have proved that $\gr\,\pi:\gr\,\Yi^{\min}\to\gr\,\Yi$ is an isomorphism, and so is $\pi:\Yi^{\min}\to \Yi$.

\subsection{A strengthening for split iYangians}\label{ssec:split-strengthening}
Recall the associated-graded argument for
the minimalistic presentation of the split iYangians in  \cite[Theorem 3.1]{Lu26min}. In that
presentation an extra relation as \eqref{eq:extraA1} was imposed in types
$\mathsf A_1$, $\mathsf B_2\cong\mathsf C_2$, and $\mathsf G_2$.
The following argument shows that the last two cases are not exceptional. By an abuse of notation, we denote by 
$\Yi$ the split iYangian and in this case $\g$ can be any finite-dimensional simple Lie algebra. Note that in this case $\I=\I_0$.

Let $\g$ be simple of rank at least two and let $i,j\in\I$ be adjacent
vertices. Set
\[
 a_{ii}=(\alpha_i,\alpha_i),\qquad
 a_{jj}=(\alpha_j,\alpha_j),\qquad
 a_{ij}=(\alpha_i,\alpha_j)\ne0.
\]
We use the recursively defined elements $\bb_{k,r}$, $\bh_{k,2r+1}$
 and the identities
\begin{align}
 [\bh_{k,1},\bb_{l,r}]&=2(\alpha_k,\alpha_l)\bb_{l,r+1},
 \label{eq:split-mode-action}\\
 \bh_{k,1}&=[\bb_{k,0},\bb_{k,1}],
 &
 \bh_{k,3}&=[\bb_{k,2},\bb_{k,1}]
             =[\bb_{k,0},\bb_{k,3}],
 \label{eq:split-low-modes}\\
 [\bh_{k,3},\bb_{l,r}]&=2(\alpha_k,\alpha_l)\bb_{l,r+3}
 \qquad(k\ne l)
 \label{eq:split-h3-action}
\end{align}
that are established before the extra relation is used; see
\cite[(3.24)--(3.32)]{Lu26min}.

\begin{lem}\label{lem:split-extra-redundant-gr}
In the split classical minimalistic presentation, the relations other
than the extra relation imply
\begin{equation}
 [\bh_{k,1},\bh_{k,3}]=0
 \qquad(k\in\I)
 \label{eq:split-h1h3-derived}
\end{equation}
whenever $\g$ is simple of rank at least two. Consequently,
\begin{equation}
 \Big[\bh_{k,1},
   \big[\bb_{k,1},[\bh_{k,1},\bb_{k,1}]\big]
 \Big]=0
 \label{eq:split-extra-graded-derived}
\end{equation}
is redundant. In particular, this applies to
$\mathsf B_2\cong\mathsf C_2$ and $\mathsf G_2$.
\end{lem}

\begin{proof}
Put $\mathscr C_i=[\bh_{i,1},\bh_{i,3}]$ and $\mathscr C_j=[\bh_{j,1},\bh_{j,3}]$. Using \eqref{eq:split-mode-action} and
\eqref{eq:split-low-modes}, we first obtain
\begin{align}
 \mathscr C_i
 &=[\bh_{i,1},[\bb_{i,2},\bb_{i,1}]]
   =2a_{ii}[\bb_{i,3},\bb_{i,1}],
 \label{eq:split-C-first}\\
 \mathscr C_i
 &=[\bh_{i,1},[\bb_{i,0},\bb_{i,3}]]
   =2a_{ii}\bigl(
      [\bb_{i,1},\bb_{i,3}]
      +[\bb_{i,0},\bb_{i,4}]
    \bigr),
 \label{eq:split-C-second}
\end{align}
see also
\cite[(3.35)--(3.36)]{Lu26min}. Comparing
\eqref{eq:split-C-first} and \eqref{eq:split-C-second} gives
\begin{equation}
 [\bb_{i,3},\bb_{i,1}]=\frac{\mathscr C_i}{2a_{ii}} ,
 \qquad
 [\bb_{i,0},\bb_{i,4}]=\frac{\mathscr C_i}{a_{ii}}.
 \label{eq:split-C-consequences}
\end{equation}

We now use the adjacent vertex $j$. By
\eqref{eq:split-low-modes}, \eqref{eq:split-h3-action}, and
\eqref{eq:split-C-consequences},
\begin{align}
 [\bh_{j,3},\bh_{i,1}]
 =[\bh_{j,3},[\bb_{i,0},\bb_{i,1}]]
=2a_{ij}\bigl(
      [\bb_{i,3},\bb_{i,1}]
      +[\bb_{i,0},\bb_{i,4}]
    \bigr)
 =\frac{3a_{ij}}{a_{ii}}\mathscr C_i.
 \label{eq:split-cross-first}
\end{align}
On the other hand, applying \eqref{eq:split-mode-action} directly to
$\bh_{j,3}=[\bb_{j,2},\bb_{j,1}]$ yields
\begin{align}
 [\bh_{i,1},\bh_{j,3}]
 &=[\bh_{i,1},[\bb_{j,2},\bb_{j,1}]]
 =2a_{ij}[\bb_{j,3},\bb_{j,1}]
 =\frac{a_{ij}}{a_{jj}}\mathscr C_j.
 \label{eq:split-cross-second}
\end{align}
Comparing \eqref{eq:split-cross-first}  with
\eqref{eq:split-cross-second} gives
\begin{equation*}
\mathscr C_j=-3\frac{a_{jj}}{a_{ii}}\mathscr C_i.
\end{equation*}
Interchanging $i$ and $j$ in the same argument gives
\begin{equation*}
\mathscr C_i=-3\frac{a_{ii}}{a_{jj}}\mathscr C_j.
\end{equation*}
Hence $\mathscr C_i= \mathscr C_j=0$. Since the Dynkin diagram of $\g$ is connected, the same
argument proves \eqref{eq:split-h1h3-derived} for every vertex.

Finally, by \eqref{eq:split-mode-action} and
$[\bb_{k,1},\bb_{k,2}]=-\bh_{k,3}$, we have $\big[\bh_{k,1},
   \big[\bb_{k,1},[\bh_{k,1},\bb_{k,1}]\big]
 \big]
 =-2a_{kk}[\bh_{k,1},\bh_{k,3}]$, which proves \eqref{eq:split-extra-graded-derived}.
\end{proof}

\begin{cor}\label{cor:split-extra-redundant-quantum}
Let $\g$ be simple of rank at least two. The split iYangian is
presented by the generators $h_{i,1},b_{i,0},b_{i,1}$ and the relations
\cite[(3.1)--(3.3)]{Lu26min}, together with the finite Serre relations \cite[(2.30)--(2.33)]{Lu26min},
without imposing the extra relation
\begin{equation*}
 \Big[h_{i,1},
   \big[b_{i,1},[h_{i,1},b_{i,1}]\big]
 \Big]
 =a_{ii}^2[b_{i,1}^2,h_{i,1}].
\end{equation*}
Thus the extra relation in the minimalistic presentation of split iYangians is
needed only in type $\mathsf A_1$; in particular, it is redundant in
types $\mathsf B_2\cong\mathsf C_2$ and $\mathsf G_2$. Moreover,  \cite[Theorems~4.1 and~5.1]{Lu26min} also hold for type $\mathsf G_2$.
\end{cor}
\begin{proof}
The proof is similar to Theorem \ref{thm:min} and \cite[Theorem 3.1]{Lu26min} with the help of \cite[Proposition 3.2]{Lu26min} and Lemma \ref{lem:split-extra-redundant-gr}.
\end{proof}

\section{Proof of Theorem \ref{thm:embed}}\label{sec:pfthmB0}
In this section, we prove Theorem \ref{thm:embed}. 
\subsection{$\varphi$ induces an algebra homomorphism}\label{ssec:homo}
We first prove that $\varphi$ induces an algebra homomorphism with the help of the minimalistic presentation from Theorem \ref{thm:min}, assuming $\g$ is not of type $\mathsf A_{2n}$. The case of type $\mathsf A_{2n}$ will be treated in Appendix \ref{sec:app+} with the help of R-matrix presentation. Then we need to verify that the elements $\varphi(h_{i,r}),\varphi(b_{i,r})$ for $i\in\I$ and $r=0,1$ satisfy the relations \eqref{redqs1}--\eqref{redqs5} as the finite type Serre relations are straightforward.

\subsubsection{Some simple relations}\label{ssec:simple-relations}
Clearly, we have $\varphi(h_{\tau i,0})=-\varphi(h_{i,0})$ and $\varphi(h_{\tau i,1})=\varphi(h_{i,1})$. This verifies the second relation in \eqref{redqs1}. The first relation in \eqref{redqs1} for $r=s=0$ is obvious. Now we consider $r=0$ and $s=1$, then we have
\begin{align*}
\big[\varphi(h_{i,0}),\varphi(h_{j,1})\big]&=\Big[\xi_i-\xi_{\tau i},\xi_{j,1}+\xi_{\tau j,1}-\xi_{j}\xi_{\tau j}+\hf\sum_{\alpha\in \Phi^+}(\alpha,\alpha_j)\big\{x_\alpha^+,x_{\tau \alpha}^+\big\}\Big]\\
&=\hf\Big[\xi_i-\xi_{\tau i},\sum_{\alpha\in\Phi^+}(\alpha,\alpha_j)\big\{x_\alpha^+,x_{\tau \alpha}^+\big\}\Big]\\
&=\hf\sum_{\alpha\in\Phi^+}(\alpha,\alpha_j)\big((\alpha_i,\alpha+\tau\alpha)-(\alpha_{\tau i},\alpha+\tau \alpha)\big)\big\{x_\alpha^+,x_{\tau \alpha}^+\big\}=0.
\end{align*}
The relation \eqref{redqs2} for $r=0$ is clear. Now we prove it for the case $r=1$ using the relation \eqref{redqs3} below. Set $\Xi_i:=\varphi(h_{i,1})-\hf\varphi(h_{i,0})^2$ for $i\in \I$. Then it follows from \eqref{redqs2} for $r=0$ and \eqref{redqs3} that
\begin{align*}
\big[\Xi_i,\varphi(b_{j,0})\big]&=[\varphi(h_{i,1})-\hf\varphi(h_{i,0})^2,\varphi(b_{j,0})]\\
&=(c_{ij}+c_{\tau i,j})\varphi(b_{j,1})+\hf(c_{ij}-c_{\tau i,j})\big\{\varphi(h_{i,0}),\varphi(b_{j,0})\big\}\\
&~~~~-\hf(c_{ij}-c_{\tau i,j})\varphi(b_{j,0})\varphi(h_{i,0})-\hf(c_{ij}-c_{\tau i,j})\varphi(h_{i,0})\varphi(b_{j,0})\\
&=(c_{ij}+c_{\tau i,j})\varphi(b_{j,1}).
\end{align*}
Since $\big[\varphi(h_{i,0}),\varphi(h_{j,0})\big]=\big[\varphi(h_{i,0}),\varphi(h_{j,1})\big]=0$, we have $\big[\varphi(h_{i,0}),\Xi_j\big]=0$ and hence
\begin{align*}
\big[\Xi_j,\big[\varphi(h_{i,0}),\varphi(b_{j,0})\big]\big]=\big[\varphi(h_{i,0}),\big[\Xi_j,\varphi(b_{j,0})\big]\big]=(c_{jj}+c_{\tau j,j})\big[\varphi(h_{i,0}),\varphi(b_{j,1})\big].
\end{align*}
On the other hand, we also have
\[
\big[\Xi_j,\big[\varphi(h_{i,0}),\varphi(b_{j,0})\big]\big]=\big[\Xi_j,(c_{ij}-c_{\tau i,j})\varphi(b_{j,0})\big]=(c_{ij}-c_{\tau i,j})(c_{jj}+c_{\tau j,j})\varphi(b_{j,1}).
\]
Combining the above two equations, we obtain the first relation in \eqref{redqs2} for $r=1$.
\subsubsection{The relation \eqref{redqs3}}
We start with the following lemma.
\begin{lem}\label{lem:helper1}
For any $i,j\in\I$, we have
\begin{align*}
\Big[\sum_{\alpha\in \Phi^+}(\alpha,\alpha_i)\big\{x_\alpha^+,x_{\tau \alpha}^+\big\},x_j^+-x_{\tau j}^-\Big]=(c_{ij}+c_{\tau i,j})\Big(\sum_{\alpha\in\Phi^+}\big\{[x_j^+,x_\alpha^+],x_{\tau \alpha}^+\big\}-\big\{x_j^+,\xi_{\tau j}\big\}\Big).
\end{align*}
\end{lem}
\begin{proof}
The LHS is equal to
\begin{align*}
&\sum_{\alpha\in \Phi^+}(\alpha,\alpha_i)\big\{[x_\alpha^+,x_j^+],x_{\tau \alpha}^+\big\}+\sum_{\alpha\in \Phi^+}(\alpha,\alpha_i)\big\{x_\alpha^+,[x_{\tau \alpha}^+,x_j^+]\big\}\\
&-\sum_{\alpha\in \Phi^+}(\alpha,\alpha_i)\big\{[x_\alpha^+,x_{\tau j}^-],x_{\tau \alpha}^+\big\}-\sum_{\alpha\in \Phi^+}(\alpha,\alpha_i)\big\{x_\alpha^+,[x_{\tau \alpha}^+,x_{\tau j}^-]\big\}\\
=&\sum_{\alpha\in \Phi^+}(\alpha,\alpha_i)\big\{[x_\alpha^+,x_j^+],x_{\tau \alpha}^+\big\}-\sum_{\alpha\in \Phi^+}(\alpha+\alpha_j,\alpha_i)\big\{x_{\alpha+\alpha_j}^+,[x_{\tau \alpha+\alpha_{\tau j}}^+,x_{\tau j}^-]\big\}\\
&+\sum_{\alpha\in \Phi^+}(\alpha,\alpha_i)\big\{x_\alpha^+,[x_{\tau \alpha}^+,x_j^+]\big\}-\sum_{\alpha\in \Phi^+}(\alpha+\alpha_{\tau j},\alpha_i)\big\{[x_{\alpha+\alpha_{\tau j}}^+,x_{\tau j}^-],x_{\tau \alpha+\alpha_j}^+\big\}\\
&-(\alpha_{\tau j},\alpha_i)\big\{\xi_{\tau j},x_j^+\big\}-(\alpha_j,\alpha_i)\big\{x_j^+,\xi_{\tau j}\big\}\\
=&(c_{ij}+c_{\tau i,j})\sum_{\alpha\in \Phi^+}\big\{[x_j^+,x_\alpha^+],x_{\tau \alpha}^+\big\}-(c_{ij}+c_{\tau i,j})\big\{x_j^+,\xi_{\tau j}\big\},
\end{align*}
completing the proof of the lemma. Here we have used the substitution $\alpha\mapsto \alpha+\alpha_{j}$ and $\alpha\mapsto \alpha+\alpha_{\tau j}$. We also have applied \eqref{eta} and \eqref{eta-tau} to obtain 
\[
\big\{[x_\alpha^+,x_j^+],x_{\tau \alpha}^+\big\}=\big\{x_{\alpha+\alpha_j}^+,[x_{\tau \alpha+\alpha_{\tau j}}^+,x_{\tau j}^-]\big\},\quad \big\{x_\alpha^+,[x_{\tau \alpha}^+,x_j^+]\big\}=\big\{[x_{\alpha+\alpha_{\tau j}}^+,x_{\tau j}^-],x_{\tau \alpha+\alpha_j}^+\big\},\]
so that many terms are cancelled.
\end{proof}

Now we verify the relation \eqref{redqs3}. By Lemma \ref{lem:helper1}, the LHS of \eqref{redqs3} is equal to
\begin{align*}
&\Big[\xi_{i,1}+\xi_{\tau i,1}-\xi_i\xi_{\tau i}+\hf\sum_{\alpha\in\Phi^+}(\alpha,\alpha_i)\big\{x_\alpha^+,x_{\tau\alpha}^+\big\},x_j^+-x_{\tau j}^-\Big]\\
=&\,c_{ij}x_{j,1}^++\hf c_{ij}\{\xi_i,x_j^+\}+c_{\tau i,j}x_{\tau j,1}^-+\hf c_{\tau i,j}\{\xi_i,x_{\tau j}^-\}+c_{\tau i,j}x_{j,1}^++\hf c_{\tau i,j}\{\xi_{\tau i},x_j^+\}\\
&+c_{i,j}x_{\tau j,1}^-+\hf c_{ij}\{\xi_{\tau i},x_{\tau j}^-\}-c_{ij}x_j^+\xi_{\tau i}-c_{\tau i,j}\xi_ix_j^+-c_{i,\tau j}x_{\tau j}^-\xi_{\tau i}-c_{ij}\xi_ix_{\tau j}^-\\
&+\hf(c_{ij}+c_{\tau i,j})\Big(\sum_{\alpha\in\Phi^+}\big\{[x_j^+,x_\alpha^+],x_{\tau \alpha}^+\big\}-\big\{x_j^+,\xi_{\tau j}\big\}\Big)\\
=&\, \big(c_{ij}+c_{\tau i,j}\big)\big(x_{j,1}^++x_{\tau j,1}^-\big)+\hf\big(c_{ij}-c_{\tau i,j}\big)\big\{\xi_i-\xi_{\tau i},x_j^+-x_{\tau j}^-\big\}\\
&+\hf\big(c_{ij}+c_{\tau i,j}\big)\Big(\sum_{\alpha\in\Phi^+}\big\{[x_j^+,x_\alpha^+],x_{\tau \alpha}^+\big\}-\big\{x_j^+,\xi_{\tau j}\big\}\Big)\\
=&\,\big(c_{ij}+c_{\tau i,j}\big)\varphi(b_{j,1})+\hf\big(c_{ij}-c_{\tau i,j}\big)\big\{\varphi(h_{i,0}),\varphi(b_{j,0})\big\}.
\end{align*}

\subsubsection{The relation \eqref{redqs5} for $j\ne\tau i$}
We need to verify
\begin{align}
\big[\varphi(b_{i,1}),\varphi(b_{j,0})\big]-\big[\varphi(b_{i,0}),\varphi(b_{j,1})\big]= \tfrac12(\alpha_i,\alpha_j)\big\{\varphi(b_{i,0}),\varphi(b_{j,0})\big\},	\quad (j\ne \tau i).\label{bbinej}
\end{align}
\begin{lem}\label{lem:helper2}
We have
\begin{align}
&\Big[x_i^+-x_{\tau i}^-,\sum_{\alpha\in\Phi^+}\big\{[x_j^+,x_\alpha^+],x_{\tau\alpha}^+\big\}\Big]+\Big[x_j^+-x_{\tau j}^-,\sum_{\alpha\in\Phi^+}\big\{[x_i^+,x_\alpha^+],x_{\tau\alpha}^+\big\}\Big]\label{helper2}\\
=&\big\{[x_i^+,\xi_{\tau j}],x_j^+\big\}+\big\{[x_j^+,\xi_{\tau i}],x_i^+\big\}+\big\{[x_i^+,x_j^+],\xi_{\tau j}\big\}+\big\{[x_j^+,x_i^+],\xi_{\tau i}\big\}.\notag
\end{align}
\end{lem}
\begin{proof}
The LHS of \eqref{helper2} is equal to
\begin{align*}
&\sum_{\alpha\in\Phi^+\setminus\{\alpha_j\}}\bigg(\Big(\big\{\big[[x_i^+,x_j^+],x_\alpha^+\big],x_{\tau\alpha}^+\big\}+\big\{\big[x_j^+,[x_i^+,x_\alpha^+]\big],x_{\tau\alpha}^+\big\}\\
&+\big\{[x_j^+,x_\alpha^+],[x_i^+,x_{\tau\alpha}^+]\big\}\Big)-\Big(\big\{\big[x_j^+,[x_{\tau i}^-,x_\alpha^+]\big],x_{\tau\alpha}^+\big\}+\big\{[x_j^+,x_\alpha^+],[x_{\tau i}^-,x_{\tau\alpha}^+]\big\}\Big)\bigg)\\
+&\sum_{\alpha\in\Phi^+\setminus\{\alpha_i\}}\bigg(\big\{\big[[x_j^+,x_i^+],x_\alpha^+\big],x_{\tau\alpha}^+\big\}+\Big(\big\{\big[x_i^+,[x_j^+,x_\alpha^+]\big],x_{\tau\alpha}^+\big\}\\
&+\big\{[x_i^+,x_\alpha^+],[x_j^+,x_{\tau\alpha}^+]\big\}\Big)-\Big(\big\{\big[x_i^+,[x_{\tau j}^-,x_\alpha^+]\big],x_{\tau\alpha}^+\big\}+\big\{[x_i^+,x_\alpha^+],[x_{\tau j}^-,x_{\tau\alpha}^+]\big\}\Big)\bigg)\\
\stackrel{(*)}{=}&\sum_{\alpha\in\Phi^+\setminus\{\alpha_i,\alpha_j\}}\bigg(\Big(\big\{\big[x_i^+,[x_j^+,x_\alpha^+]\big],x_{\tau\alpha}^+\big\}-\big\{[x_i^+,x_{\alpha+\alpha_j}^+],[x_{\tau j}^-,x_{\alpha+\alpha_{\tau j}}^+]\big\}\Big)\\
&\qquad\qquad\qquad-\Big(\big\{\big[x_j^+,[x_{\tau i}^-,x_{\alpha+\alpha_{\tau i}}^+]\big],x_{\tau\alpha+\alpha_i}^+\big\}-\big\{[x_i^+,x_{\tau\alpha}^+],[x_j^+,x_\alpha^+]\big\}\Big)\bigg)\\
+&\sum_{\alpha\in\Phi^+\setminus\{\alpha_i,\alpha_j\}}\bigg(\Big(\big\{\big[x_j^+,[x_i^+,x_\alpha^+]\big],x_{\tau\alpha}^+\big\}-\big\{[x_j^+,x_{\alpha+\alpha_i}^+],[x_{\tau i}^-,x_{\tau\alpha+\alpha_{\tau i}}^+]\big\}\Big)\\
&\qquad\qquad\qquad-\Big(\big\{\big[x_i^+,[x_{\tau j}^-,x_{\alpha+\alpha_{\tau j}}^+]\big],x_{\tau\alpha+\alpha_j}^+\big\}-\big\{[x_j^+,x_{\tau\alpha}^+],[x_i^+,x_\alpha^+]\big\}\Big)\bigg)\\
+&\,\big\{[x_i^+,\xi_{\tau j}],x_j^+\big\}+\big\{[x_j^+,\xi_{\tau i}],x_i^+\big\}+\big\{[x_i^+,x_j^+],\xi_{\tau j}\big\}+\big\{[x_j^+,x_i^+],\xi_{\tau i}\big\}\\
\stackrel{\eqref{eta}}{=}&\,\big\{[x_i^+,\xi_{\tau j}],x_j^+\big\}+\big\{[x_j^+,\xi_{\tau i}],x_i^+\big\}+\big\{[x_i^+,x_j^+],\xi_{\tau j}\big\}+\big\{[x_j^+,x_i^+],\xi_{\tau i}\big\}.
\end{align*}
In $(*)$ we applied \eqref{eq:relXX} and suitable substitutions for $\alpha$ in certain terms, and then arranged the terms properly so that the terms in parenthesis eventually cancel.
\end{proof}
Applying Lemma \ref{lem:helper2}, we obtain that the LHS of \eqref{bbinej} is
\begin{align*}
&\Big[x_{i,1}^++x_{\tau i,1}^-+\hf\sum_{\alpha\in\Phi^+}\big\{[x_i^+,x_\alpha^+],x_{\tau \alpha}^+\big\}-\hf\big\{x_i^+,\xi_{\tau i}\big\},x_j^+-x_{\tau j}^-\Big]\\
&-\Big[x_i^+-x_{\tau i}^-,x_{j,1}^++x_{\tau j,1}^-+\hf\sum_{\alpha\in\Phi^+}\big\{[x_j^+,x_\alpha^+],x_{\tau \alpha}^+\big\}-\hf\big\{x_j^+,\xi_{\tau j}\big\}\Big]
\\
=\,&[x_{i,1}^+,x_j^+]-[x_i^+,x_{j,1}^+]-[x_{\tau i,1}^-,x_{\tau j}^-]+[x_{\tau i}^-,x_{\tau j,1}^-]
\\
&-\hf\big\{[x_i^+,x_j^+],\xi_{\tau i}\big\}-\hf\big\{x_i^+,[\xi_{\tau i},x_j^+]\big\}-\hf c_{ij}\big\{x_i^+,x_{\tau j}^-\big\}
\\
&+\hf\big\{[x_i^+,x_j^+],\xi_{\tau j}\big\}+\hf\big\{x_j^+,[x_i^+,\xi_{\tau j}]\big\}-\hf c_{ij}\big\{x_j^+,x_{\tau i}^-\big\}
\\
&-\hf\big\{[x_i^+,\xi_{\tau j}],x_j^+\big\}-\hf\big\{[x_j^+,\xi_{\tau i}],x_i^+\big\}-\hf\big\{[x_i^+,x_j^+],\xi_{\tau j}\big\}-\hf\big\{[x_j^+,x_i^+],\xi_{\tau i}\big\}\\
=\,&\hf c_{ij}\big(\{x_i^+,x_j^+\}+\{x_{\tau i}^-,x_{\tau j}^-\}-\{x_i^+,x_{\tau j}^-\}-\{x_j^+,x_{\tau i}^-\}\big)\\
=\,&\hf c_{ij}\{x_i^+-x_{\tau i}^-,x_j^+-x_{\tau j}^-\}=\hf c_{ij}\{\varphi(b_{i,0}),\varphi(b_{j,0})\}.
\end{align*}

\subsubsection{The relation \eqref{redqs5} for $j=i=\tau i$} 
We need to verify
\begin{align*}
&\big[\varphi(b_{i,0}),\varphi(b_{i,1})\big]= \varphi(h_{i,1})-\tfrac12(\alpha_i,\alpha_i)\big(\varphi(b_{i,0})\big)^2, & (\tau i =i),
\end{align*}
using the following lemma.
\begin{lem}\label{lem:helper3}
We have
\begin{align}
&\Big[x_i^+ -x_i^-,\sum_{\alpha\in\Phi^+}\big\{[x_i^+,x_{\alpha}^+],x_{\tau \alpha}^+\big\}\Big]=\sum_{\alpha\in\Phi^+\setminus\{\alpha_i\}}(\alpha,\alpha_i)\big\{x_{\alpha}^+,x_{\tau\alpha}^+\big\},\quad (\tau i=i).\label{helper3}\end{align}
\end{lem}
\begin{proof}
The LHS of \eqref{helper3} is equal to
\begin{align*}
& \sum_{\alpha\in\Phi^+\setminus\{\alpha_i\}}\Big(\big\{\big[x_i^+,[x_i^+,x_{\alpha}^+]\big],x_{\tau\alpha}^+\big\}+\big\{[x_i^+,x_{\alpha}^+],\big[x_i^+,x_{\tau\alpha}^+\big]\big\}
\\&\hskip2cm+\big\{\big[\xi_i,x_{\alpha}^+\big],x_{\tau\alpha}^+\big\}-\big\{\big[x_i^+,[x_i^-,x_{\alpha}^+]\big],x_{\tau\alpha}^+\big\}-\big\{[x_i^+,x_{\alpha}^+],\big[x_i^-,x_{\tau\alpha}^+\big]\big\}\Big)\\
=&\sum_{\alpha\in\Phi^+\setminus\{\alpha_i\}}\big\{\big[\xi_i,x_{\alpha}^+\big],x_{\tau\alpha}^+\big\}=\sum_{\alpha\in\Phi^+\setminus\{\alpha_i\}}(\alpha,\alpha_i)\big\{x_{\alpha}^+,x_{\tau\alpha}^+\big\}.
\end{align*}
Here we used \eqref{eta} and \eqref{eta-tau} to obtain the cancellations 
\begin{align*}
&\big\{\big[x_i^+,[x_i^+,x_{\alpha}^+]\big],x_{\tau\alpha}^+\big\}=\big\{[x_i^+,x_{\alpha+\alpha_i}^+],\big[x_i^-,x_{\tau\alpha+\alpha_i}^+\big]\big\},\\
&\big\{[x_i^+,x_{\alpha}^+],\big[x_i^+,x_{\tau\alpha}^+\big]\big\}=	\big\{\big[x_i^+,[x_i^-,x_{\alpha+\alpha_i}^+]\big],x_{\tau\alpha+\alpha_i}^+\big\},
\end{align*}
as in Lemma \ref{lem:helper1}. 
\end{proof}

Hence, if $\tau i=i$, then by Lemma \ref{lem:helper3} we have
\begin{align*}
\big[\varphi(b_{i,0}),\varphi(b_{i,1})\big]&=\Big[x_i^+-x_i^-,x_{i,1}^++x_{i,1}^-+\hf\sum_{\alpha\in \Phi^+}\big\{[x_i^+,x_\alpha^+],x_{\tau\alpha}^+\big\}-\hf\big\{x_i^+,\xi_i\big\}\Big]\\
&=[x_i^+,x_{i,1}^+]+\xi_{i,1}+c_{ii}(x_i^+)^2+\xi_{i,1}+[x_{i,1}^-,x_i^-]-\xi_i^2 \\
&\quad~ +\hf c_{ii}\{x_i^+,x_i^-\}+\hf\big[x_i^+ -x_i^-,\sum_{\alpha\in\Phi^+}\big\{[x_i^+,x_{\alpha}^+],x_{\tau \alpha}^+\big\}\big]\\
&=-\hf c_{ii}(x_i^+)^2+2\xi_{i,1}-\xi_i^2+c_{ii}(x_i^+)^2-\hf c_{ii}(x_i^-)^2+\hf c_{ii}\{x_i^+,x_i^-\}\\
&\quad~ +\hf \sum_{\alpha\in\Phi^+\setminus\{\alpha_i\}}(\alpha,\alpha_i)\big\{x_{\alpha}^+,x_{\tau\alpha}^+\big\}\\
&=2\xi_{i,1}-\xi_i^2-\hf c_{ii}(x_i^+-x_i^-)^2+\hf c_{ii}\{x_{i}^+,x_{i}^+\}\\
&\quad~ +\hf \sum_{\alpha\in\Phi^+\setminus\{\alpha_i\}}(\alpha,\alpha_i)\big\{x_{\alpha}^+,x_{\tau\alpha}^+\big\}\\
&= \varphi(h_{i,1})-\hf c_{ii}\varphi(b_{i,0})^2.
\end{align*}

\subsubsection{The relation \eqref{redqs5} for $j=\tau i\ne i$} 
We need to verify
\begin{align}
\big[\varphi(b_{i,1}),\varphi(b_{\tau i,0})\big]-\big[\varphi(b_{i,0}),\varphi(b_{\tau i,1})\big]= \hf c_{i,\tau i}\big\{\varphi(b_{i,0}),\varphi(b_{\tau i,0})\big\}-2 \varphi(h_{\tau i,1 }).\label{bbtaui=jnei}
\end{align}
\begin{lem}\label{lem:helper4}
We have
\begin{align*}
&\Big[
x_i^+-x_{\tau i}^{-},
\sum_{\alpha\in\Phi^+}
\left\{[x_{\tau i}^{+},x_\alpha^{+}],
x_{\tau\alpha}^{+}\right\}
\Big]
+
\Big[
x_{\tau i}^{+}-x_i^{-},
\sum_{\alpha\in\Phi^+}
\left\{[x_i^{+},x_\alpha^{+}],
x_{\tau\alpha}^{+}\right\}
\Big] \\
=\,&
\left\{ [x_{\tau i}^+,x_i^+ ],\xi_{\tau i}\right\}
+
\left\{ [x_i^+,x_{\tau i}^+ ],\xi_i\right\}+
\sum_{\alpha\in\Phi^+\setminus\{\alpha_{ i},\alpha_{\tau i}\}}
(\alpha_i,\alpha+\tau\alpha )
\left\{x_\alpha^+,x_{\tau\alpha}^+\right\}  . 
\end{align*} 
\end{lem}
\begin{proof}
Expanding the LHS, we find that	
\begin{align*}
&\sum_{\alpha\in\Phi^+\setminus\{\alpha_{\tau i}\}}
\Big(
\left\{\left[[x_i^+,x_{\tau i}^+],x_\alpha^+\right],
x_{\tau\alpha}^+\right\}
+
\left\{\left[x_{\tau i}^+,[x_i^+,x_\alpha^+]\right],
x_{\tau\alpha}^+\right\}  +
\left\{[x_{\tau i}^+,x_\alpha^+],
[x_i^+,x_{\tau\alpha}^+]\right\}                     \\
&\qquad\qquad~~~~
+
\left\{ [\xi_{\tau i},x_\alpha^+ ],
x_{\tau\alpha}^+\right\}                       
-
\left\{ [x_{\tau i}^+,x_\alpha^+ ],
 [x_{\tau i}^{-},x_{\tau\alpha}^+ ]\right\}
-
\left\{\left[x_{\tau i}^+,
 [x_{\tau i}^{-},x_\alpha^+ ]\right],
x_{\tau\alpha}^+\right\}
\Big)                                          \\
+&
\sum_{\alpha\in\Phi^+\setminus\{\alpha_i\}}
\Big(
\left\{\left[[x_{\tau i}^+,x_i^+],x_\alpha^+\right],
x_{\tau\alpha}^+\right\}
+
\left\{\left[x_i^+,[x_{\tau i}^+,x_\alpha^+]\right],
x_{\tau\alpha}^+\right\}  +
\left\{ [\xi_i,x_\alpha^+ ],
x_{\tau\alpha}^+\right\}                     \\
&\qquad\qquad
+
\left\{ [x_i^+,x_\alpha^+ ],
 [x_{\tau i}^+,x_{\tau\alpha}^+ ]\right\}
-
\left\{\left[x_i^+, [x_i^-,x_\alpha^+ ]\right],
x_{\tau\alpha}^+\right\}
-
\left\{ [x_i^+,x_\alpha^+ ],
 [x_i^-,x_{\tau\alpha}^+ ]\right\}
\Big).                                          
  \end{align*}
 The lemma follows by observing cancellations after suitable substitution as before.
\end{proof}

Now using Lemma \ref{lem:helper4}, the LHS of \eqref{bbtaui=jnei} is given by
\begin{align*}
&\Big[
x_{i,1}^{+}+x_{\tau i,1}^{-}
+\hf\sum_{\alpha\in\Phi^+}
\left\{ [x_i^+,x_\alpha^+ ],x_{\tau\alpha}^+\right\}
-\hf\left\{x_i^+,\xi_{\tau i}\right\},
x_{\tau i}^+-x_i^-
\Big]                                      \\
&-
\Big[
x_i^+-x_{\tau i}^{-},
x_{\tau i,1}^{+}+x_{i,1}^{-}
+\hf \sum_{\alpha\in\Phi^+}
\left\{ [x_{\tau i}^+,x_\alpha^+ ],x_{\tau\alpha}^+\right\}
-\hf\left\{x_{\tau i}^+,\xi_i\right\}
\Big]                                      \\
=\,&
\left[x_{i,1}^+,x_{\tau i}^+\right]
-\xi_{i,1}-\xi_{\tau i,1}
-\left[x_{\tau i,1}^{-},x_i^{-}\right]
-\hf\left\{ [x_i^+,x_{\tau i}^+ ],\xi_{\tau i}\right\}-\hf c_{ii}
\left\{x_i^+,x_{\tau i}^+\right\} \\
& 
-\hf c_{\tau i,i}
\left\{x_i^+,x_i^-\right\}
+\hf\left\{\xi_i,\xi_{\tau i}\right\} -\left[x_i^+,x_{\tau i,1}^{+}\right] -\xi_{i,1}-\xi_{\tau i,1}+\left[x_{\tau i}^{-},x_{i,1}^{-}\right]   \\
&
+\hf\left\{ [x_i^+,x_{\tau i}^+ ],\xi_i\right\}-\hf c_{ii}
\left\{x_{\tau i}^+,x_i^+\right\}
+\hf\left\{\xi_{\tau i},\xi_i\right\}
-\hf c_{\tau i,i}
\left\{x_{\tau i}^+,x_{\tau i}^{-}\right\}
\\
&
-\hf\left\{ [x_{\tau i}^+,x_i^+ ],\xi_{\tau i}\right\}
-\hf\left\{ [x_i^+,x_{\tau i}^+ ],\xi_i\right\}-\sum_{\alpha\in\Phi^+\setminus\{\alpha_i\}}
(\alpha,\alpha_i)
\left\{x_\alpha^+,x_{\tau\alpha}^+\right\}\\
=\,&\hf c_{\tau i,i}\left\{x_{\tau i}^+,x_i^+\right\}+\hf c_{\tau i,i}\left\{x_{\tau i}^-,x_i^-\right\}-2\xi_{i,1}-2\xi_{\tau i,1}+2\xi_i\xi_{\tau i}\\
&-\sum_{\alpha\in\Phi^+}
(\alpha,\alpha_i)
\left\{x_\alpha^+,x_{\tau\alpha}^+\right\}-\hf c_{\tau i,i}\left\{x_{i}^+,x_i^-\right\}-\hf c_{\tau i,i}\left\{x_{\tau i}^+,x_{\tau i}^-\right\}\\
=\,&\hf c_{\tau i,i}\left\{x_{\tau i}^+-x_i^-,x_i^+-x_{\tau i}^-\right\}-2\varphi(h_{i,1})=\hf c_{\tau i,i}\left\{\varphi(b_{i,0}),\varphi(b_{\tau i,0})\right\}-2\varphi(h_{i,1}).
\end{align*}

\subsubsection{The relation \eqref{redqs1} for $r=s=1$} Recall that the relation \eqref{redqs1} except for $r=s=1$ are already proved in \S\ref{ssec:simple-relations}. We need to verify 
\beq \label{hhi=0}
[\varphi(h_{i,1}), \varphi(h_{j,1})] = 0 .
\eeq
For that purpose, let $\tilde{v}_i := J(\xi_i)-\tilde{\xi}_{i,1}$, see \eqref{tlxidef} and \eqref{eq:J}. Then it follows from \eqref{hi1-embedding} that
\begin{gather*}
\varphi(h_{i,1}) = J(\xi_i) - \tilde{v}_i
+ J(\xi_{\tau i}) - \tilde{v}_{\tau i} +\tfrac{1}{2}(\xi_i-\xi_{\tau i})^2 +\tfrac{1}{2}\sum_{\alpha \in \Phi^{+}}
(\alpha,\alpha_i)\{x_{\alpha}^{+}, x_{\tau\alpha}^+\}.
\end{gather*}
It is known from \cite[Theorem 2.6]{GRW19} that $[J(\xi_i)-\tilde{v}_i, J(\xi_j)-\tilde{v}_j]=[J(\xi_i)-\tilde{v}_i, \xi_j] = 0$. Thus the LHS of \eqref{hhi=0} is reduced to
\beq
\begin{split}\label{expandhh1}
    \tfrac{1}{2}\sum_{\alpha\in\Phi^+} (\alpha, \alpha_j+\alpha_{\tau j})\{[J(\xi_i)+J(\xi_{\tau i}), x_\alpha^+], x_{\tau \alpha}^+\} -\tfrac{1}{2}\sum_{\alpha\in\Phi^+}(\alpha, \alpha_j+\alpha_{\tau j})\{[\tilde{v}_i+\tilde{v}_{\tau i}, x_\alpha^+], x_{\tau\alpha}^+\} \\
    -\tfrac{1}{2}\sum_{\beta\in\Phi^+} (\beta, \alpha_i+\alpha_{\tau i})\{[J(\xi_j)+J(\xi_{\tau j}), x_\beta^+], x_{\tau \beta}^+\} +\tfrac{1}{2}\sum_{\beta\in\Phi^+}(\beta,\alpha_i+\alpha_{\tau i})\{[\tilde{v}_j+\tilde{v}_{\tau j}, x_\beta^+], x_{\tau\beta}^+\}\\
    +\tfrac{1}{4}\sum_{\alpha, \beta\in \Phi^+}(\alpha, \alpha_i)(\beta, \alpha_j)[\{x_\alpha^+, x_{\tau\alpha}^+\}, \{x_\beta^+, x_{\tau \beta}^+\}].
\end{split}
\eeq
By \cite[Proposition 3.21]{GNW18}, we have
\[
[J(\xi_i) + J(\xi_{\tau i}), x_\alpha^+]= (\alpha,\alpha_i+\alpha_{\tau i})J(x_\alpha^+),\qquad [J(\xi_j)+J(\xi_{\tau j}), x_\beta^+]=(\beta,\alpha_j+\alpha_{\tau j})J(x_\beta^+).
\]
Thus the two summations in \eqref{expandhh1} involving $J(\xi_i)$ or $J(\xi_j)$ cancel. Further notice that 
\begin{align*}
    [\tilde{v}_i+\tilde{v}_{\tau i}, x_\alpha^+]&=\Big[ \tfrac{1}{8}c_\g(\xi_i+\xi_{\tau i})+\tfrac{1}{2}\sum_{\beta\in\Phi^+}(\beta, \alpha_i+\alpha_{\tau i})x_\beta^-x_\beta^+,x_\alpha^+\Big]\\
    & = \tfrac{1}{8}c_\g(\alpha,\alpha_i+\alpha_{\tau i})x_\alpha^++\tfrac{1}{2}\sum_{\beta\in\Phi^+}(\beta, \alpha_i+\alpha_{\tau i})\left([x_\beta^-, x_\alpha^+]x_\beta^+ + x_\beta^-[x_\beta^+, x_\alpha^+]\right)
\end{align*}
and a similar equality for $[\tilde{v}_j+\tilde{v}_{\tau j}, x_\beta^+]$ from \cite[$\S$3.2 and $\S$4.4]{GNW18}, where $c_\g$ is the eigenvalue of the Casimir element $C_\g \in \mathrm U(\g)$ acting on the adjoint representation. Therefore, the expression \eqref{expandhh1} can be rewritten as
\begin{align*}
-&\tfrac{1}{4}
\sum_{\alpha\in\Phi^{+}}
\sum_{\beta\in\Phi^{+}}
(\alpha,\alpha_j+\alpha_{\tau j})
(\beta,\alpha_i+\alpha_{\tau i})
\left\{
[x_{\beta}^{-},x_{\alpha}^{+}]x_{\beta}^{+}
+x_{\beta}^{-}[x_{\beta}^{+},x_{\alpha}^{+}],
x_{\tau\alpha}^{+}
\right\}
\\
+& \tfrac{1}{4}
\sum_{\alpha\in\Phi^{+}}
\sum_{\beta\in\Phi^{+}}
(\alpha,\alpha_j+\alpha_{\tau j})
(\beta,\alpha_i+\alpha_{\tau i})
\left\{
[x_{\alpha}^{-},x_{\beta}^{+}]x_{\alpha}^{+}
+x_{\alpha}^{-}[x_{\alpha}^{+},x_{\beta}^{+}],
x_{\tau\beta}^{+}
\right\}
\\
+&\tfrac{1}{4}
\sum_{\alpha\in\Phi^{+}}
\sum_{\beta\in\Phi^{+}}
(\alpha,\alpha_j+\alpha_{\tau j})
(\beta,\alpha_i+\alpha_{\tau i})
\left\{
[x_{\beta}^{+},x_{\alpha}^{+}]
x_{\tau\beta}^{+}
+x_{\tau\beta}^{+}
[x_{\beta}^{+},x_{\alpha}^{+}],
x_{\tau\alpha}^{+}
\right\}.
\end{align*}
We need to show that all terms cancel.

Clearly, the term for $\alpha = \beta$ cancels.
Next, we show that the terms involving exactly one negative root vector cancel. Here we treat the commutators as root vectors. There are two situations:
\[
x_{\alpha_0}^{-} x_{\alpha_0+\beta_0}^{+} x_{\tau\beta_0}^{+}
\quad \text{or} \quad
x_{\tau\beta_0}^{+} x_{\alpha_0}^{-} x_{\alpha_0+\beta_0}^{+}.
\]
for any given $\alpha_0$, $\beta_0 \in \Phi^+$ such that $\alpha_0+\beta_0\in\Phi^+$. We first consider the occurrence of $x_{\alpha_0}^{-} x_{\alpha_0+\beta_0}^{+} x_{\tau\beta_0}^{+}$. Then it appears in the following $4$ terms
\begin{align*}
-&\tfrac{1}{4}
(\beta_0,\alpha_j+\alpha_{\tau j})
(\alpha_0+\beta_0,\alpha_i+\alpha_{\tau i})
[x_{\alpha_0+\beta_0}^{-},x_{\beta_0}^{+}]
x_{\alpha_0+\beta_0}^{+}x_{\tau\beta_0}^{+}
\\
&\qquad = -\tfrac{1}{4}
(\beta_0,\alpha_j+\alpha_{\tau j})
(\alpha_0+\beta_0,\alpha_i+\alpha_{\tau i})
\eta_{-\alpha_0-\beta_0, \beta_0}x_{\alpha_0}^{-}
x_{\alpha_0+\beta_0}^{+}x_{\tau\beta_0}^{+},
\\
&\tfrac{1}{4}(\alpha_0+\beta_0,\alpha_j+\alpha_{\tau j})
(\beta_0,\alpha_i+\alpha_{\tau i})
[x_{\alpha_0+\beta_0}^{-},x_{\beta_0}^{+}]
x_{\alpha_0+\beta_0}^{+}x_{\tau\beta_0}^{+}
\\
&\qquad = \tfrac{1}{4}(\alpha_0+\beta_0,\alpha_j+\alpha_{\tau j})
(\beta_0,\alpha_i+\alpha_{\tau i})
\eta_{-\alpha_0-\beta_0, \beta_0}x_{\alpha_0}^{-}
x_{\alpha_0+\beta_0}^{+}x_{\tau\beta_0}^{+},
\\
&\tfrac{1}{4}(\alpha_0,\alpha_j+\alpha_{\tau j})
(\beta_0,\alpha_i+\alpha_{\tau i})
x_{\alpha_0}^{-}
[x_{\alpha_0}^{+},x_{\beta_0}^{+}]
x_{\tau\beta_0}^{+} 
\\ 
&\qquad = \tfrac{1}{4}(\alpha_0,\alpha_j+\alpha_{\tau j})
(\beta_0,\alpha_i+\alpha_{\tau i}) \eta_{\alpha_0, \beta_0}
x_{\alpha_0}^{-}
x_{\alpha_0+\beta_0}^+
x_{\tau\beta_0}^{+},
\\
-&\tfrac{1}{4}
(\beta_0,\alpha_j+\alpha_{\tau j})
(\alpha_0,\alpha_i+\alpha_{\tau i})
x_{\alpha_0}^{-}
[x_{\alpha_0}^{+},x_{\beta_0}^{+}]
x_{\tau\beta_0}^{+}
\\ 
&\qquad = -\tfrac{1}{4}
(\beta_0,\alpha_j+\alpha_{\tau j})
(\alpha_0,\alpha_i+\alpha_{\tau i})\eta_{\alpha_0, \beta_0}
x_{\alpha_0}^{-}
x_{\alpha_0+\beta_0}^+
x_{\tau\beta_0}^{+}.
\end{align*}
By \eqref{eta}, we have $\eta_{\alpha_0,\beta_0} + \eta_{-\alpha_0-\beta_0, \beta_0} = 0$ and hence these $4$ terms cancel. A similar calculation also works for the case $x_{\tau\beta_0}^{+} x_{\alpha_0}^{-} x_{\alpha_0+\beta_0}^{+}$.

Finally, we consider the remaining case when all monomials are products of positive root vectors. It suffices to consider the products of the form $x_{\alpha_0}^+x_{\beta_0}^+x_{\tau\alpha_0+\tau\beta_0}^+$ for fixed $\alpha_0,\beta_0\in \Phi^+$ such that $\alpha_0+\beta_0\in\Phi^+$. Then by a similar calculation as above using \eqref{eta} and \eqref{eta-tau}, these products simplify as 
\begin{align*}
&\tfrac{1}{4}(\alpha_0,\alpha_i+\alpha_{\tau i})
(\beta_0,\alpha_j+\alpha_{\tau j})
\eta_{\alpha_0,\beta_0}
x_{\alpha_0}^{+}x_{\beta_0}^{+}x_{\tau\alpha_0+\tau\beta_0}^{+}
\\
+~&\tfrac{1}{4}
(\alpha_0,\alpha_j+\alpha_{\tau j})
(\beta_0,\alpha_i+\alpha_{\tau i})
\eta_{\alpha_0,\beta_0}
x_{\beta_0}^{+}x_{\alpha_0}^{+}x_{\tau\alpha_0+\tau\beta_0}^{+}
\\
+~&\tfrac{1}{4}
(\alpha_0,\alpha_j+\alpha_{\tau j})
(\beta_0,\alpha_i+\alpha_{\tau i})
\eta_{\alpha_0,\beta_0}
x_{\tau\alpha_0+\tau\beta_0}^{+}x_{\alpha_0}^{+}x_{\beta_0}^{+}
\\
+~&\tfrac{1}{4}
(\alpha_0,\alpha_i+\alpha_{\tau i})
(\beta_0,\alpha_j+\alpha_{\tau j})
\eta_{\alpha_0,\beta_0}
x_{\tau\alpha_0+\tau\beta_0}^{+}x_{\beta_0}^{+}x_{\alpha_0}^{+}
\\
+~&\tfrac{1}{4}
\Big(
(\alpha_0,\alpha_j+\alpha_{\tau j})
(\beta_0,\alpha_i+\alpha_{\tau i})
-
(\alpha_0,\alpha_i+\alpha_{\tau i})
(\beta_0,\alpha_j+\alpha_{\tau j})
\Big)
\eta_{\alpha_0,\beta_0}
x_{\alpha_0}^{+}
x_{\tau\alpha_0+\tau\beta_0}^{+}
x_{\beta_0}^{+}
\\
+~&\tfrac{1}{4}
\Big(
(\alpha_0,\alpha_i+\alpha_{\tau i})
(\beta_0,\alpha_j+\alpha_{\tau j})
-
(\alpha_0,\alpha_j+\alpha_{\tau j})
(\beta_0,\alpha_i+\alpha_{\tau i})
\Big)
\eta_{\alpha_0,\beta_0}
x_{\beta_0}^{+}
x_{\tau\alpha_0+\tau\beta_0}^{+}
x_{\alpha_0}^{+}
\\
=~&\tfrac{1}{4}
(\alpha_0,\alpha_i+\alpha_{\tau i})
(\beta_0,\alpha_j+\alpha_{\tau j})
\eta_{\alpha_0,\beta_0}
\cdot
[x_{\alpha_0}^{+},
[x_{\tau\alpha_0+\tau\beta_0}^{+},x_{\beta_0}^{+}]]
\\ 
-~&\tfrac{1}{4}
(\alpha_0,\alpha_j+\alpha_{\tau j})
(\beta_0,\alpha_i+\alpha_{\tau i})
\eta_{\alpha_0,\beta_0}
\cdot
[x_{\beta_0}^{+},
[x_{\tau\alpha_0+\tau\beta_0}^{+},x_{\alpha_0}^{+}]],
\end{align*}
which is clearly zero because $\alpha+\tau\alpha\notin\Phi$ for any $\alpha\in\Phi^+$.

Now the verification of \eqref{expandhh1} is complete.

\subsection{Completing the proof of Theorem \ref{thm:embed}}

For Part (1), we have verified that $\varphi$ induces an algebra homomorphism in \S\ref{sec:pfthmB0} (we also need Appendix \ref{sec:app+} for types $\mathsf A_r$ if $r=1,2n$
). Therefore, it remains to show the homomorphism $\varphi$ is injective. 

Recall the filtration on $\Y$ given by setting $\mathrm{deg}\,\xi_{i,r}=\mathrm{deg}\,x_{i,r}^\pm=r$ and the algebra isomorphism $\rho$ from \eqref{ass}. Also recall the filtration on $\Yi$ given by setting $\mathrm{deg}\,h_{i,r}=\mathrm{deg}\,b_{i,r}=r$ and the algebra isomorphism $\varrho$ from \eqref{assi}. Clearly, the homomorphism $\varphi$ is a filtered algebra homomorphism and hence induces the associated graded homomorphism
\begin{align*}
\mathrm{\gr}\,\varphi:~&\mathrm{\gr}\,\Yi\to \mathrm{\gr}\,\Y,\\
&\bar h_{i,r}\mapsto (\bar\xi_{i,r}-(-1)^r\bar\xi_{\tau i,r})+\tfrac14\delta_{r0}\wp_i,\quad \bar{b}_{i,r}\mapsto (\bar x_{i,r}^+-(-1)^r\bar x_{\tau i,r}^-).
\end{align*}
Let
$\iota:\mathrm U(\g[z]^{\check\omega})\hookrightarrow\mathrm U(\g[z])$ be the natural inclusion,
and write
\[
\Psi=\rho^{-1}\circ\gr\varphi\circ\varrho:
\mathrm U(\g[z]^{\check\omega})\longrightarrow\mathrm U(\g[z]).
\]
The map $\Psi$ is not literally $\iota$ in type
$\mathsf A_{2n}$ because of the scalar term above. Equip both
enveloping algebras with their standard PBW filtrations. Then $\Psi$
is PBW-filtered, and the scalar term has PBW degree zero, whereas the
current-algebra generators have PBW degree one. Consequently,
$\gr_{\mathrm{PBW}}\Psi=\iota$. By the PBW theorem, $\iota$ is injective. Hence $\Psi$, and
therefore $\gr\varphi$ and $\varphi$, are injective.

Now we prove Part (2). We identify $\Yi$ as a subalgebra of $\Y$ via $\varphi$. Clearly, we have 
\[
\Delta(b_{i,0})=b_{i,0}\otimes 1+1\otimes b_{i,0},\quad 
\Delta(h_{i,0})=h_{i,0}\otimes 1+1\otimes h_{i,0}-\tfrac14\wp_i.
\]
Note that $[x_{\alpha}^+,x_{\tau \alpha}^+]=0$ for all $\alpha\in\Phi^+$ if $\g$ is not of type $\mathsf A_{2n}$. For these cases we can define $\Phi^+_{i,<}$ and $\Phi^+_{i,>}$ as a partition of $\Phi^+$: $\Phi^+=\Phi^+_{i,<}\sqcup \Phi^+_{i,>}$. Then the formula \eqref{eq:hi1-simle} holds for all types  and hence we have
\begin{align*}
\tl h_{i,1}
={}& h_{i,1}-\tfrac12 h_{i,0}^{2}
\\
={}&
\xi_{i,1}+\xi_{\tau i,1}-\xi_i\xi_{\tau i}
+\tfrac{\wp_i}{4}(\xi_i-\xi_{\tau i})
-\tfrac12
\left(
\xi_i-\xi_{\tau i}+\tfrac{\wp_i}{4}
\right)^2
\\
&\quad
+\sum_{\alpha\in \Phi^+_{i,<}}
(\alpha,\alpha_i)\,
 x_\alpha^+ x_{\tau\alpha}^+
+\sum_{\alpha\in \Phi^+_{i,>}}
(\alpha,\alpha_i)\,
 x_{\tau\alpha}^+ x_\alpha^+
\\
={}&
\tl\xi_{i,1}
+\tl\xi_{\tau i,1}
-\tfrac{1}{32}\wp_i^2
+\sum_{\alpha\in \Phi^+_{i,<}}
(\alpha,\alpha_i)\,
 x_\alpha^+ x_{\tau\alpha}^+
+\sum_{\alpha\in \Phi^+_{i,>}}
(\alpha,\alpha_i)\,
 x_{\tau\alpha}^+ x_\alpha^+.
\end{align*}
Using the identity
\[
\sum_{\alpha\in \Phi^+}
(\alpha,\alpha_{\tau i})\,
 x_\alpha^-\otimes x_\alpha^+=\sum_{\alpha\in \Phi^+}
(\alpha,\alpha_{i})\,
 x_{\tau \alpha}^-\otimes x_{\tau\alpha}^+,
\]
we further obtain
\begin{align*}
\Delta(\tl h_{i,1})
={}&
\tl\xi_{i,1}\otimes 1
+1\otimes\tl\xi_{i,1}
-\sum_{\alpha\in \Phi^+}
(\alpha,\alpha_i)\,
 x_\alpha^-\otimes x_\alpha^+
-\tfrac{1}{32}\wp_i^2
\\
&
+\tl\xi_{\tau i,1}\otimes 1
+1\otimes\tl \xi_{\tau i,1}
-\sum_{\alpha\in \Phi^+}
(\alpha,\alpha_{\tau i})\,
 x_\alpha^-\otimes x_\alpha^+
\\
&
+\sum_{\alpha\in \Phi^+_{i,<}}
(\alpha,\alpha_i)
\bigl(
   x_\alpha^+ x_{\tau\alpha}^+\otimes 1
 + x_\alpha^+\otimes x_{\tau\alpha}^+
 + x_{\tau\alpha}^+\otimes x_\alpha^+
 +1\otimes x_\alpha^+ x_{\tau\alpha}^+
\bigr)
\\
&
+\sum_{\alpha\in \Phi^+_{i,>}}
(\alpha,\alpha_i)
\bigl(
   x_{\tau\alpha}^+ x_\alpha^+\otimes 1
 + x_\alpha^+\otimes x_{\tau\alpha}^+
 + x_{\tau\alpha}^+\otimes x_\alpha^+
 +1\otimes x_{\tau\alpha}^+ x_\alpha^+
\bigr)
\\
={}&
\tl h_{i,1}\otimes 1
+1\otimes\tl h_{i,1}
+\tfrac{1}{32}\wp_i^2
\\
&+\sum_{\alpha\in \Phi^+}
(\alpha,\alpha_i)
\bigl(
 x_{\tau\alpha}^+- x_\alpha^-
\bigr)
\otimes x_\alpha^+
+\sum_{\alpha\in \Phi^+}
(\alpha,\alpha_{i})
\bigl(
 x_\alpha^+- x_{\tau\alpha}^-
\bigr)
\otimes x_{\tau\alpha}^+
\\
={}&
\tl h_{i,1}\otimes 1
+1\otimes\tl h_{i,1}+\tfrac{1}{32}\wp_i^2
+\sum_{\alpha\in \Phi^+}
\bigl(\alpha,\alpha_i+\alpha_{\tau i}\bigr)
\bigl(
 x_\alpha^+- x_{\tau\alpha}^-
\bigr)
\otimes x_{\tau\alpha}^+ .
\end{align*}
Since $\Yi$ is generated by $b_{i,0}$, $h_{i,0}$ and $\tl h_{i,1}$, we conclude that $\Delta(\Yi)\subset \Yi\otimes \Y$. Hence $\Yi$ is a right coideal subalgebra of $\Y$.

Finally, we prove Part (3). We identify $\mathrm{gr}\,\Y$ with   $\mathrm{U}(\g[z])$ via the isomorphism $\rho$ from \eqref{ass}. In the $J$ presentation, by \eqref{eq:J} the images of $x\in \g$ and $J(x)$ are given by $x$ and $xz$, respectively. Note that $\YiJ$ is generated by $\xi_{i}-\xi_{\tau i}$ for $i\in \I$, $b_{\alpha}=x_{\alpha}^+-x_{\tau \alpha}^-$ for $\alpha\in \Phi^+$, and
\[
B(\xi_i+\xi_{\tau i})=J(\xi_i+\xi_{\tau i})-\tfrac14[\xi_i+\xi_{\tau i},C_{\mathfrak k}],\qquad B(y_\alpha)=J(x_{\alpha}^++x_{\tau \alpha}^-)-\tfrac14[x_{\alpha}^++x_{\tau\alpha}^-,C_{\mathfrak k}].
\]
By \eqref{ass}, the images of $\xi_{i}-\xi_{\tau i}$, $b_{\alpha}$, $B(\xi_i+\xi_{\tau i})$ and $B(y_\alpha)$ in the associated graded of $\Y$ (considered as $\mathrm{U}(\g[z])$) are $\xi_{i}-\xi_{\tau i}$, $x_{\alpha}^+-x_{\tau\alpha}^-$, $(\xi_i+\xi_{\tau i})z$, and $(x_{\alpha}^++x_{\tau\alpha}^-)z$, respectively. Thus the image of $\YiJ$ in the associated graded is the subalgebra $\mathrm{U}(\g[z]^{\check \omega})$. Similarly, the image of $\YiJ(\mathscr K,c)$ in the associated graded is also the subalgebra $\mathrm{U}(\g[z]^{\check \omega})$. Note that it is known from above that the image of $\Yi$ in the associated graded of $\Y$ is also the subalgebra $\mathrm{U}(\g[z]^{\check \omega})$. 
Thus, to prove that $\Yi=\YiJ(\mathscr K,c)$, where $c=\tfrac18$ for type $\mathsf A_{2n}$ and $0$ otherwise, it suffices to show that $\Yi\subset \YiJ(\mathscr K,c)$. 

Since $h_{i,0},b_{i,0}\in \YiJ(\mathscr K,c)$ and $\Yi$ is generated by $b_{i,0}$, $h_{i,0}$, and $h_{i,1}$, this reduces to showing that $h_{i,1}\in \YiJ(\mathscr K,c)$. To see this, note that
\begin{align*}
&B(\xi_i+\xi_{\tau i})=J(\xi_i)+J(\xi_{\tau i})
-\tfrac14[\xi_i+\xi_{\tau i},C_{\mathfrak k}]
\\
={}&
\xi_{i,1}
+\frac14\sum_{\alpha\in\Phi^+}
(\alpha,\alpha_i)\{x_\alpha^+,x_\alpha^-\}
-\tfrac12\xi_i^2
+\xi_{\tau i,1}
+\tfrac14\sum_{\alpha\in\Phi^+}
(\alpha,\alpha_{\tau i})\{x_\alpha^+,x_\alpha^-\}
-\tfrac12\xi_{\tau i}^2
\\
&\quad
+\tfrac14
\Big[
 \xi_i+\xi_{\tau i},
 \tfrac12\sum_{\alpha\in\Phi^+}
 (x_\alpha^+-x_{\tau\alpha}^-)
 (x_{\tau\alpha}^+-x_\alpha^-)
\Big]
\\
={}&
h_{i,1}
-\tfrac12(\xi_i-\xi_{\tau i})^2
-\tfrac12\sum_{\alpha\in\Phi^+}
(\alpha,\alpha_i)
\{x_\alpha^+,x_{\tau\alpha}^+\}
+\tfrac14\sum_{\alpha\in\Phi^+}
(\alpha,\alpha_i+\alpha_{\tau i})
\{x_\alpha^+,x_\alpha^-\}
\\
&\quad
+\tfrac18\sum_{\alpha\in\Phi^+}
\Big(
 (\alpha_i+\alpha_{\tau i},\alpha)
 x_\alpha^+x_{\tau\alpha}^+
 +(\alpha_i+\alpha_{\tau i},\tau\alpha)
 x_\alpha^+x_{\tau\alpha}^+
\\
&\hspace{5.3cm}
 -(\alpha_i+\alpha_{\tau i},\tau\alpha)
 x_{\tau\alpha}^-x_\alpha^-
 -(\alpha_i+\alpha_{\tau i},\alpha)
 x_{\tau\alpha}^-x_\alpha^-
\Big)
\\
={}&
h_{i,1}
-\tfrac12(\xi_i-\xi_{\tau i})^2
-\tfrac14\sum_{\alpha\in\Phi^+}
(\alpha,\alpha_i+\alpha_{\tau i})
(x_\alpha^+-x_{\tau\alpha}^-)
(x_{\tau\alpha}^+-x_\alpha^-)
\\
={}&
h_{i,1}
-\tfrac12(\xi_i-\xi_{\tau i})^2
-\tfrac14\sum_{\alpha\in\Phi^+}
(\alpha,\alpha_i)\{b_{\alpha},b_{\tau \alpha}\}\in \YiJ=\YiJ(\mathscr K,0),
\end{align*}
completing the proof of Part (3) for the case when $\g$ is not of type $\mathsf A_{2n}$.

If $\g$ is of type $\mathsf A_{2n}$, using the equality
\[
[\mathscr K,\xi_i+\xi_{\tau i}]=2\bigl(\Theta_i-\omega(\Theta_i)\bigr),
\]
a similar calculation shows that
\[
\varphi(h_{i,1})
=
B_{\mathscr K,c}(\xi_i+\xi_{\tau i})
+\tfrac12(\xi_i-\xi_{\tau i})^2
+\tfrac14\sum_{\alpha\in\Phi^+}
(\alpha,\alpha_i)\{b_\alpha,b_{\tau\alpha}\}
+\tfrac14\bigl(\Theta_i+\omega(\Theta_i)\bigr),
\]
where $\mathscr K$ is given by \eqref{eq:K-def} and $c=\tfrac18$. We also note that $\Theta_i+\omega(\Theta_i)\in\mathfrak k$. Comparing to the non $\mathsf A_{2n}$ case, the extra term $\tfrac14\bigl(\Theta_i-\omega(\Theta_i)\bigr)$ is absorbed into $B_{\mathscr K,c}(\xi_i+\xi_{\tau i})$. Clearly, we have $\Yi\subset \YiJ(\mathscr K,\tfrac18)$. Repeating the argument for the case of non $\mathsf A_{2n}$ type, we finish the proof of Part (3).

Now the proof of Theorem \ref{thm:embed} is complete.

\section{Proof of Theorem \ref{thm:hb}}\label{sec:proof-thmC}
\subsection{Useful relations in Yangians}
We record some equalities that will be used in the proof:
\begin{align}
&[x_i^+,x_i^-(u)]=\xi_i(u)-1,\qquad [x_{i,1}^+,x_i^-(u)]
=
u\xi_i(u)-u-\xi_i,\label{xi+xi-}\\
&[x_i^-,x_i^-(u)]=\tfrac12(\alpha_i,\alpha_i)(x_i^-(u))^2,\label{xixi-}\\
&[\tl\xi_{i,1},x_{j}^\pm(u)]=\pm(\alpha_i,\alpha_j)(ux_j^\pm(u)-x_{j}^\pm),\label{xixj1}\\
&[\xi_i(u),x_{i,0}^-]=-\tfrac12(\alpha_i,\alpha_i)\{\xi_i(u),x_i^-(u)\},\label{xiuxi0}\\
&\bigl[u\xi_i(u),x_j^-\bigr]
-\bigl[\xi_i(u),x_{j,1}^-\bigr]
=
-\tfrac{1}{2}c_{ij}
 \bigl\{\xi_i(u),x_j^-\bigr\},\label{helper31}\\
&\bigl[x_{j,1}^-,x_i^-(u)\bigr]
-\bigl[x_j^-,u x_i^-(u)\bigr]
+\bigl[x_j^-,x_i^-\bigr]
=
-\tfrac{1}{2}c_{ij}
 \bigl\{x_j^-,x_i^-(u)\bigr\},\label{helper32}\\
&[\tl h_{i,1},b_j(u)]=(c_{ij}+c_{\tau i,j})(ub_j(u)-b_{j,0}),\label{hbcom1}\\
&2u h_{i}(u)
=
2u+2h_{i,0}
+\bigl[b_{\tau i,0},u b_i(u)-b_{i,0}\bigr]
-\bigl[b_{\tau i,1},b_i(u)\bigr]
+\tfrac{1}{2}c_{\tau i,i}\bigl\{b_{\tau i,0},b_i(u)\bigr\}.\label{bi0biu}
\end{align}
The equalities \eqref{xi+xi-}, \eqref{xixj1}, \eqref{helper31}, \eqref{helper32} and \eqref{hbcom1} follow from \eqref{eq:relXX},  \eqref{tlxicom}, \eqref{eq:relexHX}, \eqref{eq:relexXX} and \eqref{eq:hi1bjr-new}, respectively. The equalities \eqref{xixi-} and \eqref{xiuxi0} can be found in \cite[\S2.4]{GTL16} while the equality \eqref{bi0biu} is obtained from \eqref{qs5}. 

\subsection{The case excluding type $\mathsf A_{2n}$}\label{sec:noA2n}
In this subsection, we deal with the case of type $\mathsf A_{2n-1}$, $\mathsf D_n$, and $\mathsf E_6$. The calculations are similar to those performed in \cite{Lu26min}.
\subsubsection{}\label{ssec:b-est} We first establish \eqref{eq:b-est} whose component-wise formula is given by
\be
b_{i,r}\equiv x_{i,r}^+-(-1)^rx_{\tau i,r}^--\tfrac12\sum_{0\lle s<r}(-1)^s \{x_{i,r-s-1}^+,\xi_{\tau i,s} \}.
\ee
Here $\equiv$ stands for equality modulo $\Y_{\alpha_i+Q_+}^{\gge 0}[\![u^{-1}]\!]$. We prove it by induction on $r$. The base cases $r=0,1$ are clear from \eqref{bi0-embedding} and \eqref{bi1-embedding}. 

We only consider the case $\tau i\ne i$ as the calculation for the case $\tau i=i$ is very similar (also discussed in \cite{Lu26min}). Note that by \eqref{hi1-embedding} we have $[\tl h_{i,1},\Y_{\alpha_i+Q_+}^{\gge 0}]\subset \Y_{\alpha_i+Q_+}^{\gge 0}$. Together with \eqref{hbcom1}, it is not hard to see that the induction step is reduced to showing
\beq\label{helper5}
\begin{split}
\big[\tl h_{i,1},&\,\tfrac12 \{x_i^+(u),\xi_{\tau i}(-u)\}+x_{\tau i}^-(-u)\big]\\
&\equiv (c_{ii}+c_{\tau i,i})u\big(\tfrac12\{x_i^+(u),\xi_{\tau i}(-u)\}+x_{\tau i}^-(-u)\big)-(c_{ii}+c_{\tau i,i})(x_i^+-x_{\tau i}^-).
\end{split}
\eeq
By \eqref{tlxidef} and \eqref{hi1-embedding}, the LHS of \eqref{helper5} is equal to
\begin{align*}
&\Big[\tl\xi_{i,1}+\tl\xi_{\tau i,1}+\hf \sum_{\alpha\in\Phi^+}(\alpha_i,\alpha)\big\{x_\alpha^+,x_{\tau\alpha}^+\big\},\tfrac12 \big\{x_i^+(u),\xi_{\tau i}(-u)\big\}+x_{\tau i}^-(-u)\Big]\\
\stackrel{\eqref{xixj1}}{\equiv}\,&\hf c_{ii}\big\{ux_i^{+}(u)-x_i^+,\xi_{\tau i}(-u)\big\}+ c_{\tau i,i}\big(ux_{\tau i}^-(-u)+x_{\tau i}^-\big)
\\ \hskip0.22cm&
+\hf c_{\tau i,i}\big\{ux_i^{+}(u)-x_i^+,\xi_{\tau i}(-u)\big\}+ c_{i,i}\big(ux_{\tau i}^-(-u)+x_{\tau i}^-\big)\\
\hskip0.22cm&+\hf(c_{ii}+c_{\tau i,i})\big\{[x_{\tau i}^+,x_{\tau i}^-(-u)],x_i^+\big\}\\
\stackrel{\eqref{xi+xi-}}{=}\,&\hf(c_{ii}+c_{\tau i,i})\Big(\big\{ux_i^{+}(u)-x_i^+,\xi_{\tau i}(-u)\big\}+ 2\big(ux_{\tau i}^-(-u)+x_{\tau i}^-\big)+\big\{\xi_{\tau i}(-u)-1,x_i^+\big\}\Big)\\
=\hskip0.22cm&(c_{ii}+c_{\tau i,i})u\big(\tfrac12\{x_i^+(u),\xi_{\tau i}(-u)\}+x_{\tau i}^-(-u)\big)-(c_{ii}+c_{\tau i,i})(x_i^+-x_{\tau i}^-)
\end{align*}
as required.

\subsubsection{} \label{sssec:copro-b}
For \eqref{eq:b-co-est}, the corresponding component-wise formula is given by 
\[
\Delta(b_{i,r})\equiv \sum_{0\lle s<r}(-1)^{s+1} b_{i,r-s-1}\otimes \xi_{\tau i,s}+b_{i,r}\otimes 1+1\otimes b_{i,r}.
\]
Here $\equiv$ stands for equality modulo $\Yi\otimes \Y_{Q_+}[\![u^{-1}]\!]$. We prove it again by induction on $r$. The base case $r=0$ is clear while the case $r=1$ follows from \eqref{helper98}. 

Again, we only consider the case $\tau i\ne i$.  Note that by \eqref{coprohi1} we have $[\Delta(\tilde{h}_{i,1}),\Yi\otimes \Y_{Q_+}]\subset \Yi\otimes \Y_{Q_+}$. Together with the relation 
\begin{align*}
\big[\Delta(\tl h_{i,1}),\Delta(b_i(u))\big]= (c_{ii}+c_{\tau i,i})\big(u\Delta(b_i(u))-\Delta(b_{i,0})\big)
\end{align*}
that follows from \eqref{hbcom1}, 
it is not hard to see that the induction step is reduced to proving
\beq\label{helper21}
\begin{split}
\Big[\tl h_{i,1}\otimes 1&+ 1\otimes \tl h_{i,1}+ \sum_{\alpha\in\Phi^+}(\alpha+\tau \alpha,\alpha_i)\,b_\alpha\otimes x_{\tau\alpha}^+,b_i(u)\otimes \xi_{\tau i}(-u)+1\otimes b_i(u)\Big]\\
\equiv &~(c_{ii}+c_{\tau i,i})\big(u(b_i(u)\otimes \xi_{\tau i}(-u)+1\otimes b_i(u))-(b_{i,0}\otimes 1+1\otimes b_{i,0})\big).
\end{split}
\eeq
We list all the terms on the LHS of \eqref{helper21} as follows:
\begin{align*}
&[\tl h_{i,1}\otimes 1, b_i(u)\otimes \xi_{\tau i}(-u)]\stackrel{\eqref{hbcom1}}{=}(c_{ii}+c_{\tau i,i})(ub_i(u)-b_{i,0})\otimes \xi_{\tau i}(-u),\\
&[\tl h_{i,1}\otimes 1, 1\otimes b_i(u)]=0,\qquad [1\otimes \tl h_{i,1}, b_i(u)\otimes \xi_{\tau i}(-u)]\stackrel{\eqref{hi1-embedding}}{\equiv}0,\\
&[1\otimes \tl h_{i,1}, 1\otimes b_i(u)]\stackrel{\eqref{hbcom1}}{=} (c_{ii}+c_{\tau i,i}) 1\otimes (ub_i(u)-b_{i,0}),\\
&[b_\alpha\otimes x_{\tau \alpha}^+,b_i(u)\otimes \xi_{\tau i}(-u)]\equiv 0,\\
&[b_\alpha\otimes x_{\tau \alpha}^+,1\otimes b_i(u)]\stackrel{\eqref{eq:b-est}}{\equiv} 0,\qquad \text{ if }\alpha\ne \alpha_{i},\\
&[b_{i,0}\otimes x_{\tau i}^+,1\otimes b_i(u)]\overset{\eqref{eq:b-est}}{\underset{\eqref{xi+xi-}}{\equiv}}b_{i,0}\otimes (\xi_{\tau i}(-u)-1).
\end{align*}
Summing them up (with suitable multiples), we easily deduce \eqref{helper21}.

\subsubsection{}\label{ssec:h-est} Next we prove \eqref{eq:h-est}. Let $\equiv$ stand for equality modulo $\Y_{Q_+}[\![u^{-1}]\!]$. By \eqref{bi0biu}, we have
\beq\label{todo1}
2u h_j(u)
=
2u+2h_{j,0}
+\bigl[b_{i,0},u b_j(u)-b_{j,0}\bigr]
-\bigl[b_{i,1},b_j(u)\bigr]
+\tfrac{1}{2}c_{ij}\bigl\{b_{i,0},b_j(u)\bigr\},
\eeq
where $j=\tau i$.

The relation \eqref{helper31} implies that
\begin{align*}
-&\tfrac{1}{2}u
 \bigl\{
 x_j^+(u),
 [x_j^-,\xi_i(-u)]
 \bigr\}
-\tfrac{1}{2}
 \bigl\{
 x_j^+(u),
 [x_{j,1}^-,\xi_i(-u)]
 \bigr\}
\\
&\qquad
-\tfrac{1}{4}c_{ij}
 \bigl\{
 x_j^-,
 \{x_j^+(u),\xi_i(-u)\}
 \bigr\}
\\
={}&
\tfrac{1}{4}c_{ij}
 \bigl\{
 x_j^+(u),
 \{\xi_i(-u),x_j^-\}
 \bigr\}
-\tfrac{1}{4}c_{ij}
 \bigl\{
 x_j^-,
 \{x_j^+(u),\xi_i(-u)\}
 \bigr\}\\
={}&
\tfrac{1}{4}c_{ij}
\bigl(
 [x_j^+(u),x_j^-]\xi_i(-u)
-\xi_i(-u)[x_j^+(u),x_j^-]
\bigr)\\
={}&\tfrac{1}{4}c_{ij}
\bigl(
 (\xi_j(u)-1)\xi_i(-u)
-\xi_i(-u)(\xi_j(u)-1)
\bigr)=0.
\end{align*}

Now we calculate the estimate for each term on the RHS of \eqref{todo1} and then add them up. By \eqref{xi+xi-}, we have
\begin{align*}
\bigl[b_{i,0},u b_j(u)-b_{j,0}\bigr]
&\equiv
\big[
x_i^+-x_j^-,
\tfrac{1}{2}u
 \bigl\{x_j^+(u),\xi_i(-u)\bigr\}
+u x_i^-(-u)-x_j^++x_i^-
\big]
\\
&\equiv
u\bigl(\xi_i(-u)-1\bigr)+\xi_i
+\tfrac{1}{2}u
 \bigl\{\xi_j(u)-1,\xi_i(-u)\bigr\}
\\
&\quad
-\tfrac{1}{2}u
 \bigl\{
 x_j^+(u),
 [x_j^-,\xi_i(-u)]
 \bigr\}
-u[x_j^-,x_i^-(-u)]
-\xi_j-[x_j^-,x_i^-],
\end{align*}
\begin{align*}
-\bigl[b_{i,1},b_j(u)\bigr]
&\equiv
-\big[
x_{i,1}^+ +x_{j,1}^-
-\tfrac{1}{2}\{x_i^+,\xi_j\},
\tfrac{1}{2}\{x_j^+(u),\xi_i(-u)\}
+x_i^-(-u)
\big]
\\
&\equiv
u\xi_i(-u)+\xi_i-u
+\tfrac{1}{2}
 \bigl\{
 u\xi_j(u)-\xi_j-u,\xi_i(-u)
 \bigr\}-[x_{j,1}^-,x_i^-(-u)]
\\
&\quad
-\tfrac{1}{2}
 \bigl\{
 x_j^+(u),
 [x_{j,1}^-,\xi_i(-u)]
 \bigr\}
+\tfrac{1}{2}
 \bigl\{\xi_i(-u)-1,\xi_j\bigr\}
-\tfrac{1}{2}c_{ij}
 \bigl\{x_i^+,x_i^-(-u)\bigr\},
\end{align*}
and
\begin{align*}
&\tfrac{1}{2}c_{ij}\bigl\{b_{i,0},b_j(u)\bigr\}\\
&\equiv
\tfrac{1}{2}c_{ij}
\left\{
x_i^+-x_j^-,
\tfrac{1}{2}\{x_j^+(u),\xi_i(-u)\}
+x_i^-(-u)
\right\}
\\
&\equiv
\tfrac{1}{2}c_{ij}
 \bigl\{x_i^+,x_i^-(-u)\bigr\}
-\tfrac{1}{4}c_{ij}
 \bigl\{
 x_j^-,
 \{x_j^+(u),\xi_i(-u)\}
 \bigr\}
-\tfrac{1}{2}c_{ij}
 \bigl\{x_j^-,x_i^-(-u)\bigr\}.
\end{align*}
Summing them up and using the observations above, we obtain that the RHS of \eqref{todo1} is equal to $2u\xi_j(u)\xi_i(-u)$ modulo $\Y_{Q_+}[\![u^{-1}]\!]$, completing the proof of \eqref{eq:h-est}.

\subsubsection{} \label{sssec:copro-h}

Finally, we prove \eqref{eq:h-co-est}. For $j\in\I$, set
\begin{equation*}
Q_{j,+}:=
\Big\{\sum_{k\in\I}m_k\alpha_k
\mathrel{\Big|}m_j>0,\ m_k\in\bZ\text{ for }k\in\I\Big\}.
\end{equation*}

\begin{lem}\label{lem:cap=0}
For every $j\in\I$, we have $\Yi\cap\Y_{Q_{j,+}}=\{0\}$ and $
\Yi\cap\Y_{Q_+}=\{0\}$.
\end{lem}

\begin{proof}
The proof is parallel to 
\cite[Lemma~5.4]{Lu26min}.

\end{proof}

Fix $i\in\I$ and abbreviate
\[\varsigma_i(u):=1+\tfrac{1}{4u}\wp_i,
\qquad
\mathcal H_i(u):=\xi_i(u)\xi_{\tau i}(-u).
\]
By \eqref{eq:h-est}, there is a series
$\theta_i(u)\in\Y_{Q_+}[\![u^{-1}]\!]$ such that
\begin{equation}\label{eq:hi-theta-qs}
h_i(u)=\varsigma_i(u)\mathcal H_i(u)+\theta_i(u).
\end{equation}
For $j\in\I$, Lemma~\ref{lem:copro} gives
\[
\Delta(\xi_j(v))
=\xi_j(v)\otimes\xi_j(v)+\mc Z_j(v),
\qquad
\mc Z_j(v)\in
\Y_{Q_-}\otimes\Y_{Q_+}[\![v^{-1}]\!].
\]
Consequently,
\begin{equation}\label{eq:copro-Ai-qs}
\Delta(\mc H_i(u))=\mc H_i(u)\otimes \mc H_i(u)+\mc E_i(u),
\end{equation}
where
\begin{align*}
\mc E_i(u)=
\bigl(\xi_i(u)\otimes\xi_i(u)\bigr)\mc Z_{\tau i}(-u)
&+\mc Z_i(u)\bigl(\xi_{\tau i}(-u)\otimes\xi_{\tau i}(-u)\bigr)\\
&+\mc Z_i(u)\mc Z_{\tau i}(-u)
\in\Y_{Q_-}\otimes\Y_{Q_+}[\![u^{-1}]\!].
\end{align*}
Define $\Theta_i(u):=
\Delta(h_i(u))-h_i(u)\otimes \mc H_i(u)$.
Equations \eqref{eq:hi-theta-qs} and \eqref{eq:copro-Ai-qs}
imply that
\begin{equation}\label{eq:Theta-qs}
\Theta_i(u)
=\varsigma_i(u)\mc E_i(u)
+\Delta(\theta_i(u))-\theta_i(u)\otimes \mc H_i(u).
\end{equation}
Since $\Yi$ is a right coideal subalgebra of $\Y$, we also have
\begin{equation}\label{eq:Theta-coideal-qs}
\Theta_i(u)\in\Yi\otimes\Y[\![u^{-1}]\!].
\end{equation}

We now argue degree by degree. For $\beta\in Q$, let
$\operatorname{pr}_\beta:\Y\to\Y_\beta$ be the projection onto the subspace of 
degree $\beta$, and set
\[
\Theta_{i,\beta}(u)
:=(\mathrm{id}\otimes\operatorname{pr}_\beta)\Theta_i(u).
\]
Suppose that $\beta\notin Q_+$. The first term on the right-hand side
of \eqref{eq:Theta-qs} has zero projection to second tensor-factor
degree $\beta$, because the second tensor factor of $\mc E_i(u)$ has
degree in $Q_+$.

Recall that the coproduct preserves the total $Q$-grading. Hence, if
a homogeneous term of $\theta_i(u)$ has root degree $\gamma\in Q_+$,
then every contribution of
$\Delta(\theta_i(u))-\theta_i(u)\otimes \mc H_i(u)$ whose second tensor
factor has degree $\beta$ has first tensor-factor degree
$\gamma-\beta$. If $\beta=0$, it follows from
\eqref{eq:Theta-coideal-qs} that $\Theta_{i,0}(u)
\in
\bigl(\Yi\cap\Y_{Q_+}\bigr)
\otimes\Y_0[\![u^{-1}]\!]
=0$
by Lemma~\ref{lem:cap=0}.

It remains to consider $\beta\neq0$. Since
$\beta\notin Q_+$, we can write
$\beta=\sum_{k\in\I}m_k\alpha_k$ with $m_j<0$ for some $j\in\I$.
For every $\gamma\in Q_+$, the coefficient of $\alpha_j$ in
$\gamma-\beta$ is strictly positive. Therefore, again by Lemma~\ref{lem:cap=0}, we have
\[
\Theta_{i,\beta}(u)
\in
\bigl(\Yi\cap\Y_{Q_{j,+}}\bigr)
\otimes\Y_\beta[\![u^{-1}]\!]
=0.
\]
We have thus shown that
$\Theta_{i,\beta}(u)=0$ for every $\beta\notin Q_+$. Consequently, $\Theta_i(u)
\in\Yi\otimes\Y_{Q_+}[\![u^{-1}]\!]$,
which is precisely \eqref{eq:h-co-est}.

\subsection{The case of type $\mathsf A_{2n}$}
In this subsection, we complete the proof of Theorem \ref{thm:hb} for the special case of type $\mathsf A_{2n}$.
\subsubsection{}
We start with establishing \eqref{eq:b-est2}. The case for $i\ne n,n+1$ is essentially the same as \S\ref{ssec:b-est}. We only prove the case $i=n$ as the case $i=n+1$ is similar.

As in \S\ref{ssec:b-est}, we need to prove
\begin{align*}
&\bigl[\tl h_{n,1},
\xi_{n+1}(-u)x_n^+(u)+x_{n+1}^-(-u)\bigr]\\
&\qquad\equiv
u\bigl(\xi_{n+1}(-u)x_n^+(u)+x_{n+1}^-(-u)\bigr)
-\bigl(x_n^+-x_{n+1}^-\bigr).
\end{align*}
Indeed,
\begin{align*}
\mathrm{LHS}
&\equiv
\bigl[
\tl\xi_{n,1}+\tl\xi_{n+1,1}
+\tfrac12\{x_n^+,x_{n+1}^+\}
-\tfrac12[x_n^+,x_{n+1}^+],
\xi_{n+1}(-u)x_n^+(u)+x_{n+1}^-(-u)
\bigr]\\
&\equiv
\bigl[
\tl\xi_{n,1}+\tl\xi_{n+1,1}+x_{n+1}^+x_n^+,
\xi_{n+1}(-u)x_n^+(u)+x_{n+1}^-(-u)
\bigr]\\
&\equiv
\xi_{n+1}(-u)\bigl(ux_n^+(u)-x_n^+\bigr)
-\bigl(-ux_{n+1}^-(-u)-x_{n+1}^-\bigr)
+\bigl(\xi_{n+1}(-u)-1\bigr)x_n^+\\
&\equiv
u\bigl(\xi_{n+1}(-u)x_n^+(u)+x_{n+1}^-(-u)\bigr)
-\bigl(x_n^+-x_{n+1}^-\bigr)
\equiv\mathrm{RHS}.
\end{align*}

\subsubsection{}Then we prove \eqref{eq:h-est} for type $\mathsf A_{2n}$ with the same strategy as in \S\ref{ssec:h-est}. Note that we have
\beq\label{eq:todo2}
\begin{split}
2u h_{n+1}(u)
={}&
2u+2h_{n+1,0}\\
&+\bigl[b_{n,0},u b_{n+1}(u)-b_{n+1,0}\bigr]
-\bigl[b_{n,1},b_{n+1}(u)\bigr]
-\tfrac{1}{2}\bigl\{b_{n,0},b_{n+1}(u)\bigr\}.
\end{split}
\eeq
Then we estimate the terms in the second line of \eqref{eq:todo2}. We have
\begin{align*}
&\bigl[b_{n,0},u b_{n+1}(u)-b_{n+1,0}\bigr]\\
\equiv{}&
\bigl[
 x_n^+-x_{n+1}^-,
 u x_{n+1}^+(u)\xi_n(-u)+u x_n^-(-u)-x_{n+1}^++x_n^-
\bigr]\\
\equiv{}&
u\bigl(\xi_n(-u)-1\bigr)+\xi_n
+u\bigl(\xi_{n+1}(u)-1\bigr)\xi_n(-u)
-u\bigl[x_{n+1}^-,x_n^-(-u)\bigr]\\
&\hskip3.03cm
-\xi_{n+1}-\bigl[x_{n+1}^-,x_n^-\bigr]
-u x_{n+1}^+(u)\bigl[x_{n+1}^-,\xi_n(-u)\bigr],
\end{align*}
\begin{align*}
&-\bigl[b_{n,1},b_{n+1}(u)\bigr]\\
\equiv{}&
-\bigl[
 x_{n,1}^+ +x_{n+1,1}^- -\xi_{n+1}x_n^+,
 x_{n+1}^+(u)\xi_n(-u)+x_n^-(-u)
\bigr]\\
\equiv{}&
-\bigl(-u\xi_n(-u)+u-\xi_n\bigr)
+\bigl(u\xi_{n+1}(u)-u-\xi_{n+1}\bigr)\xi_n(-u)\\
&\hskip2.03cm
-x_{n+1}^+(u)\bigl[x_{n+1,1}^-,\xi_n(-u)\bigr]
-\bigl[x_{n+1,1}^-,x_n^-(-u)\bigr]\\
&\hskip3.03cm
+\xi_{n+1}\bigl(\xi_n(-u)-1\bigr)
+\bigl[\xi_{n+1},x_n^-(-u)\bigr]x_n^+,
\end{align*}
and
\begin{align*}
&-\tfrac12\bigl\{b_{n,0},b_{n+1}(u)\bigr\}\\
\equiv{}&-\tfrac12
\bigl\{
 x_n^+-x_{n+1}^-,
 x_{n+1}^+(u)\xi_n(-u)+x_n^-(-u)
\bigr\}\\
\equiv{}&
-\tfrac12\bigl\{x_n^+,x_n^-(-u)\bigr\}
+\tfrac12\bigl\{x_{n+1}^-,x_{n+1}^+(u)\xi_n(-u)\bigr\}
+\tfrac12\bigl\{x_{n+1}^-,x_n^-(-u)\bigr\}.
\end{align*}
Note also that $2u+2h_{n+1,0}=2u+2\xi_{n+1}-2\xi_n-\tfrac12$. Hence, applying \eqref{helper32}, the RHS of \eqref{eq:todo2} is rewritten as
\begin{align*}
&2u\xi_{n+1}(u)\xi_n(-u)
-u x_{n+1}^+(u)\bigl[x_{n+1}^-,\xi_n(-u)\bigr]
-x_{n+1}^+(u)\bigl[x_{n+1,1}^-,\xi_n(-u)\bigr]\\
&\qquad
+\bigl[\xi_{n+1},x_n^-(-u)\bigr]x_n^+
-\tfrac12 \{x_n^+,x_n^-(-u) \}
+\tfrac12 \{x_{n+1}^-,x_{n+1}^+(u)\xi_n(-u) \}
-\tfrac12,
\end{align*}
which transforms to (with the help of \eqref{helper31})
\begin{align*}
&2u\xi_{n+1}(u)\xi_n(-u)
-\tfrac12x_{n+1}^+(u)\bigl\{\xi_n(-u),x_{n+1}^-\bigr\}
+\tfrac12\bigl[x_n^-(-u),x_n^+\bigr]\\
&\qquad
+\tfrac12x_{n+1}^-x_{n+1}^+(u)\xi_n(-u)
+\tfrac12x_{n+1}^+(u)\xi_n(-u)x_{n+1}^-
-\tfrac12\\
={}&
2u\xi_{n+1}(u)\xi_n(-u)
-\tfrac12x_{n+1}^+(u)\xi_n(-u)x_{n+1}^-
-\tfrac12x_{n+1}^+(u)x_{n+1}^-\xi_n(-u)\\
&
-\tfrac12\bigl(\xi_n(-u)-1\bigr)
+\tfrac12x_{n+1}^-x_{n+1}^+(u)\xi_n(-u)
+\tfrac12x_{n+1}^+(u)\xi_n(-u)x_{n+1}^-
-\tfrac12\\
={}&
2u\xi_{n+1}(u)\xi_n(-u)
-\tfrac12\bigl[x_{n+1}^+(u),x_{n+1}^-\bigr]\xi_n(-u)
-\tfrac12\xi_n(-u)\\
={}&
2u\xi_{n+1}(u)\xi_n(-u)
-\tfrac12\bigl(\xi_{n+1}(u)-1\bigr)\xi_n(-u)
-\tfrac12\xi_n(-u)\\
={}&\left(2u-\tfrac12\right)\xi_{n+1}(u)\xi_n(-u),
\end{align*}
verifying \eqref{eq:h-est} for $i=n+1$. The remaining cases are similar and hence we omit the details.

\subsubsection{}
Note that by \eqref{coprohi1}, we find that the formula for $\Delta(\tl h_{i,1})$ in the case of type $\mathsf A_{2n}$ only differs from the formula for $\Delta(\tl h_{i,1})$ in other types by a scalar. Moreover, the base case $\Delta(b_{i,0})$ is precisely the same for all $\mathsf{ADE}$ types. Therefore, by exactly the same calculation in \S\ref{sssec:copro-b}, we conclude that the coproduct estimate for $\Delta(b_i(u))$ has the same form as in the other cases by induction.

The proof of \eqref{eq:h-co-est} in the case of type $\mathsf A_{2n}$ is parallel to the other cases as discussed in \S\ref{sssec:copro-h}. Now the proof of Theorem \ref{thm:hb} is complete.

\appendix

\section{Proof of Proposition \ref{prop:red}, Part II}\label{sec:pf-prop2}
Throughout this section, let $\g=\mathfrak{sl}_{2n+1}$, $\I=\{1,\ldots,2n\}$, $\tau i=2n+1-i$. Thus $\tau n=n+1$, $c_{n,n+1}=c_{n+1,n}=-1$, whereas $c_{i,\tau i}=0$
for $i\notin\{n,n+1\}$.
Recall from \eqref{eq:gamma-def} that
\begin{equation*}
\gamma_{ii}=c_{ii}+c_{\tau i,i}
=
\begin{cases}
1,&\text{ if }i=n,n+1,\\
2,&\text{ if }i\notin\{n,n+1\}.
\end{cases}
\end{equation*}
As before, define the higher modes recursively by
\begin{align}
\bb_{i,r+1}
&:=
\gamma_{ii}^{-1}[\bh_{i,1},\bb_{i,r}],
\notag
\\
\bh_{i,r+1}&:=[\bb_{\tau i,0},\bb_{i,r+1}], &\text{ if }i\ne n,n+1,\label{eq:A2n-rec-h-sim}\\
2\bh_{i,r+1}
&:=
[\bb_{\tau i,0},\bb_{i,r+1}]
-
[\bb_{\tau i,1},\bb_{i,r}], &\text{ if }i= n,n+1,
\label{eq:A2n-rec-h}
\end{align}
for $i\in\I$ and $r\in\bN$.  Notice that
\eqref{eq:A2n-rec-h}, rather than the simplified formula \eqref{eq:A2n-rec-h-sim} valid when
$c_{i,\tau i}=0$, has to be used on the central orbit
$\{n,n+1\}$. 

For these recursively defined modes, relation \eqref{hi1bjr}
remains valid. Moreover, the proof of Lemma~\ref{lem:bb} applies
to vertices in distinct $\tau$-orbits. We shall use
Lemma~\ref{lem:h2bj} only when its first index belongs to a
noncentral orbit, so that $c_{i,\tau i}=0$.

The proof naturally separates into two parts.  We first record the
rank-two reconstruction.  We then show that, for $n\gge2$, its two
additional relations are forced by the neighboring noncentral
$\tau$-orbit.

\subsection{The rank-two reconstruction}\label{ssec:A2-reduction}
In this subsection, we consider the case $n=1$. In this case, we shall prove that the degree-zero and degree-one relations, the
finite type Serre relations, and
\eqref{eq:A2-extra-Serre}--\eqref{eq:A2-extra-Cartan} imply all the
relations in Lemma~\ref{lem:full-gr}.

Define recursively
\begin{align}
\bb_{1,r+1}
&=
[\bh_{1,1},\bb_{1,r}],
\qquad 
\bb_{2,r+1}=
[\bh_{1,1},\bb_{2,r}],
\label{eq:A2-rec-b}
\\
2\bh_{1,r+1}
&=
[\bb_{2,0},\bb_{1,r+1}]
-
[\bb_{2,1},\bb_{1,r}],
\label{eq:A2-rec-h1}
\\
2\bh_{2,r+1}
&=
[\bb_{1,0},\bb_{2,r+1}]
-
[\bb_{1,1},\bb_{2,r}]
\label{eq:A2-rec-h2}
\end{align}
for $r\in\bN$.  For $r=0$, the last two formulas agree with \begin{equation}\label{eq:A2-mixed-low}
[\bb_{1,1},\bb_{2,0}]
-
[\bb_{1,0},\bb_{2,1}]
=
-2\bh_{1,1},
\end{equation}
and $\bh_{2,1}=\bh_{1,1}$. Note that \eqref{hi1bjr} also works for this case.

Set
\begin{align*}
\mathsf S_+
:=
\big[\bb_{1,1},[\bb_{1,1},\bb_{2,0}]\big],
\qquad 
\mathsf S_-
:=
\big[\bb_{2,1},[\bb_{2,1},\bb_{1,0}]\big].
\end{align*}
Then \eqref{eq:A2-extra-Serre} says $\mathsf S_++\mathsf S_-=0$. Applying $\ad(\bh_{1,0})$ to it, we obtain $3\mathsf S_+-3\mathsf S_-=0$. Thus, we have 
\begin{equation}\label{eq:A2-two-shifted-Serre}
\big[\bb_{1,1},[\bb_{1,1},\bb_{2,0}]\big]=0,
\qquad
\big[\bb_{2,1},[\bb_{2,1},\bb_{1,0}]\big]=0.
\end{equation}
We shall use these two relations repeatedly below.

\subsubsection*{Step 1. The degree-two Cartan mode}

Applying $\ad(\bh_{1,1})$ to
\eqref{eq:A2-mixed-low} and using
\eqref{eq:A2-rec-b}, we obtain
\beq\label{eq:A2-b10b20}
[\bb_{1,2},\bb_{2,0}]
=
[\bb_{1,0},\bb_{2,2}]. 
\eeq
Using \eqref{eq:A2-rec-h1} with $r=1$, we find
\begin{align}
2\bh_{1,2}
=
[\bb_{2,0},\bb_{1,2}]
-
[\bb_{2,1},\bb_{1,1}]=
[\bb_{1,1},\bb_{2,1}]
-
[\bb_{1,0},\bb_{2,2}].
\label{eq:A2-h12}
\end{align}
Similarly, \eqref{eq:A2-rec-h2} gives
\begin{align*}
2\bh_{2,2}
=
[\bb_{1,0},\bb_{2,2}]
-
[\bb_{1,1},\bb_{2,1}]=
-2\bh_{1,2}.
\end{align*}
Applying $\mathrm{ad}(\bh_{1,1})$ to \eqref{eq:A2-h12} and using \eqref{eq:A2-extra-Cartan}, we have
\begin{equation}\label{eq:A2-h11h12}
[\bh_{1,1},\bh_{1,2}]
=
0.
\end{equation}

We next compute the action of $\bh_{1,2}$ on the degree-zero
generators. 

The finite type Serre relation is
\begin{equation}\label{eq:A2-Serre-10}
[\bb_{1,0},[\bb_{1,0},\bb_{2,0}]]
=
4\bb_{1,0}.
\end{equation}
Applying $\ad(\bh_{1,1})$ to
\eqref{eq:A2-Serre-10}, we obtain
\begin{align*}
4\bb_{1,1}
=
[\bb_{1,1},[\bb_{1,0},\bb_{2,0}]]
+
[\bb_{1,0},[\bb_{1,1},\bb_{2,0}]]
+
[\bb_{1,0},[\bb_{1,0},\bb_{2,1}]].
\end{align*}
Since $[\bb_{1,0},\bb_{1,1}]=0$ by \eqref{tcbb}, the Jacobi identity gives
\begin{equation}\label{eq:A2-helper-A}
2[\bb_{1,0},[\bb_{1,1},\bb_{2,0}]]
+
[\bb_{1,0},[\bb_{1,0},\bb_{2,1}]]
=
4\bb_{1,1}.
\end{equation}

On the other hand, applying $\ad(\bb_{1,0})$ to
\eqref{eq:A2-mixed-low}, we obtain
\begin{align*}
[\bb_{1,0},[\bb_{1,1},\bb_{2,0}]]
-
[\bb_{1,0},[\bb_{1,0},\bb_{2,1}]]
=
-2[\bb_{1,0},\bh_{1,1}]
=
2\bb_{1,1}.
\end{align*}
Combining this with \eqref{eq:A2-helper-A}, we obtain
\begin{align}\label{eq:A2-helper-B}
[\bb_{1,0},[\bb_{1,1},\bb_{2,0}]]
=
2\bb_{1,1},
\qquad 
[\bb_{1,0},[\bb_{1,0},\bb_{2,1}]]
=
0.
\end{align}

We now apply $\ad(\bh_{1,1})$ to the second relation in 
\eqref{eq:A2-helper-B}.  Using
\eqref{eq:A2-two-shifted-Serre}, we find
\begin{align*}
0
={}&
[\bb_{1,1},[\bb_{1,0},\bb_{2,1}]]
+
[\bb_{1,0},[\bb_{1,1},\bb_{2,1}]]
+
[\bb_{1,0},[\bb_{1,0},\bb_{2,2}]]
\\
={}&
-2\bb_{1,2}
-2\bb_{1,2}
+
[\bb_{1,0},[\bb_{1,0},\bb_{2,2}]].
\end{align*}
Here we used
\begin{align}\label{helper865}
[\bb_{1,0},[\bb_{1,1},\bb_{2,1}]]=[\bb_{1,1},[\bb_{1,0},\bb_{2,1}]]
=
[\bb_{1,1},
[\bb_{1,1},\bb_{2,0}]+2\bh_{1,1}]
=
-2\bb_{1,2}.
\end{align}
Consequently,
\begin{equation}\label{eq:A2-b10b10b22}
[\bb_{1,0},[\bb_{1,0},\bb_{2,2}]]
=
4\bb_{1,2}.
\end{equation}

Using \eqref{eq:A2-h12}, we now compute
\begin{align*}
2[\bh_{1,2},\bb_{1,0}]
={}&
[[\bb_{1,1},\bb_{2,1}],\bb_{1,0}]
-
[[\bb_{1,0},\bb_{2,2}],\bb_{1,0}]
\\
={}&
-[\bb_{1,1},[\bb_{1,0},\bb_{2,1}]]
+
[\bb_{1,0},[\bb_{1,0},\bb_{2,2}]]
\\
={}&
2\bb_{1,2}+4\bb_{1,2}
=
6\bb_{1,2}.
\end{align*}
Thus
\begin{equation}\label{eq:A2-h12b10}
[\bh_{1,2},\bb_{1,0}]
=
3\bb_{1,2}.
\end{equation}

Similarly, we have 
\begin{equation}\label{eq:A2-h22b20}
[\bh_{2,2},\bb_{2,0}]
=
3\bb_{2,2},\qquad [\bh_{1,2},\bb_{2,0}]
=
-3\bb_{2,2}.
\end{equation}
Combining \eqref{eq:A2-rec-b}  with
\eqref{eq:A2-h11h12},
\eqref{eq:A2-h12b10},
and
\eqref{eq:A2-h22b20}, we obtain, by induction on $r$,
\begin{align}
[\bh_{1,2},\bb_{1,r}]
=
3\bb_{1,r+2},
\label{eq:A2-h12b1r}
\qquad 
[\bh_{1,2},\bb_{2,r}]
=
-3\bb_{2,r+2}.
\end{align}

\subsubsection*{Step 2. Commutativity of $\bb_{i,r},\bb_{i,s}$}
We prove
\begin{equation}\label{eq:A2-same-sign}
[\bb_{1,r},\bb_{1,s}]
=
[\bb_{2,r},\bb_{2,s}]
=
0
\qquad
(r,s\in\bN).
\end{equation}
It suffices to consider the first family as the second one is similar.  We use induction on
$\ell=r+s$.  The cases $\ell=0,1$ are clear, while
\begin{equation*}
[\bb_{1,0},\bb_{1,2}]
=
[\bb_{1,0},[\bh_{1,1},\bb_{1,1}]]
=
[\bh_{1,1},[\bb_{1,0},\bb_{1,1}]]
=0,
\end{equation*}
so the case $\ell=2$ also holds.

Suppose that $[\bb_{1,r},\bb_{1,\ell-1-r}]=0$ for $0\lle r\lle \ell-1$. Applying $\ad(\bh_{1,1})$ gives
\begin{equation}\label{eq:A2-same-chain}
[\bb_{1,r+1},\bb_{1,\ell-1-r}]
+
[\bb_{1,r},\bb_{1,\ell-r}]
=
0.
\end{equation}
Set $\mathsf V_r:=
[\bb_{1,r},\bb_{1,\ell-r}]$, for $0\lle r\lle \ell$. Then \eqref{eq:A2-same-chain} gives
$\mathsf V_{r+1}=-\mathsf V_r$ and hence $
\mathsf V_r=(-1)^r\mathsf V_0$.

If $\ell$ is 
even, then we have $\mathsf V_\ell=(-1)^\ell\mathsf V_0=\mathsf V_0$. But skew-symmetry gives $\mathsf V_\ell=-\mathsf V_0$. Thus $\mathsf V_0=0$.

Suppose now that $\ell $ is odd.  Then $\ell\gge3$.  By the induction
hypothesis, $[\bb_{1,0},\bb_{1,\ell-2}]=0$. Applying $\ad(\bh_{1,2})$ and using
\eqref{eq:A2-h12b1r}, we obtain
\begin{align*}
0
=
3[\bb_{1,2},\bb_{1,\ell-2}]
+
3[\bb_{1,0},\bb_{1,\ell}]
=
3(\mathsf V_2+\mathsf V_0)
=
6\mathsf V_0.
\end{align*}
Hence again $\mathsf V_0=0$. This proves $[\bb_{1,r},\bb_{1,s}]=0$ for all $r,s\in\bN$.

\subsubsection*{Step 3. Boundary Cartan action} Set
\begin{equation*}
\mathfrak q:=[\bb_{1,0},\bb_{2,0}].
\end{equation*}
The finite type Serre relations give $[\bb_{1,0},\mathfrak q]=4\bb_{1,0}$ and 
$[\bb_{2,0},\mathfrak q]=-4\bb_{2,0}$.
By \eqref{eq:A2-helper-B}, we have $[\bb_{1,1},\mathfrak q]
=
2\bb_{1,1}$ and similarly $[\bb_{2,1},\mathfrak q]
=
-2\bb_{2,1}$
by interchanging the indices. Moreover, by \eqref{eq:A2-b10b20} and \eqref{eq:A2-h12b1r},
\begin{align*}
[\bh_{1,2},\mathfrak q]
&=
[[\bh_{1,2},\bb_{1,0}],\bb_{2,0}]
+
[\bb_{1,0},[\bh_{1,2},\bb_{2,0}]]
\\
&=
3[\bb_{1,2},\bb_{2,0}]
-
3[\bb_{1,0},\bb_{2,2}]
=
0.
\end{align*}
Therefore, using
\eqref{eq:A2-h12b1r}, we obtain by induction on $r$ that 
\begin{align}
[\bb_{1,r},\mathfrak q]
=
\bigl(3+(-1)^r\bigr)\bb_{1,r},
\label{eq:A2-q-b1}
\qquad 
[\bb_{2,r},\mathfrak q]
=
-\bigl(3+(-1)^r\bigr)\bb_{2,r}
\end{align}
for all $r\in\bN$. 

We now compute the action of an arbitrary Cartan mode on
$\bb_{1,0}$. From \eqref{eq:A2-mixed-low} and
$[\bh_{1,1},\mathfrak q]
=
[\bb_{1,1},\bb_{2,0}]
+
[\bb_{1,0},\bb_{2,1}]$,
we obtain
\begin{equation}\label{eq:A2-b10b21}
[\bb_{1,0},\bb_{2,1}]
=
\bh_{1,1}
+
\tfrac12[\bh_{1,1},\mathfrak q].
\end{equation}

For $r\gge1$, using \eqref{eq:A2-rec-h1},
\eqref{eq:A2-same-sign}, and
\eqref{eq:A2-b10b21}, we compute
\begin{align*}
2[\bh_{1,r},\bb_{1,0}]
={}&
[[\bb_{2,0},\bb_{1,r}],\bb_{1,0}]
-
[[\bb_{2,1},\bb_{1,r-1}],\bb_{1,0}]
\\
={}&
[\bb_{1,r},\mathfrak q]
-
[\bb_{1,r-1},[\bb_{1,0},\bb_{2,1}]]
\\
={}&
(3+(-1)^r)\bb_{1,r}
-
[\bb_{1,r-1},\bh_{1,1}]
-\tfrac12
[\bb_{1,r-1},[\bh_{1,1},\mathfrak q]]
\\
={}&
(3+(-1)^r)\bb_{1,r}
+\bb_{1,r}
-\tfrac12
\left(
[[\bb_{1,r-1},\bh_{1,1}],\mathfrak q]
+
[\bh_{1,1},[\bb_{1,r-1},\mathfrak q]]
\right)
\\
={}&
(3+(-1)^r)\bb_{1,r}
+\bb_{1,r}
-\tfrac12
\left(
-(3+(-1)^r)\bb_{1,r}
+
(3+(-1)^{r-1})\bb_{1,r}
\right)
\\
={}&
2\bigl(2+(-1)^r\bigr)\bb_{1,r}.
\end{align*}
Hence we have
\begin{equation}\label{eq:A2-hm-b10}
[\bh_{1,r},\bb_{1,0}]
=
\bigl(2+(-1)^r\bigr)\bb_{1,r}.
\end{equation}
Repeating the same calculation with the indices $1$ and $2$
interchanged, using \eqref{eq:A2-rec-h2} and
\eqref{eq:A2-q-b1}, gives
\begin{equation*}
[\bh_{2,r},\bb_{2,0}]
=
\bigl(2+(-1)^r\bigr)\bb_{2,r}.
\end{equation*}

\subsubsection*{Step 4. Higher shifted Serre relations and the full Cartan action}
For $r,s,t\in\bN$, set
\begin{equation*}
\Lambda(r,s\,|\,t)
:=
(-1)^r+(-1)^s+2(-1)^t.
\end{equation*}
We claim that
\begin{align}
[\bb_{1,r},[\bb_{1,s},\bb_{2,t}]]
&=
\Lambda(r,s\,|\,t)\bb_{1,r+s+t},
\label{eq:A2-higher-Serre1}
\\
[\bb_{2,r},[\bb_{2,s},\bb_{1,t}]]
&=
\Lambda(r,s\,|\,t)\bb_{2,r+s+t}.
\label{eq:A2-higher-Serre2}
\end{align}

We only prove \eqref{eq:A2-higher-Serre1} as the proof of \eqref{eq:A2-higher-Serre2} is very similar.  Define
\begin{equation}\label{eq:A2-P}
\mathsf P(r,s\,|\,t)
:=
[\bb_{1,r},[\bb_{1,s},\bb_{2,t}]]
-
\Lambda(r,s\,|\,t)\bb_{1,r+s+t}.
\end{equation}
By \eqref{eq:A2-same-sign},
\begin{equation}\label{eq:A2-P-sym}
\mathsf P(r,s\,|\,t)=\mathsf P(s,r\,|\,t).
\end{equation}

The relations in total degree $0,1,2$ follow from
\eqref{eq:A2-two-shifted-Serre}, \eqref{eq:A2-b10b20},
\eqref{eq:A2-Serre-10},
\eqref{eq:A2-helper-B}, \eqref{helper865}
and \eqref{eq:A2-b10b10b22}: 
\[
[\bb_{1,0},[\bb_{1,2},\bb_{2,0}]]=[\bb_{1,0},[\bb_{1,0},\bb_{2,2}]]=4\bb_{1,2}.
\]
Assume inductively that $\mathsf P(r,s\,|\,t)=0$ if $r+s+t<\ell$. Applying $\ad(\bh_{1,1})$ to a relation of total degree $\ell-1$ gives
\begin{equation}\label{eq:A2-P-shift1}
\mathsf P(r+1,s\,|\,t)
+
\mathsf P(r,s+1\,|\,t)
+
\mathsf P(r,s\,|\,t+1)
=
0,
\end{equation}
for $r+s+t=\ell-1$. Applying $\ad(\bh_{1,2})$ to a relation of total degree $\ell-2$ and
using
\eqref{eq:A2-h12b1r}, we obtain
\begin{equation}\label{eq:A2-P-shift2}
\mathsf P(r+2,s\,|\,t)
+
\mathsf P(r,s+2\,|\,t)
-
\mathsf P(r,s\,|\,t+2)
=
0,
\end{equation}
for $r+s+t=\ell-2$.

We first show that
\begin{equation}\label{eq:A2-P-interior}
\mathsf P(r,s\,|\,t)=0
\qquad
\text{if }r,s\gge1,\quad r+s+t=\ell.
\end{equation}
Indeed, applying \eqref{eq:A2-P-shift1} successively gives
\begin{align*}
\mathsf P(r,s\,|\,t+2)
={}&
-\mathsf P(r+1,s\,|\,t+1)
-\mathsf P(r,s+1\,|\,t+1)
\\
={}&
\mathsf P(r+2,s\,|\,t)
+
2\mathsf P(r+1,s+1\,|\,t)
+
\mathsf P(r,s+2\,|\,t).
\end{align*}
Comparing this with \eqref{eq:A2-P-shift2}, we obtain $2\mathsf P(r+1,s+1\,|\,t)=0$,
which proves \eqref{eq:A2-P-interior}.

It remains to consider the boundary terms.  If $s\gge1$, then
\eqref{eq:A2-P-shift1} and
\eqref{eq:A2-P-interior} imply
\begin{equation}\label{eq:A2-P-boundary-rec}
\mathsf P(0,s+1\,|\,t)
=
-\mathsf P(0,s\,|\,t+1).
\end{equation}

We next use \eqref{eq:A2-hm-b10}.  From
\eqref{eq:A2-rec-h1} and \eqref{eq:A2-same-sign}, we have
\begin{align*}
2[\bh_{1,\ell},\bb_{1,0}]
={}&
[\bb_{1,\ell},[\bb_{1,0},\bb_{2,0}]]
-
[\bb_{1,\ell-1},[\bb_{1,0},\bb_{2,1}]].
\end{align*}
Substituting \eqref{eq:A2-P} into this equality and comparing the
coefficients with \eqref{eq:A2-hm-b10}, we obtain
\begin{equation*}
\mathsf P(\ell,0\,|\,0)
=
\mathsf P(\ell-1,0\,|\,1).
\end{equation*}
Using \eqref{eq:A2-P-sym} and \eqref{eq:A2-P-boundary-rec}, this becomes
$\mathsf P(0,\ell-1\,|\,1)
=\mathsf P(0,\ell\,|\,0)=-\mathsf P(0,\ell-1\,|\,1)
$. Therefore, we obtain $\mathsf P(0,\ell\,|\,0)
=
\mathsf P(0,\ell-1\,|\,1)
=
0$.
Repeated use of \eqref{eq:A2-P-boundary-rec} then gives
$\mathsf P(0,s\,|\,t)=0$ for
$s\gge1$, $s+t=\ell$.
Finally, taking $r=s=0$ in
\eqref{eq:A2-P-shift1}, we obtain
\begin{equation*}
2\mathsf P(0,1\,|\,\ell-1)
+
\mathsf P(0,0\,|\,\ell)
=
0.
\end{equation*}
Thus $\mathsf P(0,0\,|\,\ell)=0$.  This completes the induction and proves
\eqref{eq:A2-higher-Serre1}.

Using \eqref{eq:A2-rec-h1} and
\eqref{eq:A2-same-sign}, we have, for $r,s\in\bN$ with $r\gge1$,
\begin{align*}
2[\bh_{1,r},\bb_{1,s}]
={}&
[\bb_{1,r},[\bb_{1,s},\bb_{2,0}]]
-
[\bb_{1,r-1},[\bb_{1,s},\bb_{2,1}]].
\end{align*}
Applying \eqref{eq:A2-higher-Serre1}, we obtain
\begin{equation}\label{eq:A2-full-hb11}
[\bh_{1,r},\bb_{1,s}]
=
\bigl(2+(-1)^r\bigr)\bb_{1,r+s}.
\end{equation}
Similarly, we have
\begin{equation}\label{eq:A2-full-hb12}
[\bh_{1,r},\bb_{2,s}]
=
-\bigl(1+2(-1)^r\bigr)\bb_{2,r+s}.
\end{equation}

\subsubsection*{Step 5. The remaining relations}

Set $\mathscr A(r,s)
:=
[\bb_{1,r+1},\bb_{2,s}]-[\bb_{1,r},\bb_{2,s+1}]$.
We prove
\begin{equation}\label{eq:A2-full-mixed}
\mathscr A(r,s)
=
-2(-1)^s\bh_{1,r+s+1}
\qquad
(r,s\in\bN).
\end{equation}

The case $r=s=0$ is
\eqref{eq:A2-mixed-low}.  For total degree one,
applying $\ad(\bh_{1,1})$ to
\eqref{eq:A2-mixed-low} gives
$\mathscr A(1,0)+\mathscr A(0,1)=0$. On the other hand, by \eqref{eq:A2-rec-h1}, $\mathscr A(1,0)=-2\bh_{1,2}$.
Hence $\mathscr A(0,1)=2\bh_{1,2}$,
so \eqref{eq:A2-full-mixed} holds in total degree one.

Assume now that \eqref{eq:A2-full-mixed} has been proved in all total
degrees strictly smaller than $\ell$. 
Applying $\ad(\bh_{1,1})$ to
\eqref{eq:A2-full-mixed} in total degree $\ell-1$, we have
\begin{equation*}
\mathscr A(r+1,s)+\mathscr A(r,s+1)
=
-2(-1)^s[\bh_{1,1},\bh_{1,\ell}],
\qquad
r+s=\ell-1.
\end{equation*}
By \eqref{eq:A2-rec-h1}, $\mathscr A(\ell,0)=-2\bh_{1,\ell+1}$. Thus we obtain recursively
\begin{equation}\label{eq:A2-A-general}
\mathscr A(\ell-s,s)
=
-2(-1)^s\bh_{1,\ell+1}
+
2s(-1)^s[\bh_{1,1},\bh_{1,\ell}].
\end{equation}

It remains to prove $[\bh_{1,1},\bh_{1,\ell}]=0$.  For $\ell=1$, this is trivial.
For $\ell=2$, it is \eqref{eq:A2-h11h12}.  Assume $\ell\gge3$. Apply $\ad(\bh_{1,2})$ to
\eqref{eq:A2-full-mixed} in total degree $\ell-2$.  Using
\eqref{eq:A2-h12b1r}, we obtain
\begin{equation}\label{eq:A2-A-rec2}
3\bigl(\mathscr A(\ell-s,s)-\mathscr A(\ell-s-2,s+2)\bigr)
=
-2(-1)^s[\bh_{1,2},\bh_{1,\ell-1}].
\end{equation}
Substituting \eqref{eq:A2-A-general} into \eqref{eq:A2-A-rec2}, we find
\begin{equation}\label{eq:A2-h12hm-first}
[\bh_{1,2},\bh_{1,\ell-1}]
=
6[\bh_{1,1},\bh_{1,\ell}].
\end{equation}

By \eqref{eq:A2-rec-h1}, \eqref{eq:A2-h12}, and \eqref{eq:A2-full-hb11}--\eqref{eq:A2-full-hb12}, we have
\begin{align*}
2[\bh_{1,\ell-1},\bh_{1,2}]
={}&
-(1-2(-1)^\ell)[\bb_{2,\ell-1},\bb_{1,2}]
+
(2-(-1)^\ell)[\bb_{2,0},\bb_{1,\ell+1}]
\\
&
+(1-2(-1)^\ell)[\bb_{2,\ell},\bb_{1,1}]
-
(2-(-1)^\ell)[\bb_{2,1},\bb_{1,\ell}]
\\
={}&
2(2-(-1)^\ell)\bh_{1,\ell+1}
+(1-2(-1)^\ell)\mathscr A(1,\ell-1).
\end{align*}
By \eqref{eq:A2-A-general}, we find
\begin{equation}\label{eq:A2-h12hm-second}
[\bh_{1,2},\bh_{1,\ell-1}]
=
-(\ell-1)(2-(-1)^\ell)[\bh_{1,1},\bh_{1,\ell}].
\end{equation}
Comparing
\eqref{eq:A2-h12hm-first} and
\eqref{eq:A2-h12hm-second}, we obtain
\begin{equation*}
\bigl(6+(\ell-1)(2-(-1)^\ell)\bigr)[\bh_{1,1},\bh_{1,\ell}]=0\Longrightarrow [\bh_{1,1},\bh_{1,\ell}]=0,
\end{equation*}
proving \eqref{eq:A2-full-mixed} in total degree $\ell$ and
completing the induction.

By \eqref{eq:A2-rec-h2}, $2\bh_{2,s+1}
=
-\mathscr A(0,s)$. Using \eqref{eq:A2-full-mixed}, we obtain $2\bh_{2,s+1}
=
2(-1)^s\bh_{1,s+1}$. Thus together with \eqref{tch-sym} we have
\begin{equation}\label{eq:A2-full-hsym}
\bh_{2,r}
=
(-1)^{r+1}\bh_{1,r}
\qquad
(r\in\bN).
\end{equation}

It remains to prove that all Cartan modes commute.  First, by \eqref{hi1bjr} we have $[\bh_{1,0},\bb_{1,r}]=3\bb_{1,r}$ and $[\bh_{1,0},\bb_{2,s}]=-3\bb_{2,s}$. Hence by \eqref{eq:A2-rec-h1} we have $[\bh_{1,0},\bh_{1,r}]=0$.

Let $r,s\gge1$.  Using \eqref{eq:A2-rec-h1} and
\eqref{eq:A2-full-hb11}--\eqref{eq:A2-full-hb12}, we obtain
\begin{align*}
2[\bh_{1,r},\bh_{1,s}]
={}&
[\bh_{1,r},
[\bb_{2,0},\bb_{1,s}]
-
[\bb_{2,1},\bb_{1,s-1}]
]
\\
={}&
2(2+(-1)^r)\bh_{1,r+s}
+(1+2(-1)^r)\mathscr A(s-1,r).
\end{align*}
By \eqref{eq:A2-full-mixed}, $\mathscr A(s-1,r)
=
-2(-1)^r\bh_{1,r+s}$. Hence $2[\bh_{1,r},\bh_{1,s}]=0$.
Therefore $[\bh_{1,r},\bh_{1,s}]
=
0$
for $r,s\in\bN$.
Together with \eqref{eq:A2-full-hsym}, this proves
\begin{equation*}
[\bh_{i,r},\bh_{j,s}]=0
\qquad
(i,j\in\{1,2\},\ r,s\in\bN).
\end{equation*}

Finally, \eqref{eq:A2-full-hsym} and
\eqref{eq:A2-full-hb11}--\eqref{eq:A2-full-hb12} give the full
Cartan action for $i=2$ as well.  Moreover,
\eqref{eq:A2-same-sign} gives \eqref{eq:qsclassical3} when
$i=j$, while \eqref{eq:A2-full-mixed} gives
\eqref{eq:qsclassical3} for $(i,j)=(1,2)$.  Indeed,
$-2(-1)^s\bh_{1,r+s+1}
=
-2(-1)^r\bh_{2,r+s+1}$
by \eqref{eq:A2-full-hsym}.  The case $(i,j)=(2,1)$ follows by
skew-symmetry.

We have therefore recovered all the higher relations
\eqref{eq:qsclassical1}--\eqref{eq:qsclassical3} in type
$\mathsf A_2$ from the degree-zero and degree-one relations together
with \eqref{eq:A2-extra-Serre} and
\eqref{eq:A2-extra-Cartan}.

\subsection{Extracting the two central rank-two relations}
In this subsection, we prove that the extra relations are not needed if $n\gge 2$.  

Assume from now on that $n\gge2$, and set $i:=n-1, j:=n$.
Thus
$
\tau i=n+2,
\tau j=n+1,
$
and
\begin{equation*}
c_{ij}=c_{\tau i,\tau j}=-1,
\qquad
c_{\tau i,j}=c_{i,\tau j}=0,
\qquad
c_{i,\tau i}=0,
\qquad
c_{j,\tau j}=-1.
\end{equation*}

The orbit $\{i,\tau i\}$ is noncentral; in particular,
$c_{i,\tau i}=0$.  It supplies the two
derivations
\begin{equation*}
\cD_1:=\ad(\bh_{i,1}),
\qquad
\cD_2:=\ad(\bh_{i,2}).
\end{equation*}
Therefore, by \eqref{hi1bjr} and Lemma~\ref{lem:h2bj}, respectively,
\begin{align}
\cD_1(\bb_{j,r})
&=-\bb_{j,r+1},
&
\cD_1(\bb_{\tau j,r})
&=-\bb_{\tau j,r+1},
\label{eq:A2n-D1-action}
\\
\cD_2(\bb_{j,r})
&=-\bb_{j,r+2},
&
\cD_2(\bb_{\tau j,r})
&=\bb_{\tau j,r+2}.
\label{eq:A2n-D2-action}
\end{align}
The point of the next lemma is that $\cD_2$ also annihilates the
degree-one Cartan generator on the central orbit.

\begin{lem}\label{lem:A2n-neighboring-orbit}
With the notation above, we have
\begin{equation}\label{eq:A2n-outer-Cartan}
[\bh_{i,1},\bh_{i,2}]=0, \qquad \cD_2(\bh_{j,1})=[\bh_{i,2},\bh_{j,1}]
=0.
\end{equation}
\end{lem}

\begin{proof}
We first derive the higher-mode consequences of the finite Serre relation between $i$ and $j$.
Since $i$ and $j$ are in distinct $\tau$-orbits,
Lemma~\ref{lem:bb} gives
\begin{equation}\label{eq:A2n-cross-shift}
[\bb_{i,r+1},\bb_{j,s}]
=
[\bb_{i,r},\bb_{j,s+1}].
\end{equation}
Starting from $[\bb_{i,0},[\bb_{i,0},\bb_{j,0}]]=0$, we claim that
\begin{equation}\label{eq:A2n-Serre-prop}
[\bb_{i,0},[\bb_{i,0},\bb_{j,s}]]=0
\qquad
(s\in\bN).
\end{equation}
Indeed, if the assertion holds for $s$, then applying
$\ad(\bh_{j,1})$ gives
\begin{align*}
0=
-[\bb_{i,1},[\bb_{i,0},\bb_{j,s}]]
-[\bb_{i,0},[\bb_{i,1},\bb_{j,s}]]+
[\bb_{i,0},[\bb_{i,0},\bb_{j,s+1}]].
\end{align*}
Since $[\bb_{i,0},\bb_{i,1}]=0$, the first two terms are equal, and
\eqref{eq:A2n-cross-shift} identifies each of them with
$[\bb_{i,0},[\bb_{i,0},\bb_{j,s+1}]]$.
Hence \eqref{eq:A2n-Serre-prop} follows by induction.  In particular,
using \eqref{eq:A2n-cross-shift} twice,
\begin{align}
[\bb_{i,1},[\bb_{i,1},\bb_{j,s}]]
&=
[\bb_{i,1},[\bb_{i,0},\bb_{j,s+1}]]
\notag\\
&=
[\bb_{i,0},[\bb_{i,1},\bb_{j,s+1}]]=
[\bb_{i,0},[\bb_{i,0},\bb_{j,s+2}]]
=0.
\label{eq:A2n-shifted-neighbor-Serre}
\end{align}

We also need the nonadjacent pair $i,\tau j$.  Since
$c_{i,\tau j}=0$, $[\bb_{i,0},\bb_{\tau j,0}]=0$. Assume $[\bb_{i,0},\bb_{\tau j,s}]=0$.  Applying
$\ad(\bh_{i,1})$ gives
\begin{equation*}
2[\bb_{i,1},\bb_{\tau j,s}]
-
[\bb_{i,0},\bb_{\tau j,s+1}]
=0,
\end{equation*}
while Lemma~\ref{lem:bb} gives $[\bb_{i,1},\bb_{\tau j,s}]
=
[\bb_{i,0},\bb_{\tau j,s+1}]$. Therefore, $[\bb_{i,1},\bb_{\tau j,s}]
=
[\bb_{i,0},\bb_{\tau j,s+1}]=0$. By induction we have
\begin{equation}\label{eq:A2n-nonadjacent}
[\bb_{i,0},\bb_{\tau j,s}]
=
[\bb_{i,1},\bb_{\tau j,s}]
=0
\qquad
(s\in\bN).
\end{equation}

Take $s=0$ in \eqref{eq:A2n-shifted-neighbor-Serre} and bracket
with $\bb_{\tau j,1}$.  By
\eqref{eq:A2n-nonadjacent},
\begin{equation*}
0
=
[\bb_{i,1},
[\bb_{i,1},
[\bb_{j,0},\bb_{\tau j,1}]]].
\end{equation*}
Plugging 
$
[\bb_{j,0},\bb_{\tau j,1}]
=
[\bb_{j,1},\bb_{\tau j,0}]
+
2\bh_{j,1}$ from \eqref{tcbb} into it,
the first term on the right gives zero by
\eqref{eq:A2n-nonadjacent} and
\eqref{eq:A2n-shifted-neighbor-Serre} with $s=1$.  Hence
\begin{equation}\label{eq:A2n-bi-bi-hj}
[\bb_{i,1},[\bb_{i,1},\bh_{j,1}]]=0.
\end{equation}

For the noncentral orbit $\{i,\tau i\}$, we have
\begin{equation}\label{eq:A2n-hi2-expression}
\bh_{i,2}=[\bb_{i,1},\bb_{\tau i,1}].
\end{equation}
Applying $\ad(\bb_{\tau i,0})$ to
\eqref{eq:A2n-bi-bi-hj} and using $[\bb_{\tau i,0},\bb_{i,1}]=\bh_{i,1}$, $
[\bb_{\tau i,0},\bh_{j,1}]=\bb_{\tau i,1}$, and $[\bh_{i,1},\bh_{j,1}]=0$,
we obtain
\begin{equation*}
[\bh_{i,1},[\bb_{i,1},\bh_{j,1}]]
+
[\bb_{i,1},\bh_{i,2}]
=0.
\end{equation*}
Apply $\ad(\bb_{\tau i,0})$ once more.  Since
\begin{align*}
[\bb_{\tau i,0},\bh_{i,1}]
&=-2\bb_{\tau i,1},
&
[\bb_{\tau i,0},\bh_{i,2}]
&=2\bb_{\tau i,2},
\\
[\bb_{\tau i,0},
[\bb_{i,1},\bh_{j,1}]]
&=\bh_{i,2},
&
[\bb_{i,1},\bh_{j,1}]
&=\bb_{i,2},
\end{align*}
we find
\begin{equation}\label{eq:A2n-neighbor-second}
0=
2[\bb_{i,2},\bb_{\tau i,1}]
+
2[\bh_{i,1},\bh_{i,2}]
+
2[\bb_{i,1},\bb_{\tau i,2}].
\end{equation}
On the other hand, applying $\ad(\bh_{i,1})$ to
\eqref{eq:A2n-hi2-expression} gives
\begin{equation}\label{eq:A2n-neighbor-third}
[\bh_{i,1},\bh_{i,2}]
=
2[\bb_{i,2},\bb_{\tau i,1}]
+
2[\bb_{i,1},\bb_{\tau i,2}].
\end{equation}
Comparing \eqref{eq:A2n-neighbor-second} and
\eqref{eq:A2n-neighbor-third}, we obtain $3[\bh_{i,1},\bh_{i,2}]=0$,
which proves this first relation in \eqref{eq:A2n-outer-Cartan}.

Finally,
\begin{align*}
[\bh_{j,1},\bh_{i,2}]
&=
[\bh_{j,1},
[\bb_{i,1},\bb_{\tau i,1}]]
\\
&=
-[\bb_{i,2},\bb_{\tau i,1}]
-
[\bb_{i,1},\bb_{\tau i,2}]
=
-\tfrac12[\bh_{i,1},\bh_{i,2}]
=0,
\end{align*}
where the last equality follows from
\eqref{eq:A2n-neighbor-third}.  This proves the second relation in 
\eqref{eq:A2n-outer-Cartan}.
\end{proof}

The previous lemma allows us to use $\cD_1^2+\cD_2$ on the two
degree-zero/degree-one relations at the central orbit.  This gives a
short uniform derivation of both additional relations appearing in
the isolated $\mathsf A_2$ case.

\begin{lem}\label{lem:A2n-central-seeds}
Assume $n\gge2$.  Then the central pair $\{j,\tau j\}$ satisfies
\begin{align}
&
[\bb_{j,1},[\bb_{j,1},\bb_{\tau j,0}]]
+
[\bb_{\tau j,1},[\bb_{\tau j,1},\bb_{j,0}]]
=0,
\label{eq:A2n-central-seed-Serre}
\\
&
\big[[\bh_{j,1},\bb_{j,1}],\bb_{\tau j,1}\big]
-
\big[\bb_{j,0},
[\bh_{j,1},[\bh_{j,1},\bb_{\tau j,1}]]
\big]
=0.
\label{eq:A2n-central-seed-Cartan}
\end{align}
In fact, the two summands in
\eqref{eq:A2n-central-seed-Serre} vanish separately.
\end{lem}

\begin{proof}
Introduce
\begin{align*}
\mathcal R_j
&:=
[\bb_{j,0},[\bb_{j,0},\bb_{\tau j,0}]]
-4\bb_{j,0},
\\
\mathcal M_j
&:=
[\bb_{j,1},\bb_{\tau j,0}]
-
[\bb_{j,0},\bb_{\tau j,1}]
+
2\bh_{j,1}.
\end{align*}
Thus $\mathcal R_j=0$ and $\mathcal M_j=0$.

We first need one elementary consequence of these two relations. 
Applying $\ad(\bh_{j,1})$ to $\mathcal R_j=0$, and using
$[\bb_{j,0},\bb_{j,1}]=0$, gives
\begin{equation*}
2[\bb_{j,0},[\bb_{j,1},\bb_{\tau j,0}]]+[\bb_{j,0},[\bb_{j,0},\bb_{\tau j,1}]]
=
4\bb_{j,1}.
\end{equation*}
Applying $\ad(\bb_{j,0})$ to $\mathcal M_j=0$ gives
\begin{equation*}
[\bb_{j,0},[\bb_{j,1},\bb_{\tau j,0}]]-[\bb_{j,0},[\bb_{j,0},\bb_{\tau j,1}]]
=
2\bb_{j,1}.
\end{equation*}
Hence
\begin{equation}\label{9999}
[\bb_{j,0},[\bb_{j,1},\bb_{\tau j,0}]]=2\bb_{j,1},
\qquad
[\bb_{j,0},[\bb_{j,0},\bb_{\tau j,1}]]=0.
\end{equation}

We now apply $\cD_1^2+\cD_2$ to $\mathcal R_j=0$.  Set
\begin{equation*}
\mathcal S_j
:=
[\bb_{j,1},[\bb_{j,1},\bb_{\tau j,0}]],\quad 
\mathcal U_j
:=
[\bb_{j,0},[\bb_{j,1},\bb_{\tau j,1}]],
\quad
\mathcal V_j
:=
[\bb_{j,0},[\bb_{j,0},\bb_{\tau j,2}]].
\end{equation*}
Using \eqref{eq:A2n-D1-action}--\eqref{eq:A2n-D2-action}, a direct
expansion gives
\begin{align}
\cD_1^2(\mathcal R_j)
={}&
[\bb_{j,2},[\bb_{j,0},\bb_{\tau j,0}]]
+
2\mathcal S_j
+
2[\bb_{j,1},[\bb_{j,0},\bb_{\tau j,1}]]
\notag\\
&
+
[\bb_{j,0},[\bb_{j,2},\bb_{\tau j,0}]]
+
2\mathcal U_j
+
\mathcal V_j
-
4\bb_{j,2},
\label{eq:A2n-D1sq-R}
\\
\cD_2(\mathcal R_j)
={}&
-[\bb_{j,2},[\bb_{j,0},\bb_{\tau j,0}]]
-
[\bb_{j,0},[\bb_{j,2},\bb_{\tau j,0}]]
+
\mathcal V_j
+
4\bb_{j,2}.
\label{eq:A2n-D2-R}
\end{align}
Since $[\bb_{j,0},\bb_{j,1}]=0$, we have $[\bb_{j,1},[\bb_{j,0},\bb_{\tau j,1}]]
=
\mathcal U_j$. Adding \eqref{eq:A2n-D1sq-R} and
\eqref{eq:A2n-D2-R}, we therefore obtain
\begin{equation}\label{eq:A2n-Dsum-R}
0
=
(\cD_1^2+\cD_2)(\mathcal R_j)
=
2\mathcal S_j+4\mathcal U_j+2\mathcal V_j.
\end{equation}

On the other hand, applying $\ad(\bh_{j,1})$ to
$[\bb_{j,0},[\bb_{j,0},\bb_{\tau j,1}]]=0$ from \eqref{9999} gives
\begin{equation}\label{eq:A2n-hT}
0
=
\Big[\bh_{j,1},\big[\bb_{j,0},[\bb_{j,0},\bb_{\tau j,1}]\big]\Big]
=
2\mathcal U_j+\mathcal V_j.
\end{equation}
Subtracting twice \eqref{eq:A2n-hT} from
\eqref{eq:A2n-Dsum-R}, we obtain $\mathcal S_j=0$. Applying $\tau$ gives $[\bb_{\tau j,1},[\bb_{\tau j,1},\bb_{j,0}]]=0$. This proves \eqref{eq:A2n-central-seed-Serre}.

We next apply the same operator $\cD_1^2+\cD_2$ to
$\mathcal M_j=0$.  By the degree-one Cartan relation and
Lemma~\ref{lem:A2n-neighboring-orbit}, $\cD_1(\bh_{j,1})=\cD_2(\bh_{j,1})=0$. Hence
\begin{align*}
\cD_1^2(\mathcal M_j)
={}&
[\bb_{j,3},\bb_{\tau j,0}]
+
[\bb_{j,2},\bb_{\tau j,1}]
-
[\bb_{j,1},\bb_{\tau j,2}]
-
[\bb_{j,0},\bb_{\tau j,3}],
\\
\cD_2(\mathcal M_j)
={}&
-[\bb_{j,3},\bb_{\tau j,0}]
+
[\bb_{j,1},\bb_{\tau j,2}]
+
[\bb_{j,2},\bb_{\tau j,1}]
-
[\bb_{j,0},\bb_{\tau j,3}].
\end{align*}
Adding these two equations gives
\begin{equation}\label{eq:A2n-C-seed-short}
[\bb_{j,2},\bb_{\tau j,1}]
=
[\bb_{j,0},\bb_{\tau j,3}].
\end{equation}
Since $\gamma_{jj}=1$, we have $\bb_{j,2}
=
[\bh_{j,1},\bb_{j,1}]$, $\bb_{\tau j,3}
=
[\bh_{j,1},
[\bh_{j,1},\bb_{\tau j,1}]]$,
and \eqref{eq:A2n-C-seed-short} is exactly
\eqref{eq:A2n-central-seed-Cartan}.
\end{proof}

\begin{rem}\label{rem:A2n-D-idea}
The role of the neighboring orbit can be summarized by the two
degree shifts
\begin{align*}
&\cD_1:
(\bb_{j,r},\bb_{\tau j,r})
\longmapsto
(-\bb_{j,r+1},-\bb_{\tau j,r+1}),\\
&\cD_2:
(\bb_{j,r},\bb_{\tau j,r})
\longmapsto
(-\bb_{j,r+2},\bb_{\tau j,r+2}).
\end{align*}
The combination $\cD_1^2+\cD_2$ cancels the unwanted outer terms and
isolates precisely the two rank-two obstructions.  This operator is
not available in the isolated type $\mathsf A_2$, which explains
why the two additional classical relations have to be imposed there.
\end{rem}

\subsection{Completion of the higher-rank proof}

We can now finish the proof without repeating the rank-two
calculation.  By Lemma~\ref{lem:A2n-central-seeds}, the two
extra relations hold on the central
orbit.  Hence the calculations in Appendix \ref{ssec:A2-reduction} apply verbatim after the
relabeling
\begin{equation*}
1\longmapsto j=n,
\qquad
2\longmapsto\tau j=n+1.
\end{equation*}
Thus all the relations in Lemma~\ref{lem:full-gr} supported on the
central orbit are recovered.  In particular,
\begin{align}
&[\bh_{j,r},\bh_{j,s}]
=0,
\qquad \qquad
\bh_{\tau j,r}
=(-1)^{r+1}\bh_{j,r},
\notag
\\
&[\bh_{j,r},\bb_{j,s}]
=
\bigl(2+(-1)^r\bigr)\bb_{j,r+s},
\notag
\\
&[\bh_{j,r},\bb_{\tau j,s}]
=
-\bigl(1+2(-1)^r\bigr)\bb_{\tau j,r+s},
\notag
\\
&[\bb_{j,r},\bb_{j,s}]
=
[\bb_{\tau j,r},\bb_{\tau j,s}]
=0,
\notag
\\
&[\bb_{j,r+1},\bb_{\tau j,s}]
-
[\bb_{j,r},\bb_{\tau j,s+1}]
=
-2(-1)^s\bh_{j,r+s+1}.
\label{eq:A2n-central-mixed-full}
\end{align}

For the neighboring noncentral orbit $\{i,\tau i\}$,
Lemma~\ref{lem:A2n-neighboring-orbit} gives precisely the seed
$[\bh_{i,1},\bh_{i,2}]=0$ required in Step~4 of the proof of
Proposition~\ref{prop:red} in \S\ref{sec:pf-prop1}.  Hence all same-orbit relations on
$\{i,\tau i\}$ follow from that argument.  Moving successively away
from the center, every remaining noncentral orbit is treated by the
same proof as in Proposition~\ref{prop:red} performed in \S\ref{sec:pf-prop1}; no central rank-two
calculation enters again.

\begin{lem}\label{lem:A2n-central-cross-hb}
Let $j$ belong to the central $\tau$-orbit and let
$k\notin\{j,\tau j\}$.  Then, for all $r,s\gge 0$,
\begin{equation}\label{eq:A2n-central-cross-hb}
 [\bh_{j,r},\bb_{k,s}]
 =
 \bigl(c_{jk}-(-1)^r c_{\tau j,k}\bigr)
 \bb_{k,r+s}.
\end{equation}
\end{lem}

\begin{proof}
The case $r=0$ follows from the first identity in
\eqref{hi1bjr}.  We may therefore assume that $r\gge 1$.

Suppose first that $c_{jk}=-1$ and $c_{\tau j,k}=0$. By \eqref{eq:A2n-nonadjacent} and Lemma~\ref{lem:bb}, we have $[\bb_{\tau j,p},\bb_{k,q}]=0$ for all $p,q\gge 0$. Set $\ell=r+s$.  Using \eqref{eq:A2n-rec-h}, the
$\bb$--$\bb$ shift relation, and the Jacobi identity, we obtain
\begin{align}
 2[\bh_{j,r},\bb_{k,s}]
 &=
 \bigl[[\bb_{\tau j,0},\bb_{j,r}],\bb_{k,s}\bigr]
 -
 \bigl[[\bb_{\tau j,1},\bb_{j,r-1}],\bb_{k,s}\bigr]
 \notag\\
 &=
 \bigl[\bb_{\tau j,0},
       [\bb_{j,0},\bb_{k,\ell}]\bigr]
 -
 \bigl[\bb_{\tau j,1},
       [\bb_{j,0},\bb_{k,\ell-1}]\bigr]
 \notag\\
 &=
 \bigl[[\bb_{\tau j,0},\bb_{j,0}],\bb_{k,\ell}\bigr]
 -
 \bigl[[\bb_{\tau j,1},\bb_{j,0}],
       \bb_{k,\ell-1}\bigr].
 \label{eq:A2n-central-cross-hb-calc}
\end{align}
The degree-one central mixed relation gives
$[\bb_{\tau j,1},\bb_{j,0}]
 =
 [\bb_{\tau j,0},\bb_{j,1}]
 -2\bh_{j,1}$.
Moreover, the already established degree-one Cartan action gives $ [\bh_{j,1},\bb_{k,\ell-1}]
 =
 (c_{jk}+c_{\tau j,k})\bb_{k,\ell}
 =
 -\bb_{k,\ell}$.
Consequently,
\begin{align*}
 \bigl[[\bb_{\tau j,1},\bb_{j,0}],
       \bb_{k,\ell-1}\bigr]
 &=
 \bigl[\bb_{\tau j,0},
       [\bb_{j,1},\bb_{k,\ell-1}]\bigr]
 -2[\bh_{j,1},\bb_{k,\ell-1}]
 \\
 &=
 \bigl[\bb_{\tau j,0},
       [\bb_{j,0},\bb_{k,\ell}]\bigr]
 +2\bb_{k,\ell}
 \\
 &=
 \bigl[[\bb_{\tau j,0},\bb_{j,0}],
       \bb_{k,\ell}\bigr]
 +2\bb_{k,\ell}.
\end{align*}
Substituting this into
\eqref{eq:A2n-central-cross-hb-calc} yields $[\bh_{j,r},\bb_{k,s}]
 =
 -\bb_{k,r+s}$,
which is \eqref{eq:A2n-central-cross-hb} in this case.

Finally, if instead $c_{jk}=0$ and $c_{\tau j,k}=-1$,
the preceding argument with $j$ and $\tau j$ interchanged gives the desired equality. If $c_{jk}=c_{\tau j,k}=0$, the same argument, using
the nonadjacent $\bb$--$\bb$ relations, gives
$[\bh_{j,r},\bb_{k,s}]=0$.

This completes the proof.
\end{proof}

Together with the calculation for a noncentral first Cartan
index, Lemma~\ref{lem:A2n-central-cross-hb} proves the full
mixed $\bh$--$\bb$ relation for vertices in distinct
$\tau$-orbits.  It remains only to prove Cartan commutativity
when one of the two orbits is central.  We now explain the
required cancellation explicitly.

Let $k$ be in a $\tau$-orbit distinct from
$\{j,\tau j\}$ and set $\gamma_{kj}^{(r)}
:=
c_{kj}-(-1)^r c_{\tau k,j}$.
Using \eqref{eq:A2n-rec-h} and the mixed Cartan action already
obtained, we have
\begin{align}
2[\bh_{k,r},\bh_{j,s+1}]
={}&
-(-1)^r \gamma_{kj}^{(r)}
[\bb_{\tau j,r},\bb_{j,s+1}]
+
\gamma_{kj}^{(r)}
[\bb_{\tau j,0},\bb_{j,r+s+1}]
\notag\\
&+
(-1)^r \gamma_{kj}^{(r)}
[\bb_{\tau j,r+1},\bb_{j,s}]
-
\gamma_{kj}^{(r)}
[\bb_{\tau j,1},\bb_{j,r+s}].
\label{eq:A2n-cross-hh-short}
\end{align}
Rewriting
\eqref{eq:A2n-cross-hh-short} using \eqref{eq:A2n-central-mixed-full}, we find
\begin{align*}
2[\bh_{k,r},\bh_{j,s+1}]
=
(-1)^r \gamma_{kj}^{(r)}
\bigl(-2(-1)^r\bh_{j,r+s+1}\bigr)
-
\gamma_{kj}^{(r)}
\bigl(-2\bh_{j,r+s+1}\bigr)
=0.
\end{align*}
Thus the Cartan generators commute also across the central orbit.
All the relations in Lemma~\ref{lem:full-gr} have now been
recovered. Consequently, the recursively defined elements satisfy all the
relations of Lemma~\ref{lem:full-gr}. They therefore define a
homomorphism inverse to \(\psi\). Hence \(\psi\) is an isomorphism,
which completes the proof of Proposition~\ref{prop:red}, Part~II.

\section{R-matrix presentation}\label{sec:app+}

In this section, using the explicit identification between iYangians of quasi-split type $\mathsf A$ and twisted Yangians in R-matrix presentation established in \cite{LZ24}, we express $h_{i,r},b_{i,r}$ for $r=0,1$ in terms of the Drinfeld generators of the Yangian.

\subsection{Yangians}\label{sec:Y-R}
We first recall Yangians and twisted Yangians in R-matrix presentation (we shall only use relations in generating series form) from \cite{MNO96,MR02,LZ24}.

The \textit{Yangian} $\Y_{\mathscr R}$ corresponding to the Lie algebra $\gl_N$ is a unital associative algebra with generators $t_{ij}^{(r)}$, where $1\lle i,j\lle N$ and $r\in\bZ_{>0}$, subject to the defining relations 
\be
(u-v)[t_{ij}(u),t_{kl}(v)]=t_{kj}(u)t_{il}(v)-t_{kj}(v)t_{il}(u).
\ee
Here we have used the generating series in an indeterminate $u$
\[
t_{ij}(u)=  \delta_{ij}+t_{ij}^{(1)}u^{-1}+t_{ij}^{(2)}u^{-2}+\cdots.
\]
Introduce new generating series via Gauss decomposition, for $i<j$, 
\beq\label{tij-Gauss}
\begin{split}
t_{ii}(u)&=D_i(u)+\sum_{k<i}F_{ik}(u)D_k(u)E_{ki}(u), \\
t_{ij}(u)&=D_i(u)E_{ij}(u)+\sum_{k<i}F_{ik}(u)D_k(u)E_{kj}(u),\\
t_{ji}(u)&=F_{ji}(u)D_i(u)+\sum_{k<i}F_{jk}(u)D_k(u)E_{ki}(u).
\end{split}
\eeq
It has been established in \cite{BK05} that the Yangian $\Y(\mathfrak{sl}_N)$ in Drinfeld presentation can be identified as a subalgebra of $\Y_{\mathscr R}$ via the following correspondence:
\beq\label{eq:Y2-R}
\begin{split}
x_i^+(u)&=\sqrt{-1}F_{i+1,i}(u+\vartheta_i),\\
x_i^-(u)&=-\sqrt{-1}E_{i,i+1}(u+\vartheta_i),\\
\xi_i(u)&= D_{i}(u+\vartheta_i)^{-1}D_{i+1}(u+\vartheta_i),
\end{split}
\eeq
where $\vartheta_i=\tfrac{N-2i}{4}$. Note that we have $\vartheta_i+\vartheta_{\tau i}=0$.

Define $\wtl T(u):=\big(T(u)\big)^{-1}=\big(\tl t_{ij}(u)\big)$ and let $\tl t_{ij}(u)$ for $1\lle i,j\lle N$ be its matrix elements, 
$$
\tl t_{ij}(u)=\delta_{ij}+\sum_{r>0} \tl {t}_{ij}^{(r)}u^{-r}.
$$
Then
\be
\tl t_{ij}(u)=\delta_{ij}+\sum_{k>0} (-1)^k\sum_{a_1,\cdots,a_{k-1}=1}^N t_{ia_1}^\circ(u)t_{a_1a_2}^\circ(u)\cdots t_{a_{k-1}j}^\circ(u),
\ee
where $t_{ij}^\circ(u)=t_{ij}(u)-\delta_{ij}$. In particular, by taking the coefficient of $u^{-r}$, for $r\gge 1$, one obtains 
\be
\tl t_{ij}^{(r)}=\sum_{k=1}^r (-1)^k\sum_{a_1,\cdots,a_{k-1}=1}^N\sum_{r_1+\cdots+r_k=r}t_{ia_1}^{(r_1)}t_{a_1a_2}^{(r_2)}\cdots t_{a_{k-1}j}^{(r_k)},
\ee
where $r_i$ for $1\lle i\lle k$ are positive integers. In particular, we have
\begin{align*}
\tl t_{ij}^{(1)}= -t_{ij}^{(1)}, \qquad \tl t_{ij}^{(2)}= -t_{ij}^{(2)}+\sum_{k=1}^{N}t_{ik}^{(1)}t_{kj}^{(1)}.
\end{align*}

Now we express the Drinfeld generators in terms of $t_{ij}^{(r)}$. By \eqref{tij-Gauss}, we clearly have $t_{ii}^{(1)}=D_i^{(1)}$ for $1\lle i\lle N$ and $t_{ji}^{(1)}=F_{ji}^{(1)}$, $t_{ij}^{(1)}=E_{ij}^{(1)}$ for $i<j$. In particular $t_{i+1,i}^{(1)}=F_i^{(1)}$. Furthermore, we have
\begin{align}
&t_{i,i}^{(2)}=D_i^{(2)}+\sum_{k<i}F_{ik}^{(1)}E_{ki}^{(1)},\label{eq:t-diag-used}\\
&t_{i+1,i}^{(2)}=F_{i+1,i}^{(2)}+F_{i+1,i}^{(1)}D_i^{(1)}+\sum_{k<i}F_{i+1,k}^{(1)}E_{ki}^{(1)},\notag\\
&t_{i,i+1}^{(2)}=E_{i,i+1}^{(2)}+D_i^{(1)}E_{i,i+1}^{(1)}+\sum_{k<i}F_{i,k}^{(1)}E_{k,i+1}^{(1)}.\notag
\end{align}
On the other hand, it follows from \eqref{eq:Y2-R} that
\begin{align}
&x_{i,0}^+=\sqrt{-1}F_{i+1,i}^{(1)}, \qquad x_{i,0}^-=-\sqrt{-1}E_{i,i+1}^{(1)},\qquad \xi_{i,0}=D_{i+1}^{(1)}-D_i^{(1)},\label{eq:xpm0}\\
&x_{i,1}^+=\sqrt{-1}\big(F_{i+1,i}^{(2)}-\vartheta_i F_{i+1,i}^{(1)}\big), \quad  x_{i,1}^-=-\sqrt{-1}\big(E_{i,i+1}^{(2)}-\vartheta_i E_{i,i+1}^{(1)}\big),\notag\\
&\xi_{i,1}=D_{i+1}^{(2)}-D_i^{(2)}+\big(D_{i}^{(1)}\big)^2-\vartheta_i D_{i+1}^{(1)}+\vartheta_i D_{i}^{(1)}-D_{i}^{(1)}D_{i+1}^{(1)}.\label{eq:xi1}
\end{align}

Let $\mathbf e_{ij}$ be the standard generators of $\gl_N$. Then with the identification above, we have
\[
t_{ij}^{(1)}=-(\sqrt{-1})^{i-j}\mathbf e_{ji}.
\]

\subsection{Twisted Yangians}
For $1\lle i,j\lle N$, we introduce the generating series 
\beq\label{siju}
s_{ij}(u)=  \delta_{ij}+s_{ij}^{(1)}u^{-1}+s_{ij}^{(2)}u^{-2}+\cdots.
\eeq
The \textit{twisted Yangian} $\Yi_{\mathscr R}$ of quasi-split type $A$ is the unital associative algebra  with generators $s_{ij}^{(r)}$, where $1\lle i,j\lle N$ and $r\in\bZ_{>0}$, whose generating series \eqref{siju} satisfy the following quaternary relation:
\be
\begin{split}
(u^2-v^2)[s_{ij}(u),s_{kl}(v)]=  (u+v)(s_{kj}(u)s_{il}(v)-s_{kj}(v)s_{il}(u))&\\
  + (u-v)\Big(\delta_{kj'}\sum_{a=1}^N s_{ia'}(u)s_{al}(v)-\delta_{il'}\sum_{a=1}^N s_{ka'}(v)s_{aj}(u)\Big)&\\
 - \delta_{ij'}\Big(\sum_{a=1}^N s_{ka'}(u)s_{al}(v)-\sum_{a=1}^N s_{ka'}(v)s_{al}(u)\Big)&
\end{split}
\ee
and the unitary relations
\be
\sum_{a=1}^N s_{ia'}(u)s_{aj}(-u)=\delta_{ij'}.
\ee

The twisted Yangian $\Yi_{\mathscr R}$ can be identified as a subalgebra of $\Y_{\mathscr R}$ via
\beq\label{eq:embed-R}
s_{ij}(u)=\sum_{k=1}^N \tl t_{i'k'}(-u)t_{kj}(u).
\eeq
This version is slightly different from the one in \cite{LZ24}, in order to make it a \textit{right} coideal subalgebra, cf. \cite{LPTTW25} and \cite[Appendix]{Lu26min}. Thus, we have
\begin{align}
  s_{ij}^{(1)}
   &=t_{ij}^{(1)}-\tilde t_{i'j'}^{(1)}
     =t_{ij}^{(1)}+t_{i'j'}^{(1)},\notag\\
  s_{ij}^{(2)}
   &=t_{ij}^{(2)}+\tilde t_{i'j'}^{(2)}
     -\sum_{k=1}^{N}\tilde t_{i'k'}^{(1)}t_{kj}^{(1)}=t_{ij}^{(2)}-t_{i'j'}^{(2)}
     +\sum_{k=1}^{N}t_{i'k}^{(1)}t_{kj'}^{(1)}
     +\sum_{k=1}^{N}t_{i'k'}^{(1)}t_{kj}^{(1)}.
     \label{eq:s-second}
\end{align}

Introduce new generating series via Gauss decomposition, for $i<j$, 
\beq\label{eq:sij-Gauss}
\begin{split}
s_{ii}(u)&=d_i(u)+\sum_{k<i}f_{ik}(u)d_k(u)e_{ki}(u),\\
s_{ij}(u)&=d_i(u)e_{ij}(u)+\sum_{k<i}f_{ik}(u)d_k(u)e_{kj}(u),\\
s_{ji}(u)&=f_{ji}(u)d_i(u)+\sum_{k<i}f_{jk}(u)d_k(u)e_{ki}(u).
\end{split}
\eeq
It is shown in \cite[Lemma~6.8]{LZ24} that $f_i(u)=-e_{\tau i}(-u)$. It has been established in \cite{LZ24} that the iYangian $\Yi$ in Drinfeld presentation can be identified as a subalgebra of $\Yi_{\mathscr R}$ via the following correspondence:
\beq\label{eq:Y2i-R}
b_i(u)=\sqrt{-1}f_i(u+\vartheta_i),\qquad 
h_i(u)=\big(1+\tfrac{\wp_i}{4u}\big) d_{i}(u+\vartheta_i)^{-1}d_{i+1}(u+\vartheta_i)
\eeq
for $1\lle i<N$, where $\wp_i$ is defined in \eqref{eq:wpi-def}.

\subsection{Obtaining $h_{i,0}$ and $h_{i,1}$}\label{ssec:obtain-h}
It follows from \eqref{eq:sij-Gauss} that \begin{align}
&d_{i}^{(1)}=s_{ii}^{(1)}=t_{ii}^{(1)}+t_{i'i'}^{(1)}=D_i^{(1)}+D_{i'}^{(1)},\label{eq:di-first}\\
&f_{i+1,i}^{(1)}=s_{i+1,i}^{(1)}=t_{i'-1,i'}^{(1)}+t_{i+1,i}^{(1)}.\notag
\end{align}
Moreover, we have
\begin{align}
&s_{ii}^{(2)}
   =d_i^{(2)}+
     \sum_{k<i}f_{ik}^{(1)}e_{ki}^{(1)}=d_i^{(2)}+
     \sum_{k<i}\bigl(t_{ik}^{(1)}+t_{i'k'}^{(1)}\bigr)
              \bigl(t_{ki}^{(1)}+t_{k'i'}^{(1)}\bigr).\label{eq:sii-gauss}
\end{align}
On the other hand, we obtain from \eqref{eq:t-diag-used} and \eqref{eq:s-second} that
\beq\label{eq:sii-rtt}
  s_{ii}^{(2)}
   =D_i^{(2)}+
     \sum_{k<i}t_{ik}^{(1)}t_{ki}^{(1)}
     -D_{i'}^{(2)}-
     \sum_{k<i'}t_{i'k}^{(1)}t_{ki'}^{(1)}
     +\sum_{k=1}^{N}t_{i'k}^{(1)}t_{ki'}^{(1)}
     +\sum_{k=1}^{N}t_{i'k'}^{(1)}t_{ki}^{(1)}.
\eeq
Combining \eqref{eq:sii-gauss} and \eqref{eq:sii-rtt} gives
\beq\label{eq:di-second-final}
\begin{split}
  d_i^{(2)}
   ={}&D_i^{(2)}+
     \sum_{k<i}t_{ik}^{(1)}t_{ki}^{(1)}
     -D_{i'}^{(2)}-
     \sum_{k<i'}t_{i'k}^{(1)}t_{ki'}^{(1)} \\
   &+\sum_{k=1}^{N}t_{i'k}^{(1)}t_{ki'}^{(1)}
     +\sum_{k=1}^{N}t_{i'k'}^{(1)}t_{ki}^{(1)}
     -\sum_{k<i}
      \bigl(t_{ik}^{(1)}+t_{i'k'}^{(1)}\bigr)
      \bigl(t_{ki}^{(1)}+t_{k'i'}^{(1)}\bigr) \\
   ={}&D_i^{(2)}-D_{i'}^{(2)}
     +\bigl(D_{i'}^{(1)}\bigr)^2
     +D_i^{(1)}D_{i'}^{(1)}
     +\sum_{k<i'}t_{i'k}^{(1)}t_{k'i}^{(1)}
     -\sum_{k<i}t_{ik}^{(1)}t_{k'i'}^{(1)}.
\end{split}
\eeq
By \eqref{eq:Y2i-R}, we have
\begin{align*}
  h_{i,0}&=d_{i+1}^{(1)}-d_i^{(1)}+\tfrac14 \wp_i\stackrel{\eqref{eq:di-first}}{=}\xi_i-\xi_{\tau i}+\tfrac14 \wp_i,\\
  h_{i,1}
  &=d_{i+1}^{(2)}-d_i^{(2)}+\bigl(d_i^{(1)}\bigr)^2
   +\big(\tfrac{1}{4}\wp_i-\vartheta_i\big) \big(d_{i+1}^{(1)}-d_i^{(1)}\big)-d_i^{(1)}d_{i+1}^{(1)}.
\end{align*}
Using \eqref{eq:di-first} and \eqref{eq:di-second-final}, we obtain
\be
\begin{split}
  h_{i,1}
  =\,&D_{i+1}^{(2)}-D_{i'-1}^{(2)}
     +\bigl(D_{i'-1}^{(1)}\bigr)^2
     +D_{i+1}^{(1)}D_{i'-1}^{(1)}
     +\sum_{k<i'-1}t_{i'-1,k}^{(1)}t_{k',i+1}^{(1)}
     \\&-\sum_{k<i+1}t_{i+1,k}^{(1)}t_{k',i'-1}^{(1)}
  -D_i^{(2)}+D_{i'}^{(2)}
     -\bigl(D_{i'}^{(1)}\bigr)^2
     -D_i^{(1)}D_{i'}^{(1)}\\
     &-\sum_{k<i'}t_{i'k}^{(1)}t_{k'i}^{(1)}
     +\sum_{k<i}t_{ik}^{(1)}t_{k'i'}^{(1)}
+\bigl(D_i^{(1)}\bigr)^2+\bigl(D_{i'}^{(1)}\bigr)^2
     +2D_i^{(1)}D_{i'}^{(1)}\\
&     +\big(\tfrac{1}{4}\wp_i-\vartheta_i\big)\bigl(D_{i+1}^{(1)}+D_{i'-1}^{(1)}\bigr)
     -\big(\tfrac{1}{4}\wp_i-\vartheta_i\big)\bigl(D_i^{(1)}+D_{i'}^{(1)}\bigr) \\
  &-D_i^{(1)}D_{i+1}^{(1)}-D_{i'}^{(1)}D_{i'-1}^{(1)}
     -D_i^{(1)}D_{i'-1}^{(1)}-D_{i+1}^{(1)}D_{i'}^{(1)}.
\end{split}
\ee
Applying \eqref{eq:xpm0} and \eqref{eq:xi1}, it becomes
\beq\label{eq:hi1-before-root}
\begin{split}
h_{i,1}
   =\,&\xi_{i,1}+\xi_{\tau i,1}
      -\xi_{i}\xi_{\tau i}+\tfrac14\wp_i\big(\xi_i-\xi_{\tau i}\big)\\
&+\sum_{k<i'-1}t_{i'-1,k}^{(1)}t_{k',i+1}^{(1)}
      -\sum_{k<i+1}t_{i+1,k}^{(1)}t_{k',i'-1}^{(1)}
      -\sum_{k<i'}t_{i'k}^{(1)}t_{k'i}^{(1)}
      +\sum_{k<i}t_{ik}^{(1)}t_{k'i'}^{(1)}.
\end{split}      
\eeq
For $\alpha=\ve_i-\ve_j\in \Phi^+$ with $1\lle i<j\lle N$, we choose the root vectors as follows:
\[
x_\alpha^+=\sqrt{-1}\,t_{ji}^{(1)}=(\sqrt{-1})^{j-i-1}\mathbf e_{ij},\qquad x_\alpha^-=-\sqrt{-1}\,t_{ij}^{(1)}=(\sqrt{-1})^{i-j+1}\mathbf e_{ji}.
\]
There are two cases depending on the parity of $N$. 
\begin{enumerate}
    \item If $N=2n$, then we find that $\wp_i=0$ for all $1\lle i<N$ and $[x_\alpha^+,x_{\tau \alpha}^+]=0$. Hence we obtain from \eqref{eq:hi1-before-root} that
    \[
    h_{i,1}=\xi_{i,1}+\xi_{\tau i,1}
      -\xi_{i}\xi_{\tau i}+\hf\sum_{\alpha\in\Phi^+}(\alpha,\alpha_i)\big\{x_\alpha^+,x_{\tau\alpha}^+\big\}.
    \]
    \item If $N=2n+1$, then both $\wp_i$ and $[x_\alpha^+,x_{\tau \alpha}^+]$ may be nonzero for $i\in\I$ and $\alpha\in \Phi^+$. Hence we have
    \[
    h_{i,1}={}\xi_{i,1}+\xi_{\tau i,1}
      -\xi_{i}\xi_{\tau i}+\tfrac14\wp_i(\xi_i-\xi_{\tau i})+\sum_{\alpha\in\Phi^+_{i,<}}(\alpha,\alpha_i)x_\alpha^+x_{\tau\alpha}^++\sum_{\alpha\in\Phi^+_{i,>}}(\alpha,\alpha_i)x_{\tau\alpha}^+x_\alpha^+.
    \]
    If we rewrite it using the anti-commutator, then
    \[
    h_{i,1}=\xi_{i,1}+\xi_{\tau i,1}
      -\xi_{i}\xi_{\tau i}+\hf\sum_{\alpha\in\Phi^+}(\alpha,\alpha_i)\big\{x_\alpha^+,x_{\tau\alpha}^+\big\}+\hf\Theta_i,
    \]
    where $\Theta_i$ is defined in \eqref{eq:Theta}. 
\end{enumerate}

\subsection{Obtaining $b_{i,0}$ and $b_{i,1}$}\label{ssec:obtain-b}
It follows from \eqref{eq:sij-Gauss} that
\begin{align*}
  f_i^{(1)}
    &=s_{i+1,i}^{(1)}
      =t_{i'-1,i'}^{(1)}+t_{i+1,i}^{(1)},\\
  s_{i+1,i}^{(2)}
    &=f_i^{(2)}+f_i^{(1)}d_i^{(1)}+
      \sum_{k<i}f_{i+1,k}^{(1)}e_{ki}^{(1)}=\tl t_{i'-1,i'}^{(2)}+t_{i+1,i}^{(2)}
      +\sum_{k=1}^{N}t_{i'-1,k'}^{(1)}t_{ki}^{(1)}.
\end{align*}
Hence we have
\begin{align*}
f_i^{(2)}
=\,&~
\tl{t}_{i'-1,i'}^{(2)}
+t_{i+1,i}^{(2)}
+\sum_{k=1}^{N} t_{i'-1,k'}^{(1)}t_{ki}^{(1)}
\\
&-\big(t_{i'-1,i'}^{(1)}+t_{i+1,i}^{(1)}\big)
 \big(t_{ii}^{(1)}+t_{i'i'}^{(1)}\big)
-\sum_{k<i}
\big(t_{i'-1,k'}^{(1)}+t_{i+1,k}^{(1)}\big)
\big(t_{k'i'}^{(1)}+t_{ki}^{(1)}\big)
\\
=\,&
-t_{i'-1,i'}^{(2)}
+\sum_{k=1}^{N}t_{i'-1,k}^{(1)}t_{ki'}^{(1)}
+F_i^{(2)}
+F_i^{(1)}D_i^{(1)}
+\sum_{k<i}F_{i+1,k}^{(1)}E_{ki}^{(1)}
+\sum_{k=1}^{N}t_{i'-1,k'}^{(1)}t_{ki}^{(1)}
\\
&\quad
-\big(t_{i'-1,i'}^{(1)}+t_{i+1,i}^{(1)}\big)
 \big(t_{ii}^{(1)}+t_{i'i'}^{(1)}\big)
-\sum_{k<i}
\big(t_{i'-1,k'}^{(1)}+t_{i+1,k}^{(1)}\big)
\big(t_{k'i'}^{(1)}+t_{ki}^{(1)}\big)
\\
=\,&
-E_{i'-1}^{(2)}
-D_{i'-1}^{(1)}E_{i'-1}^{(1)}
-\sum_{k<i'-1}t_{i'-1,k}^{(1)}t_{ki'}^{(1)}
+\sum_{k=1}^{N}t_{i'-1,k}^{(1)}t_{ki'}^{(1)}
\\
&
+F_i^{(2)}
+t_{i+1,i}^{(1)}t_{ii}^{(1)}
+\sum_{k<i}t_{i+1,k}^{(1)}t_{ki}^{(1)}
+\sum_{k=1}^{N}t_{i'-1,k'}^{(1)}t_{ki}^{(1)}
\\
&
-t_{i'-1,i'}^{(1)}t_{ii}^{(1)}
-t_{i'-1,i'}^{(1)}t_{i'i'}^{(1)}
-t_{i+1,i}^{(1)}t_{ii}^{(1)}
-t_{i+1,i}^{(1)}t_{i'i'}^{(1)}
\\
&
-\sum_{k<i}t_{i'-1,k'}^{(1)}t_{k'i'}^{(1)}
-\sum_{k<i}t_{i'-1,k'}^{(1)}t_{ki}^{(1)}
-\sum_{k<i}t_{i+1,k}^{(1)}t_{k'i'}^{(1)}
-\sum_{k<i}t_{i+1,k}^{(1)}t_{ki}^{(1)}
\\[2mm]
={}&
-E_{i'-1}^{(2)}
+F_i^{(2)}
+\sum_{k>i}t_{i'-1,k'}^{(1)}t_{ki}^{(1)}
-\sum_{k<i}t_{i+1,k}^{(1)}t_{k'i'}^{(1)}
-t_{i+1,i}^{(1)}t_{i'i'}^{(1)}
\\[2mm]
={}&
-E_{i'-1}^{(2)}
+F_i^{(2)}
+\sum_{k>i+1}t_{i'-1,k'}^{(1)}t_{ki}^{(1)}
-\sum_{k<i}t_{i+1,k}^{(1)}t_{k'i'}^{(1)}
+t_{i'-1,i'-1}^{(1)}t_{i+1,i}^{(1)}
-t_{i+1,i}^{(1)}t_{i'i'}^{(1)}
\\
=\,&
-E_{i'-1}^{(2)}
+F_i^{(2)}
+\sum_{k>i+1}
t_{i'-1,k'}^{(1)}
\big[t_{k,i+1}^{(1)},t_{i+1,i}^{(1)}\big]
-\sum_{k<i}
\big[t_{i+1,i}^{(1)},t_{ik}^{(1)}\big]
t_{k'i'}^{(1)}
\\
& 
-\tfrac{1}{2}
\big\{
t_{i'i'}^{(1)}-t_{i'-1,i'-1}^{(1)},
t_{i+1,i}^{(1)}
\big\}
-\tfrac{1}{2}
\big[t_{i+1,i}^{(1)},t_{i'i'}^{(1)}\big]
+\tfrac{1}{2}
\big[t_{i'-1,i'-1}^{(1)},t_{i+1,i}^{(1)}\big].
\end{align*}
Thus it follows from \eqref{eq:Y2i-R} that
\begin{align*}
b_{i,0}&=\imag\,f_i^{(1)}=\imag(t_{i'-1,i'}^{(1)}+t_{i+1,i}^{(1)})=x_i^+-x_{\tau i}^-,\\
  b_{i,1}&=\imag\big(f_i^{(2)}-
       \vartheta_if_i^{(1)}\big)\\
  &=\imag\Big(
      F_i^{(2)}-\vartheta_iF_i^{(1)}-E_{i'-1}^{(2)}+\vartheta_{\tau i}E_{i'-1}^{(1)}     
      \notag\\
  &\hspace{4em}
      +\sum_{k>i+1}t_{i'-1,k'}^{(1)}
        \comm{t_{k,i+1}^{(1)}}{t_{i+1,i}^{(1)}}-\sum_{k<i}\comm{t_{i+1,i}^{(1)}}{t_{ik}^{(1)}}t_{k'i'}^{(1)}\\
  &\hspace{4em}-\tfrac12\acomm{t_{i'i'}^{(1)}-t_{i'-1,i'-1}^{(1)}}{t_{i+1,i}^{(1)}}
      -\tfrac12\comm{t_{i+1,i}^{(1)}}{t_{i'i'}^{(1)}}
      +\tfrac12\comm{t_{i'-1,i'-1}^{(1)}}{t_{i+1,i}^{(1)}}
      \Big).
\end{align*}
In terms of root vectors, the final expression is
\begin{equation*}
b_{i,1}
   =x_{i,1}^{+}+x_{\tau i,1}^{-}
    -\tfrac12\acomm{\xi_{\tau i}}{x_i^{+}}+\hf\wp_i x_i^++\sum_{\alpha}
       \comm{x_i^{+}}{x_{\alpha}^{+}}x_{\tau\alpha}^{+}
    +\sum_{\beta}
       x_{\tau\beta}^{+}\comm{x_i^{+}}{x_{\beta}^{+}},
\end{equation*}
where $\alpha$ runs over all positive roots $\ve_k-\ve_i$ for $1\lle k<i$ while $\beta$ runs over all positive roots $\ve_{i+1}-\ve_k$ for $i+1<k\lle N$.

\subsection{Conclusion}
In the previous subsections, we have expressed $h_{i,r},b_{i,r}$ for $i\in \I$ and $r=0,1$ defined in \eqref{eq:sij-Gauss}--\eqref{eq:Y2i-R} in terms of the Drinfeld generators of the Yangian if we identify
\begin{itemize}
    \item $\Yi_{\mathscr R}$ as a subalgebra of $\Y_{\mathscr R}$ via \eqref{eq:embed-R},
    \item $\Y$ as a subalgebra of $\Y_{\mathscr R}$ via \eqref{tij-Gauss}--\eqref{eq:Y2-R}.
\end{itemize}
These expressions are precisely those given in \eqref{hi0-embedding2}--\eqref{bi1-embedding2}. Since it is verified in \cite{LZ24} that $h_{i,r},b_{i,r}$ for $i\in \I$ and $r\in \bN$ satisfy the relations in \eqref{qs1}--\eqref{finserre4-ty}, we conclude by Lemma \ref{lem:hbhb} that $h_{i,r},b_{i,r}$ for $i\in \I$ and $r=0,1$  satisfy the relations \eqref{redqs1}--\eqref{redqs5}, \eqref{finserre1-ty}--\eqref{finserre4-ty}, and \eqref{eq:extraA2-1}--\eqref{eq:extraA2-2}. Therefore, the map $\varphi:\Yi\to \Y$ defined by \eqref{hi0-embedding2}--\eqref{bi1-embedding2} induces an algebra homomorphism.

\bibliographystyle{amsalpha}
\bibliography{reference}

@article{Lu27super,
  AUTHOR = {Lu, K.},
  TITLE = {Twisted super {Y}angians of quasi-split type {A}},
  JOURNAL = {J. Algebra},
  FJOURNAL = {Journal of Algebra},
  VOLUME = {711},
  YEAR = {2027},
  PAGES = {474--529},
  DOI = {10.1016/j.jalgebra.2026.06.035},
  URL = {https://doi.org/10.1016/j.jalgebra.2026.06.035},
}

@article{LuPan26Drinfeld,
  AUTHOR = {Lu, M. and Pan, X.},
  TITLE = {{Drinfeld presentation of affine $\mathrm{i}$quantum groups via $\mathrm{i}$Hopf algebras}},
  JOURNAL = {preprint},
  FJOURNAL = {},
  VOLUME = {},
  YEAR = {2026},
  PAGES = {40pp}
}

@article{Ueda25twistedaffine,
  AUTHOR = {Ueda, M.},
  TITLE = {Suggestion of a definition of the twisted affine {Y}angian of type {$D$}},
  JOURNAL = {arXiv preprint \arxiv{2512.06340}},
  YEAR = {2025},
}

@article{LWW26affine,
  author  = {Lu, K. and Wang, W. and Weekes, A.},
  title   = {{Shifted affine $\imath$quantum groups of quasi-split ADE types}},
  journal = {arXiv preprint \arxiv{2603.28446}},
  year    = {2026}
}

@article{LP26GKLO,
  author  = {Li, J.-R. and Prze\'zdziecki, T.},
  title   = {{GKLO representations for shifted quantum affine symmetric pairs}},
  journal = {arXiv preprint \arxiv{2603.07250}},
  year    = {2026}
}

@article{shen2026weight,
  title={{On weight modules over truncated shifted iYangians}},
  author={Shen, Y. and Song, L. and Xiong, R.},
  journal={arXiv preprint \arxiv{2607.27596}},
  year={2026}
}

@article{wang2025quivers,
  title={{Quivers with involutions and shifted twisted Yangians via Coulomb branches II}},
  author={Wang, Z.},
  journal={arXiv preprint arXiv:2601.00039},
  year={2025}
}

@article{su2026twisted,
  title={{Twisted Yangians and Steinberg varieties of type C}},
  author={Su, C. and Yang, Y.},
  journal={arXiv preprint arXiv:2603.28084},
  year={2026}
}

@article {CP91,
    AUTHOR = {Chari, V. and Pressley, A.},
     TITLE = {Fundamental representations of {Y}angians and singularities of
              {$R$}-matrices},
   JOURNAL = {J. Reine Angew. Math.},
  FJOURNAL = {Journal f\"ur die Reine und Angewandte Mathematik. [Crelle's
              Journal]},
    VOLUME = {417},
      YEAR = {1991},
     PAGES = {87--128},
      ISSN = {0075-4102,1435-5345},
   MRCLASS = {17B37},
  MRNUMBER = {1103907},
MRREVIEWER = {Vladimir\ A.\ Stukopin},
       DOI = {10.1515/crll.1991.417.87},
       URL = {https://doi.org/10.1515/crll.1991.417.87},
}

@article {GW23,
    AUTHOR = {Gautam, S. and Wendlandt, C.},
     TITLE = {Poles of finite-dimensional representations of {Y}angians},
   JOURNAL = {Selecta Math. (N.S.)},
  FJOURNAL = {Selecta Mathematica. New Series},
    VOLUME = {29},
      YEAR = {2023},
    NUMBER = {1},
     PAGES = {Paper No. 13, 68},
      ISSN = {1022-1824,1420-9020},
   MRCLASS = {17B37 (81R05)},
  MRNUMBER = {4520258},
MRREVIEWER = {Cristian\ Vay},
       DOI = {10.1007/s00029-022-00813-y},
       URL = {https://doi.org/10.1007/s00029-022-00813-y},
}

@article {CGM14,
    AUTHOR = {Chen, H. and Guay, N. and Ma, X.},
     TITLE = {Twisted {Y}angians, twisted quantum loop algebras and affine
              {H}ecke algebras of type {$BC$}},
   JOURNAL = {Trans. Amer. Math. Soc.},
  FJOURNAL = {Transactions of the American Mathematical Society},
    VOLUME = {366},
      YEAR = {2014},
    NUMBER = {5},
     PAGES = {2517--2574},
      ISSN = {0002-9947,1088-6850},
   MRCLASS = {17B37 (20C08)},
  MRNUMBER = {3165646},
MRREVIEWER = {Guiyu\ Yang},
       DOI = {10.1090/S0002-9947-2014-05994-1},
       URL = {https://doi.org/10.1090/S0002-9947-2014-05994-1},
}

@article {Kolb14,
    AUTHOR = {Kolb, S.},
     TITLE = {Quantum symmetric {K}ac-{M}oody pairs},
   JOURNAL = {Adv. Math.},
  FJOURNAL = {Advances in Mathematics},
    VOLUME = {267},
      YEAR = {2014},
     PAGES = {395--469},
      ISSN = {0001-8708,1090-2082},
   MRCLASS = {17B37 (17B67)},
  MRNUMBER = {3269184},
MRREVIEWER = {Darren\ Funk-Neubauer},
       DOI = {10.1016/j.aim.2014.08.010},
       URL = {https://doi.org/10.1016/j.aim.2014.08.010},
}

@article{TT24,
    AUTHOR = {Tappeiner, L. and Topley, L.},
     TITLE = {Shifted twisted {Y}angians and {S}lodowy slices in classical
              {L}ie algebras},
   JOURNAL = {Forum Math. Sigma},
  FJOURNAL = {Forum of Mathematics, Sigma},
    VOLUME = {14},
      YEAR = {2026},
     PAGES = {Paper No. e103, 36},
      ISSN = {2050-5094},
       DOI = {10.1017/fms.2026.10247},
       URL = {https://doi.org/10.1017/fms.2026.10247},
}

@article {MR02,
    AUTHOR = {Molev, A. and Ragoucy, E.},
     TITLE = {Representations of reflection algebras},
   JOURNAL = {Rev. Math. Phys.},
  FJOURNAL = {Reviews in Mathematical Physics. A Journal for Both Review and
              Original Research Papers in the Field of Mathematical Physics},
    VOLUME = {14},
      YEAR = {2002},
    NUMBER = {3},
     PAGES = {317--342},
      ISSN = {0129-055X,1793-6659},
   MRCLASS = {16S99 (17B37 82B23)},
  MRNUMBER = {1894013},
MRREVIEWER = {Alexei\ P.\ Isaev},
       DOI = {10.1142/S0129055X02001156},
       URL = {https://doi.org/10.1142/S0129055X02001156},
}

@article {Zha22,
    AUTHOR = {Zhang, W.},
     TITLE = {A {D}rinfeld type presentation of affine {$\imath$}quantum
              groups {II}: split {BCFG} type},
   JOURNAL = {Lett. Math. Phys.},
  FJOURNAL = {Letters in Mathematical Physics},
    VOLUME = {112},
      YEAR = {2022},
    NUMBER = {5},
     PAGES = {Paper No. 89, 33},
      ISSN = {0377-9017},
   MRCLASS = {17B37 (17B67)},
  MRNUMBER = {4482099},
MRREVIEWER = {Andrea Appel},
       DOI = {10.1007/s11005-022-01583-6},
       URL = {https://doi.org/10.1007/s11005-022-01583-6},
}

@article {LW21,
    AUTHOR = {Lu, M. and Wang, W.},
     TITLE = {A {D}rinfeld type presentation of affine {$\imath$}quantum
              groups {I}: {S}plit {ADE} type},
   JOURNAL = {Adv. Math.},
  FJOURNAL = {Advances in Mathematics},
    VOLUME = {393},
      YEAR = {2021},
     PAGES = {Paper No. 108111, 46},
      ISSN = {0001-8708},
   MRCLASS = {17B37 (17B67)},
  MRNUMBER = {4340233},
MRREVIEWER = {Andrea Appel},
       DOI = {10.1016/j.aim.2021.108111},
       URL = {https://doi.org/10.1016/j.aim.2021.108111},
}

@article {HZ24,
    AUTHOR = {Hernandez, D. and Zhang, H.},
     TITLE = {Shifted {Y}angians and polynomial {$R$}-matrices},
   JOURNAL = {Publ. Res. Inst. Math. Sci.},
  FJOURNAL = {Publications of the Research Institute for Mathematical
              Sciences},
    VOLUME = {60},
      YEAR = {2024},
    NUMBER = {1},
     PAGES = {1--69},
      ISSN = {0034-5318,1663-4926},
   MRCLASS = {20G42 (16T25 81R50)},
  MRNUMBER = {4803333},
MRREVIEWER = {Run-Qiang\ Jian},
       DOI = {10.4171/prims/60-1-1},
       URL = {https://doi.org/10.4171/prims/60-1-1},
}

@article {GR16,
    AUTHOR = {Guay, N. and Regelskis, V.},
     TITLE = {Twisted {Y}angians for symmetric pairs of types {B}, {C}, {D}},
   JOURNAL = {Math. Z.},
  FJOURNAL = {Mathematische Zeitschrift},
    VOLUME = {284},
      YEAR = {2016},
    NUMBER = {1-2},
     PAGES = {131--166},
      ISSN = {0025-5874,1432-1823},
   MRCLASS = {17B67 (16T25 81R10)},
  MRNUMBER = {3545488},
MRREVIEWER = {Chengming\ Bai},
       DOI = {10.1007/s00209-016-1649-2},
       URL = {https://doi.org/10.1007/s00209-016-1649-2},
}

@article {Li19,
    AUTHOR = {Li, Y.},
     TITLE = {Quiver varieties and symmetric pairs},
   JOURNAL = {Represent. Theory},
  FJOURNAL = {Representation Theory. An Electronic Journal of the American
              Mathematical Society},
    VOLUME = {23},
      YEAR = {2019},
     PAGES = {1--56},
   MRCLASS = {16S30 (14J50 14L35 53D05)},
  MRNUMBER = {3900699},
MRREVIEWER = {Kevin D. Coulembier},
       DOI = {10.1090/ert/522},
       URL = {https://doi.org/10.1090/ert/522},
}

@article {BK05,
    AUTHOR = {Brundan, J. and Kleshchev, A.},
     TITLE = {Parabolic presentations of the {Y}angian {$Y({\mathfrak{gl}}_n)$}},
   JOURNAL = {Comm. Math. Phys.},
  FJOURNAL = {Communications in Mathematical Physics},
    VOLUME = {254},
      YEAR = {2005},
    NUMBER = {1},
     PAGES = {191--220},
      ISSN = {0010-3616,1432-0916},
   MRCLASS = {17B37},
  MRNUMBER = {2116743},
MRREVIEWER = {Sonia\ Natale},
       DOI = {10.1007/s00220-004-1249-6},
       URL = {https://doi.org/10.1007/s00220-004-1249-6},
}

@article {MNO96,
    AUTHOR = {Molev, A. and Nazarov, M. and Olshanski, G.},
     TITLE = {Yangians and classical {L}ie algebras},
   JOURNAL = {Uspekhi Mat. Nauk},
  FJOURNAL = {Uspekhi Matematicheskikh Nauk},
    VOLUME = {51},
      YEAR = {1996},
    NUMBER = {2(308)},
     PAGES = {27--104},
      ISSN = {0042-1316,2305-2872},
   MRCLASS = {17B37},
  MRNUMBER = {1401535},
MRREVIEWER = {Ya.\ S.\ So\u ibel\cprime man},
       DOI = {10.1070/RM1996v051n02ABEH002772},
       URL = {https://doi.org/10.1070/RM1996v051n02ABEH002772},
}

@article{SSX26,
    author = {Shen, Y. and Su, C. and Xiong, R.},
    title = {{Quantized Coulomb Branches of Separated Cotangent Type and Orthosymplectic Quivers}},
   JOURNAL = {arXiv preprint \arxiv{2608.16091}},
      YEAR = {2026}
}

@article{SSX25,
    author = {Shen, Y. and Su, C. and Xiong, R.},
    title = {{Quivers with involutions and shifted twisted Yangians via Coulomb branches}},
   JOURNAL = {to appear in Commun. Math. Phys., arXiv preprint \arxiv{2510.12118}},
      YEAR = {2026}
}

@article{BPT25,
    author = {Bartlett, R. and Przeździecki, T. and Tappeiner, L.},
    title = {{GKLO representations of twisted Yangians in type $\sf AI$ and quantizations of symmetric quotients of the affine Grassmannian}},
   JOURNAL = {arXiv preprint \arxiv{2510.12706}},
      YEAR = {2025}
}

@article {GTL16,
    AUTHOR = {Gautam, S. and Toledano Laredo, V.},
     TITLE = {Yangians, quantum loop algebras, and abelian difference
              equations},
   JOURNAL = {J. Amer. Math. Soc.},
  FJOURNAL = {Journal of the American Mathematical Society},
    VOLUME = {29},
      YEAR = {2016},
    NUMBER = {3},
     PAGES = {775--824},
      ISSN = {0894-0347,1088-6834},
   MRCLASS = {17B67 (17B37)},
  MRNUMBER = {3486172},
MRREVIEWER = {Rutwig\ Campoamor-Stursberg},
       DOI = {10.1090/jams/851},
       URL = {https://doi.org/10.1090/jams/851},
}

@article {Lev93gen,
    AUTHOR = {Levendorskii, S. Z.},
     TITLE = {On generators and defining relations of {Y}angians},
   JOURNAL = {J. Geom. Phys.},
  FJOURNAL = {Journal of Geometry and Physics},
    VOLUME = {12},
      YEAR = {1993},
    NUMBER = {1},
     PAGES = {1--11},
      ISSN = {0393-0440,1879-1662},
   MRCLASS = {17B37 (81R50 82B23)},
  MRNUMBER = {1226802},
MRREVIEWER = {Marco\ Tarlini},
       DOI = {10.1016/0393-0440(93)90084-R},
       URL = {https://doi.org/10.1016/0393-0440(93)90084-R},
}

@article {LWW25,
    AUTHOR = {Lu, K. and Wang, W. and Weekes, A.},
     TITLE = {{Shifted twisted Yangians and affine Grassmannian islices}},
   JOURNAL = {{arXiv preprint \arxiv{2510.10652}}},
      YEAR = {2025},
}

@article {LWW25SiY,
    AUTHOR = {Lu, K. and Wang, W. and Weekes, A.},
     TITLE = {{Shifted twisted Yangians of quasi-split ADE types}},
   JOURNAL = {{arXiv preprint \arxiv{2512.19998}}},
      YEAR = {2025},
}

@article {Nak25,
    AUTHOR = {Nakajima, H.},
     TITLE = {{Instantons on ALE spaces for classical groups, involutions on quiver varieties, and quantum symmetric pairs}},
   JOURNAL = {{arXiv preprint \arxiv{2510.13007}}},
      YEAR = {2025},
}

@article {LPTTW25,
    AUTHOR = {Lu, K. and Peng, Y.-N. and Tappeiner, L. and Topley, L. and Wang, W.},
     TITLE = {{Shifted twisted Yangians and finite W-algebras of classical type}},
   JOURNAL = {arXiv preprint \arxiv{2505.03316}},
      YEAR = {2025},
}

@article {LWZ25GD,
    AUTHOR = {Lu, K. and Wang, W. and Zhang, W.},
     TITLE = {{A Drinfeld type Presentation of twisted Yangians}},
   JOURNAL = {Represent. Theory},
  FJOURNAL = {Representation Theory. An Electronic Journal of the American
              Mathematical Society},
    VOLUME = {29},
      YEAR = {2025},
     PAGES = {838--870}
}

@article {LZ24,
    AUTHOR = {Lu, K. and Zhang, W.},
     TITLE = {{A Drinfeld type Presentation of twisted Yangians of quasi-split type}},
   JOURNAL = {Commun. Contemp. Math., arXiv preprint \arxiv{2408.06981}},
      YEAR = {2025},
}

@article{regelskis2026twisted,
  title={{Twisted Yangians of types BI, CI, DI and Drinfeld type current relations}},
  author={Regelskis, Vidas},
  journal={arXiv preprint \arxiv{2607.14692}},
  year={2026}
}

@article{LP26qchar,
  AUTHOR = {Li, J.-R. and Prze\'zdziecki, T.},
  TITLE = {Compatibility of {D}rinfeld presentations and {$q$}-characters
           for affine {K}ac--{M}oody quantum symmetric pairs:
           quasi-split case},
  JOURNAL = {arXiv preprint \arxiv{2601.02165}},
  YEAR = {2026},
}

@article{regelskis2026note,
  title={{A note on $X^{\mathsf{tw}}(\mathfrak o_4,\mathfrak o_3)$}},
  author={Regelskis, Vidas},
  journal={private note},
  year={2026}
}

@article {Skl88,
    AUTHOR = {Sklyanin, E.},
     TITLE = {Boundary conditions for integrable quantum systems},
   JOURNAL = {J. Phys. A},
  FJOURNAL = {Journal of Physics. A. Mathematical and General},
    VOLUME = {21},
      YEAR = {1988},
    NUMBER = {10},
     PAGES = {2375--2389},
      ISSN = {0305-4470,1751-8121},
   MRCLASS = {81F99 (58F07 81E25 82A68 82A69)},
  MRNUMBER = {953215},
MRREVIEWER = {M.\ Lawrence\ Glasser},
       URL = {http://stacks.iop.org/0305-4470/21/2375},
}

@article {Che84,
    AUTHOR = {Cherednik, I.},
     TITLE = {{Factorizing particles on a half line, and root systems}},
   JOURNAL = {Teoret. Mat. Fiz.},
  FJOURNAL = {Akademiya Nauk SSSR. Teoreticheskaya i Matematicheskaya
              Fizika},
    VOLUME = {61},
      YEAR = {1984},
    NUMBER = {1},
     PAGES = {35--44},
      ISSN = {0564-6162},
   MRCLASS = {81F15 (58F07 82A67)},
  MRNUMBER = {774205},
MRREVIEWER = {V.\ Z.\ Enol\cprime ski\u{\i}},
}

@incollection {Ols92,
    AUTHOR = {Olshanski, G.},
     TITLE = {Twisted {Y}angians and infinite-dimensional classical {L}ie
              algebras},
 BOOKTITLE = {Quantum groups ({L}eningrad, 1990)},
    SERIES = {Lecture Notes in Math.},
    VOLUME = {1510},
     PAGES = {104--119},
 PUBLISHER = {Springer, Berlin},
      YEAR = {1992},
      ISBN = {3-540-55305-3},
   MRCLASS = {17B37 (17B65)},
  MRNUMBER = {1183482},
MRREVIEWER = {Janusz\ Grabowski},
       DOI = {10.1007/BFb0101183},
       URL = {https://doi.org/10.1007/BFb0101183},
}

@article {Mac02,
    AUTHOR = {MacKay, N.},
     TITLE = {Rational {$K$}-matrices and representations of twisted
              {Y}angians},
   JOURNAL = {J. Phys. A},
  FJOURNAL = {Journal of Physics. A. Mathematical and General},
    VOLUME = {35},
      YEAR = {2002},
    NUMBER = {37},
     PAGES = {7865--7876},
      ISSN = {0305-4470,1751-8121},
   MRCLASS = {81R12 (17B37 81R50 81T40)},
  MRNUMBER = {1945798},
MRREVIEWER = {Sonia\ Natale},
       DOI = {10.1088/0305-4470/35/37/302},
       URL = {https://doi.org/10.1088/0305-4470/35/37/302},
}

@article {Lu26min,
    AUTHOR = {Lu, K.},
     TITLE = {Minimalistic presentation and coideal structure of twisted
              {Y}angians},
   JOURNAL = {Comm. Math. Phys.},
  FJOURNAL = {Communications in Mathematical Physics},
    VOLUME = {407},
      YEAR = {2026},
    NUMBER = {5},
     PAGES = {Paper No. 99, 38},
      ISSN = {0010-3616,1432-0916},
   MRCLASS = {17B37},
  MRNUMBER = {5055749},
       DOI = {10.1007/s00220-026-05615-3},
       URL = {https://doi.org/10.1007/s00220-026-05615-3},
}

@article {GRW19,
    AUTHOR = {Guay, N. and Regelskis, V. and Wendlandt, C.},
     TITLE = {Equivalences between three presentations of orthogonal and
              symplectic {Y}angians},
   JOURNAL = {Lett. Math. Phys.},
  FJOURNAL = {Letters in Mathematical Physics},
    VOLUME = {109},
      YEAR = {2019},
    NUMBER = {2},
     PAGES = {327--379},
      ISSN = {0377-9017,1573-0530},
   MRCLASS = {17B37 (16T20)},
  MRNUMBER = {3917347},
MRREVIEWER = {Chengming\ Bai},
       DOI = {10.1007/s11005-018-1108-6},
       URL = {https://doi.org/10.1007/s11005-018-1108-6},
}

@book {Hum72,
    AUTHOR = {Humphreys, J. E.},
     TITLE = {Introduction to {L}ie algebras and representation theory},
    SERIES = {Graduate Texts in Mathematics},
    VOLUME = {Vol. 9},
 PUBLISHER = {Springer-Verlag, New York-Berlin},
      YEAR = {1972},
     PAGES = {xii+169},
   MRCLASS = {17BXX},
  MRNUMBER = {323842},
MRREVIEWER = {F.\ W.\ Lemire},
}

@article {BR17,
    AUTHOR = {Belliard, S. and Regelskis, V.},
     TITLE = {Drinfeld {J} presentation of twisted {Y}angians},
   JOURNAL = {SIGMA Symmetry Integrability Geom. Methods Appl.},
  FJOURNAL = {SIGMA. Symmetry, Integrability and Geometry. Methods and
              Applications},
    VOLUME = {13},
      YEAR = {2017},
     PAGES = {Paper No. 011, 35},
      ISSN = {1815-0659}
}

@article {Dri87,
    AUTHOR = {Drinfeld, V. G.},
     TITLE = {A new realization of {Y}angians and of quantum affine
              algebras},
   JOURNAL = {Dokl. Akad. Nauk SSSR},
  FJOURNAL = {Doklady Akademii Nauk SSSR},
    VOLUME = {296},
      YEAR = {1987},
    NUMBER = {1},
     PAGES = {13--17},
      ISSN = {0002-3264},
   MRCLASS = {17B65 (16A24 17B45 58F07 81E99)},
  MRNUMBER = {914215},
MRREVIEWER = {J.\ S.\ Joel},
}

@article {Dri85,
    AUTHOR = {Drinfeld, V. G.},
     TITLE = {Hopf algebras and the quantum {Y}ang-{B}axter equation},
   JOURNAL = {Dokl. Akad. Nauk SSSR},
  FJOURNAL = {Doklady Akademii Nauk SSSR},
    VOLUME = {283},
      YEAR = {1985},
    NUMBER = {5},
     PAGES = {1060--1064},
      ISSN = {0002-3264},
   MRCLASS = {58F07 (17B20 82A15)},
  MRNUMBER = {802128},
MRREVIEWER = {Alexander\ A.\ Pankov},
}

@article {LWZ25affine,
    AUTHOR = {Lu, K. and Wang, W. and Zhang, W.},
     TITLE = {Affine {$\imath$}quantum groups and twisted {Y}angians in
              {D}rinfeld presentations},
   JOURNAL = {Comm. Math. Phys.},
  FJOURNAL = {Communications in Mathematical Physics},
    VOLUME = {406},
      YEAR = {2025},
    NUMBER = {5},
     PAGES = {Paper No. 98, 36},
      ISSN = {0010-3616,1432-0916},
   MRCLASS = {17B37 (81R50)},
  MRNUMBER = {4887617},
       DOI = {10.1007/s00220-025-05263-z},
       URL = {https://doi.org/10.1007/s00220-025-05263-z},
}

@article {GNW18,
    AUTHOR = {Guay, N. and Nakajima, H. and Wendlandt, C.},
     TITLE = {Coproduct for {Y}angians of affine {K}ac-{M}oody algebras},
   JOURNAL = {Adv. Math.},
  FJOURNAL = {Advances in Mathematics},
    VOLUME = {338},
      YEAR = {2018},
     PAGES = {865--911},
      ISSN = {0001-8708,1090-2082},
   MRCLASS = {17B37 (17B67)},
  MRNUMBER = {3861718},
MRREVIEWER = {Volodymyr\ Mazorchuk},
       DOI = {10.1016/j.aim.2018.09.013},
       URL = {https://doi.org/10.1016/j.aim.2018.09.013},
}

@article {Kn95,
    AUTHOR = {Knight, H.},
     TITLE = {Spectra of tensor products of finite-dimensional
              representations of {Y}angians},
   JOURNAL = {J. Algebra},
  FJOURNAL = {Journal of Algebra},
    VOLUME = {174},
      YEAR = {1995},
    NUMBER = {1},
     PAGES = {187--196},
      ISSN = {0021-8693,1090-266X},
   MRCLASS = {17B37 (81R50)},
  MRNUMBER = {1332866},
MRREVIEWER = {Preeti\ Parashar},
       DOI = {10.1006/jabr.1995.1123},
       URL = {https://doi.org/10.1006/jabr.1995.1123},
}

\end{document}